\documentclass[11pt,twoside]{amsart}

\usepackage{amsmath}
\usepackage{xspace}
\usepackage{amssymb}
\usepackage{bbm}
\usepackage{graphicx}
\usepackage{url}
\usepackage{pifont}
\usepackage{enumerate}
\usepackage{tikz}
\usepackage{lipsum}
\usepackage[matrix,arrow]{xy}

\usepackage{tabularx}
\usepackage{hyperref}
\numberwithin{equation}{section}

\newcommand{\mi}{\bbi\xspace}

\DeclareMathSymbol{\varnothing}{\mathord}{AMSb}{"3F}
\renewcommand{\emptyset}{\varnothing}
\DeclareMathOperator{\RE}{Re}
\DeclareMathOperator{\IM}{Im}

\DeclareMathOperator{\graph}{graph}

\newcommand{\bbC}{\mathbb{C}}

\newcommand{\bbN}{\mathbb{N}}

\newcommand{\bbP}{\mathbb{P}}
\newcommand{\bbQ}{\mathbb{Q}}
\newcommand{\bbR}{\mathbb{R}}
\newcommand{\bbS}{\mathbb{S}}

\newcommand{\bbZ}{\mathbb{Z}}

\newcommand{\N}{\mathbb{N}}
\newcommand{\R}{\mathbb{R}}
\newcommand{\C}{\mathbb{C}}

\newcommand{\calA}{\mathcal{A}}
\newcommand{\calB}{\mathcal{B}}

\newcommand{\calE}{\mathcal{E}}
\newcommand{\calF}{\mathcal{F}}

\newcommand{\calH}{\mathcal{H}}

\newcommand{\calS}{\mathcal{S}}

\newcommand{\calW}{\mathcal{W}}

\newcommand{\CPone}{\bbP^1}

\newcommand{\ax}{a^\circ}

\newcommand{\aax}{\alpha^\circ}

\newcommand{\nx}{\nu^\circ}

\newcommand{\ind}{l}
\newcommand{\indd}{k}
\theoremstyle{plain}
\newtheorem{theorem}{Theorem}[section]
\newtheorem*{theorem*}{Theorem}
\newtheorem{corollary}[theorem]{Corollary}
\newtheorem*{corollary*}{Corollary}
\newtheorem{proposition}[theorem]{Proposition}
\newtheorem*{proposition*}{Proposition}
\newtheorem{lemma}[theorem]{Lemma}
\newtheorem*{lemma*}{Lemma}

\newtheorem*{example*}{Example}
\newtheorem{definition}[theorem]{Definition}
\newtheorem*{definition*}{Definition}

\newtheorem*{notation*}{Notation}
\newtheorem{remark}[theorem]{Remark}
\newtheorem*{remark*}{Remark}

\newcommand{\bbi}{\mathbbm{i}}

\title[The space of genus two spectral curves of cmc tori]{The space of genus two spectral curves \\ of constant mean curvature tori in $\bbR^3$}

\author[E.\ Carberry]{Emma Carberry}
\address{E.\ Carberry, School of Mathematics and Statistics,\,University of Sydney,\,Australia.}
\email{emma.carberry@sydney.edu.au}

\author[M.\ Kilian]{Martin Kilian}
\address{M. Kilian, School of Mathematical Sciences,
University College Cork, Ireland.}
\email{m.kilian@ucc.ie}

\author[S.\ Klein]{Sebastian Klein}
\address{S.\ Klein, Mathematics Chair III\\
Universit\"at Mannheim\\
D-68131 Mannheim, Germany.}
\email{s.klein@math.uni-mannheim.de}

\author[M.\ U.\ Schmidt]{Martin Ulrich Schmidt}
\address{M.\ Schmidt, Mathematics Chair III,\,
Universit\"at Mannheim,\,D-68131 Mannheim,\,Germany.}
\email{schmidt@math.uni-mannheim.de}

\thanks{\,\,\, S. Klein was funded by DFG Grant 414903103.}

\begin{document}
\begin{abstract} 
We use  Whitham deformations to give a complete account of spectral data of real solutions of the sinh--Gordon equation of spectral genus $2$. We parameterise the closure of spectral data of constant mean curvature tori in $\bbR^3$ by an isosceles right triangle and analyse its boundary. We prove that the Wente family, which is described by spectral data with real coefficients, is parameterised by the bisector of the right angle. Our methods combine blowups of Whitham deformations and spectral data in an innovative way that changes the underlying integrable system.
\end{abstract}
\dedicatory{Dedicated to Nicholas Schmitt}

\maketitle
%
%%%%%%%%%%%%%%%%%%%%%%%%%%%%%%%%%%%%%%%%%%%%%%%%%%%%%%%%%%%%%%%%
%

%
%%%%%%%%%%%%%%%%%%%%%%%%%%%%%%%%%%%%%%%%%%%%%%%%%%%%%%%%%%%%%%%%
%
\section{Introduction and Definitions}
\subsection{Introduction}
Wente's discovery of constant mean curvature (cmc) tori \cite{Wen} in 1986 led to a series of articles on the subject in the 1980s and 1990s. In an attempt to produce graphics of the new examples, Abresch \cite{Ab} noticed that Wente's examples are foliated by two families of curvature lines: one planar family, and one spherical family. This prompted Abresch to single out such cmc tori, resulting in a reduction of the sinh-Gordon equation to a nonlinear ode and solutions in terms of elliptic functions. These preliminary results culminated in the work of Pinkall-Sterling \cite{PinS}, Hitchin \cite{hit:tor} and Bobenko \cite{Bob:cmc} with a systematic description of all cmc tori in terms of algebro-geometric data, namely spectral curves with holomorphic line bundles. The spectral curves are hyperelliptic Riemann surfaces, and for cmc tori in $\bbR^3$ the lowest possible genus of such spectral curves is $g=2$. The countably infinite examples studied by Wente and Abresch are contained in this simplest case of $g=2$, but there are many more cmc tori in this simplest class that are not discussed in the work of Abresch. This paper gives a full account of the space of all cmc tori with $g=2$.

The space of real solutions of the sinh--Gordon equation, or equivalently of cmc planes in $\bbR^3$, is itself quite simple when expressed in terms of spectral curves. However the spectral curves of cmc tori are characterised by transcendental conditions. They form a totally disconnected set and so we first consider the closure of spectral curves of cmc tori within the space of spectral curves of cmc planes. For $g=2$ this closure is the union of an $\bbS^1$-family of diffeomorphic spaces $\calS^2_{\lambda_0},\,\lambda_0\in\bbS^1$. In our main theorem, Theorem~\ref{Th:main}, we construct a diffeomorphism from $\calS^2_1$ to an open triangle in the real plane. The spectral curves of tori with $g=2$ and $\lambda_0=1$ correspond to the points in the intersection of the interior of the isosceles right triangle in Figure~\ref{fig:triangle} with $(\pi\bbQ)^2$. 
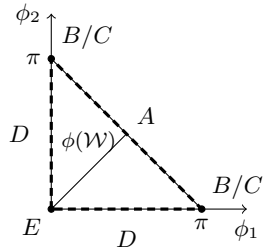
\begin{figure}[h]
\centering
\begin{tikzpicture}[scale = 2.0]
    \draw[->] (0,0) -- (1.3,0) node[at end, below] {\footnotesize$\phi_1$};
    \draw[->] (0,0) -- (0,1.3) node[at end, left] {\footnotesize$\phi_2$};

    \draw[line width = 1.2pt, densely dashed] (0,0) -- (0,1) -- (1,0) -- (0,0); 

    \draw (1,0) -- (0,1) node[midway, above right] {\footnotesize$A$};

    \draw (0,0) -- (0.5,0.5) node[midway, above, inner sep = 7pt] {\scriptsize$\phi(\calW)$};
    
    \node[left] at (0,1) {\footnotesize$\pi$};
    \node[below] at (1,0) {\footnotesize$\pi$};

    \draw[fill] (0,0) circle (0.6pt) node[below left] {\footnotesize$E$};
    \draw[fill] (1,0) circle (0.6pt) node[above right] {\footnotesize$B/C$};
    \draw[fill] (0,1) circle (0.6pt) node[above right] {\footnotesize$B/C$};
    
    \node[left, inner sep = 8pt] at (0,0.5) {\footnotesize$D$};
    \node[below, inner sep = 8pt] at (0.5,0) {\footnotesize$D$};

\end{tikzpicture}
\caption{The closure of the space of spectral curves of cmc tori of spectral genus two. Labelling and details are explained in Theorem 1.3 and Section 7. } \label{fig:triangle}
\end{figure}

The Whitham flow (isoperiodic flow \cite{GriS1}) has been an effective tool in studying moduli spaces of integrable equations.  Within integrable surface geometry (see also \cite{GerPS,H2021}), it has for example been used to classify properly embedded minimal annuli in $\bbS^2\times\bbR$ \cite{HKS2}, equivariant cmc tori in $\bbS^3$ \cite{KSS}, Alexandrov-embedded cmc tori in $\bbS^3$ \cite{HKS3} and harmonic tori in $\bbS^3$  \cite{CO:19, CO:20}, as well as to determine the closure of cmc tori in $\bbS^3$ and $\bbR^3$ within the spaces of cmc planes of finite type  \cite {CS1, CS}. Here we use Whitham flows to parameterise all cmc tori in $\bbR^3$ whose spectral curves have genus two. We study the boundary of the space of these cmc tori, which corresponds to the boundary of the isosceles right triangle in Figure~\ref{fig:triangle}. We will see in Section~\ref{Se:limits a to e} that the limit to this boundary corresponds to a transition from the sinh-Gordon equation to another limiting integrable system, and this limiting system depends on the corresponding point on the boundary of the triangle. See the table in Section~\ref{Se:limits a to e} for the five different cases that occur, along with their limit integrable systems.

In some cases we apply a blowup to spaces of polynomials whose elements describe spectral data, that is spectral curves and differentials on them. This blowup is done in such a way that the  spectral curves are simultaneously blown up at particular points. By blowing up the polynomials, we investigate either subvarieties of the space of spectral data and their singularities or Whitham vector fields along such subvarieties. The blowup of the corresponding spectral curves allows us to view the underlying isospectral flows and the blown up vector fields again as isospectral flows and Whitham vector fields of another integrable system. In this sense these blowups give maps that switch from one integrable system to another one. Section~\ref{Se:blow up sym point} introduces this method before it is applied in Sections~\ref{Se:wente} and~\ref{Se:limits a to e} in many different situations, see also \cite{CKS}.
\subsection{The definition of $\calH^g$} The spectral curves corresponding to finite type real solutions of the sinh--Gordon equation with spectral genus $g$ are hyperelliptic complex curves $\Sigma = \Sigma_a$ described by the equation
$\nu^2 = \lambda\,a(\lambda)$, where $a$ is a polynomial of degree $2g$. By compactification, we regard $\Sigma$ as a compact, hyperelliptic
surface with a holomorphic map $\lambda $ to the Riemann sphere $\CPone$, and denote the part of $\Sigma$ that is above $\C^\times \ni \lambda$ by $\Sigma^\circ = \{(\lambda,\nu) \in \C^\times \times \C \mid
\nu^2 = \lambda\,a(\lambda)\}$. By rescaling $\lambda$ and $\nu$, we may assume that the polynomial $a$ satisfies the following conditions:
\begin{enumerate}
\item
  \emph{Reality condition:} $a\in P^{2g}_\bbR$, where for $d\in \mathbb{N}$ we define $P^d_\bbR$ as the space of polynomials $p$ of degree at most $d$ which satisfy the reality condition
  $$ \lambda^{d}\, \overline{p(\bar{\lambda}^{-1})} = p(\lambda)  $$
\item
  \emph{Positivity condition:} $\lambda^{-g}\, a(\lambda) > 0$ for $\lambda\in \bbS^1$
\item
  \emph{Non--degeneracy:} The roots of $a$ are all pairwise distinct
\item
  \emph{Normalisation:} The highest coefficient of $a$ has absolute value $1$.
\end{enumerate}
The condition that the roots of $a$ are pairwise distinct ensures that the corresponding spectral curve is smooth. 
We denote the space of polynomials $a$ which satisfy these conditions by $\mathcal{H}^g$. We regard $\mathcal{H}^g$ and its subsets as topological subspaces of the space $\C^{2g}[\lambda]$ of complex polynomials in $\lambda$ of degree at most $2g$.
In the following, all topological closures of such sets are taken in $\C^{2g}[\lambda]$.

Let $a\in \mathcal{H}^g$ and $\Sigma$ be the corresponding spectral curve, that is the hyperelliptic complex curve defined by the equation $\nu^2 = \lambda\,a(\lambda)$. It comes equipped
with the holomorphic involution
\begin{equation}
  \label{eq:sigma}
  \sigma: \Sigma \to \Sigma,\; (\lambda,\nu) \mapsto (\lambda,-\nu) \,.
\end{equation}
Moreover, due to the reality condition for $a$, it also comes with the anti--holomorphic involution
\begin{equation}
  \label{eq:rho}
  \rho: \Sigma\to\Sigma,\; (\lambda,\nu) \mapsto (\bar{\lambda}^{-1}, \bar{\lambda}^{-(g+1)} \,\bar{\nu}) \; .
\end{equation}
Note that all points $(\lambda,\nu) \in \Sigma$ with $\lambda\in \bbS^1$ are fixed points of $\rho$. 
\subsection{Differentials and $\calB_a$}\label{subsection differentials}
The differential forms \,$\Theta$\, on \,$\Sigma$\, we are interested in are abelian differentials of the second kind with second order poles at \,$\lambda=0$\, and \,$\lambda=\infty$\, and no others. We further require the symmetry condition 
$$ \sigma^* \Theta = -\Theta $$
and the reality condition 
\begin {equation} 
\label {eq:Thetareal}
\rho^* \Theta = -\overline{\Theta} \; . 
\end {equation}
These are exactly the differentials of the form (here we write $\Theta(b)$ instead of $\Theta_b$ as in \cite{CS1,CS})
$$ 
\Theta(b) = \frac{b(\lambda)}{\nu}\,\frac{\mathrm{d}\lambda}{\lambda} 
\quad\text{with \,$b\in P_\bbR^{g+1}$\,.} 
$$
We let $\calB_a$ be the $\R$--linear space of polynomials $b\in P^{g+1}_{\R}$, such that the corresponding differential form $\Theta(b)$ has purely imaginary periods. This space $\calB_a$ has real dimension $2$. Moreover, any $b\in \calB_a$ is uniquely determined by the value of $b(0)$, in other words, the $\R$--linear map $\calB_a \to \C,\; b \mapsto b(0)$ is an isomorphism of $\R$--linear spaces. It follows that there exists a unique basis $(b_1,b_2)$ of $\calB_a$ with $b_1(0)=1$ and $b_2(0)=\mi$. We call this basis the \emph{normalised basis} of $\calB_a$.  

The importance of the differentials $\Theta(b)$ where $b\in\calB_a$ comes from the fact that they encode the double periodicity of translational flow on the isospectral set of the corresponding spectral curve $\Sigma$. Thus any solution of the sinh--Gordon equation of a constant mean curvature torus in $\R^3$ gives rise to a spectral curve $\Sigma_a$ which satisfies periodicity conditions. These conditions state the existence of a \emph{Sym point} $\lambda_0\in \bbS^1$, a basis $(b_1,b_2)$ of $\calB_a$ and holomorphic functions $\mu_1,\mu_2$ on $\Sigma^\circ$ such that 
\begin{enumerate}
\item For $\ind=1,2$, the function $\mu_\ind$ is the logarithmic primitive of $\Theta(b_\ind)$, so $\mathrm{d} \ln\mu_\ind=\Theta(b_\ind)$, with an anti--symmetric branch with respect to $\sigma$ on a neighbourhood of $\lambda=0$.
\item We have $b_1(\lambda_0)=b_2(\lambda_0)=0$ and $\mu_1(\lambda_0) = \mu_2(\lambda_0) = \pm 1$.
\end{enumerate}
The real planes $ \calB_a $ define a subbundle of the trivial bundle $\calH^2\times P_{\mathbb R}^3$.
\subsection{The space $\calS^2_1$} We now define for $\lambda_0\in \bbS^1$
\begin{align*}
\calS^g_{\lambda_0}&:= \{ a \in \mathcal{H}^g \mid b(\lambda_0)=0 \text{ for all $b\in \calB_a$ } \}&\text{and}&&\calS^g&:=\bigcup_{\lambda_0\in \bbS^1} \calS^g_{\lambda_0} \;.
\end{align*}
By the above characterisation, it is clear that the set $\mathcal{P}^g$ of polynomials $a \in \mathcal{H}^g$ which correspond to constant mean curvature tori in $\bbR^3$ is contained in $\calS^g$. Indeed, due to the fact that the unit circle is closed in $\bbC$, the closure of
$\mathcal{P}^g$ in $\mathcal{H}^g$ is also contained in $\calS^g$. Rotation by $\lambda_0 $ induces a diffeomorphism from $\calS^g_1 $ onto $\calS^g_{\lambda_0} $ and so it will suffice to consider $\calS^g_1 $. In this paper we investigate the case $g=2$ and define on $\calS^2_1$ the frame bundle
\begin{align}\label{def:frame}
\calF:=\{(a,b_1,b_2)\in\calS^2_1\times P_\bbR^3\times P_\bbR^3\mid b_1\text{ and }b_2\text{ form a basis of }\calB_a\}
\end{align}
with the projection $\pi:\calF\to\calS^2_1,\;(a,b_1,b_2)\mapsto a$. In \cite{CS} it was shown that that the closure of $\mathcal{P}^2$ in $\mathcal{H}^2$ is equal to $\calS^2$. Indeed this follows from \cite[Thm~5.8]{CS} together with the following statement:
\begin{lemma} \label{L:S2-smooth}
$\calS^2$ and $\calS^2_{\lambda_0} $ are smooth submanifolds of $\mathcal{H}^2$ of dimension $3$ and $2$ respectively. 
\end{lemma}
\begin{proof}
Due to \cite[Thm~3.2]{CS}, $\deg(\gcd(\calB_a))\leq 1$ for all $a\in \mathcal{H}^2$ with equality for $a \in \calS^2$. Thus by \cite[Thm~5.5(i),(ii)]{CS}, $\calS^2_{\lambda_0} $ and $\calS^2$ are smooth submanifolds of $\mathcal{H}^2$ of dimension $2$ and $3$ respectively.
\end{proof}
\subsection{The Main Result}
Our main result is a global parameterisation of $\calS^2_1$ using Whitham flows. To say that a tangent vector in $T_{(a, b_1, b_2)}\calF$ or a vector field on \,$\calF$\, is Whitham means that it infinitesimally preserves the periods of the differentials $\Theta(b_1),\Theta(b_2) $. A Whitham flow on \,$\calF$\, is the flow of a Whitham vector field. An important step of the proof shows that the composition of the Whitham flow with a suitable map $\phi:\mathcal F\rightarrow\mathbb R^2$ is a diffeomorphism onto an open triangle. To define $\phi $, we note first that the reality condition  \eqref {eq:Thetareal} satisfied by the differentials $\Theta(b) $ ensures that they are purely imaginary on the fixed point set $\Sigma_\bbR$ of the anti-holomorphic involution $\rho$. In fact their restriction to $\Sigma_\bbR$ is exact and we may take the unique primitive $q$ which obeys $\sigma^\ast q=-q$.
\begin{lemma}\label{L:theta}
For any $a\in\mathcal{H}^2$ and $b\in\calB_a$ there exists a unique function $q:\Sigma_\bbR\to \mi\bbR$ on the fixed point set $\Sigma_\bbR$ of $\rho$ with $dq=\Theta(b)|_{\Sigma_\bbR}$ and $\sigma^*q=-q$.
\end{lemma}
\begin{proof}
The fixed point set of $\rho$ is the preimage of $\bbS^1$ with respect to the map $\Sigma\to\mathbb{P}^1, (\lambda,\nu)\mapsto\lambda$. This fixed point set is  invariant under $\sigma$ and for even genus homeomorphic to $\bbS^1$. Due to the anti--symmetry of $\Theta(b)$ the integral of $\Theta(b)$ along $\Sigma_\bbR$ vanishes for all $b\in\calB_a$. For any $y\in\Sigma_\bbR$ this fixed point set decomposes into two disjoint paths from $\sigma(y)$ to $y$ and the integral of $\Theta(b)$ along both paths coincide. The condition $\sigma^*q=-q$ implies $q(y)=\tfrac{1}{2}(q(y)-q(\sigma(y)))=\frac{1}{2}\int_{\sigma(y)}^y\Theta(b)$. This defines the unique smooth function $q:\Sigma_\bbR\to \mi\bbR$ with the desired properties.
\end{proof}
The Sym point \,$\lambda=1$\, has two pre-images in  \,$\Sigma_a$\,; we write \,$y(a)$\, for the one with \,$\nu > 0$\,. By slight abuse of terminology we call also \,$y(a)$\, the Sym point of \,$\Sigma_a$\,. Then we may define the map $\phi$:
\begin{align}\label{def:phi}
\phi:\calF&\to\bbR^2,& 
(a,b_1,b_2) &\mapsto -\mi(q_1(y(a)),q_2(y(a))).
\end{align}
We shall normalise the Whitham flows in such a way that the Jacobi matrix of the composition of the flow with this $\phi$ is equal $\big(\begin{smallmatrix}1&0\\0&1\end{smallmatrix}\big)$. Hence this composition is just a translation by the value of $\phi$ at the initial point and the Whitham flow is commutative. Now we state our main result.
\begin{theorem}\label{Th:main}
There is a global smooth section $b:\calS^2_1\rightarrow\calF$ such that the composition $\phi\circ b = (\phi_1,\,\phi_2)$ is a diffeomorphism onto the open triangle $\triangle=\{\phi\in(0,\pi)^2\mid\phi_1+\phi_2<\pi\}$. Furthermore, $\Theta(b_\ind(a))$ is for $\ind=1,2$ the logarithmic derivative of a unique function $\mu_\ind$ on $\Sigma_a$ whose logarithm $\ln\mu_\ind$ has in a neighbourhood of $\lambda=0$ an anti--symmetric branch with respect to $\sigma$~\eqref{eq:sigma}. 
\end{theorem}
\begin{remark}
In \cite{MaOs}, simply periodic real solutions of the KdV equation are considered. In this situation one has a single differential which is the logarithmic derivative of a multiplier $\mu$. They prove that the sequence of values $h_k=|\ln\mu(c_k)|$ at the sequence of roots $(c_k)_{k\in\bbZ}$ of the differential are global parameters on the moduli space. We instead have a pair \,$(\Theta(b_1), \Theta(b_2))$\, of such differentials. Motivated by~\cite{MaOs} we shall parameterise $\calS^2_1$ by the values of the primitives $(q_1,q_2)$ at the Sym point $\lambda_0=1$. This parametrisation immediately identifies the cmc tori: the subspace of all $a\in\calS^2_1$ representing a spectral curve of a cmc torus in $\bbR^3$ with spectral genus 2 and Sym point at $\lambda=1$ is equal to $(\phi\circ b)^{-1}\big[\triangle\cap\big(\pi\bbQ\big)^2\big]$. Rotation by $\lambda_0\in\bbS^1$ maps $\calS_1^2$ onto $\calS_{\lambda_0}^2$, so an analogous statement holds for $\calS_{\lambda_0}^2$.
\end{remark}
\subsection{Outline of proof}
The proof of Theorem~\ref{Th:main} is contained in Section~\ref{Se:proof} and combines the results of all foregoing sections. We briefly outline the structure of the proof in logical order, before we describe the content of each section. Clearly the statement of the theorem can only be true for a special choice of the section $b$. For any $a\in\calS^2_1$ a basis $(b_1,b_2)$ of $\calB_a$ is uniquely characterised by the period map
\begin{align*}
H_1(\Sigma_a,\bbZ)&\to\bbR^2,&C&\mapsto\left(\int_C\Theta(b_1),\int_C\Theta(b_2)\right).
\end{align*}
Indeed, if any two $\Theta(b) $ have the same periods then an appropriate linear combination of them is a holomorphic differential with vanishing periods and hence the zero differential. On simply connected subsets of $\calS^2_1$ the bundle with the discrete fibres $H_1(\Sigma_a,\bbZ)$ has a natural trivialisation. It follows that on any simply-connected open neighbourhood of \,$a_0 \in \calS^2_1$\,, any basis $(b_{1,0},b_{2,0})$ of $\calB_{a_0}$ extends to a unique section of $\calF$ whose period map is constant with respect to this trivialisation. The global section $b$ whose existence is stated in the theorem is of this type. In the spectral genus two case all solutions of $\sinh$--Gordon are doubly--periodic. Consequently we have the additional structure of an apriori given $\mathrm{GL}(2,\bbZ)$--subbundle of triples $(a,b_1,b_2)\in\calF$, such that $(b_1,b_2)$ generates the lattice of all $b\in\calB_a$ whose differentials $\Theta(b)$ are logarithmic derivatives of holomorphic functions $\mu$ on $\Sigma^\circ$. This subbundle is preserved by the corresponding Whitham flow. Furthermore, the discreteness of the fibres of the subbundle leads to local smooth sections $b$ of $\calF$ which integrate the Whitham vector fields and shows that the Whitham flow defines a local action of the commutative Lie--group $\bbR^2$ on $\calF$. Here we use a special basis $(b_1,b_2)$ of the two--dimensional vector space $\calB_a$. Then pulling back the coordinate vector fields of $\phi$ we obtain two commuting Whitham vector fields.

We extend the combined flow $(t_1,t_2)=t\mapsto (a_t,b_{1,t},b_{2,t})$ of these two commuting Whitham vector fields to a maximal domain $\Omega\subset\bbR^2$. The boundary points of \,$\Omega$\, do not correspond to elements of \,$\calS^2_1$\,. We investigate the behaviour of the map $\phi$ at these boundary points. For this purpose we characterise the aforementioned period map in two steps. First it is specified up to finitely many choices which are permuted by the dihedral group $D_4$. In a second step we show that there always exists a choice such that both components of the following map are positive:
\begin{align}\label{def:image phi}
\Omega&\to\bbR^2,&t\mapsto\phi(a_t,b_{1,t},b_{2,t})\,.
\end{align}
Furthermore, at the boundary of $\Omega$ the flow has to leave $\calS^2_1$. However, since the spectral curves are determined by branchpoints of two--sheeted coverings over the compact space $\bbP^1$, any sequence of spectral curves has a convergent subsequence. By this compactness we define limits of the spectral curves, and perform the aforementioned combination of blowups of $\calF$ and of the corresponding spectral curves. Along these lines we extend the map $\phi$ to these limits. More precisely, we show that the image of $\partial\Omega$~\eqref{def:image phi} is equal to $\partial\triangle$ and $\Omega$ is equal to $\triangle-\phi(a_0,b_{1,0},b_{2,0})$.

In the final step we show that $t\mapsto a_t$ is bijective from $\Omega$ onto $\calS^2_1$. As a consequence the composition of the inverse of this map with $t\mapsto(a_t,b_{1,t},b_{2,t})$ defines a global section $b$ of $\calF$ with the properties specified in Theorem~\ref{Th:main}. The proof will be an outcome of the investigation of $\calW=\calS^2_1\cap\bbR^4[\lambda]$ which we call the Wente family. More precisely, by extending the Whitham flow along the diagonal
\begin{align}\label{def:diagonal}
\big\{(\varphi,\varphi)-\phi(a_0,b_{1,0},b_{2,0})\mid\varphi\in(0,\tfrac{\pi}{2})\big\}\subset\Omega
\end{align}
to $t_\infty=(\frac{\pi}{2},\frac{\pi}{2})-\phi(a_0,b_{1,0},b_{2,0})\in\partial\Omega$ with $a_{t_\infty}=(\lambda-1)^4$, the restriction of $t\mapsto a_t$ to~\eqref{def:diagonal} is shown to be a diffeomorphism onto $\calW$. Since this is true for any $a_0\in\calS^2_1$, the image of $t\mapsto a_t$ is $\calS^2_1$. Morever, by definition of $\phi$, any $t\ne t'\in\Omega$ with the same image in $\calS^2_1$ are mapped by~\eqref{def:image phi} onto elements of $\triangle$ with interchanged components. Now because $t\mapsto a_t$ is locally bijective, it follows that it is injective at first on a tubular neighbourhood of~\eqref{def:diagonal} and then on all of $\Omega$.
\subsection{Summary of sections} 
In Section~\ref{Se:local whitham} we construct the Whitham vector fields which are mapped by~\eqref{def:phi} onto the vector fields $\frac{\partial}{\partial\phi_1}$ and $\frac{\partial}{\partial\phi_2}$. Namely we show that the Whitham tangent vectors \,$(\dot{a}, \dot{b}_1, \dot{b}_2)$\, at a point \,$(a,b_1,b_2) \in \calF$\, are parameterised by the values of \,$\dot{q}_1$\, and \,$\dot{q}_2$\, at the Sym point. 

In Section~\ref{Se:closure} we determine in Proposition~\ref{P:boundary-S21-C4lambda} for any $\lambda_0\in\bbS^1$ the boundary of $\calS_{\lambda_0}^2$ as a subset of $P_\bbR^4$. It will turn out that this boundary decomposes into two one--dimensional smooth families of elliptic curves, each with an ordinary double point at $\lambda_0$, which are separated by the rational curve with a higher order double point at $\lambda_0$. The rotation by $\lambda_0$ induces a diffeomorphism from $\calS_1^2$ onto $\calS_{\lambda_0}^2$, so we concentrate on $\calS_1^2$. At the central element $a(\lambda)=(\lambda-1)^4$ of $\partial\calS^2_1$ all four roots of $a$ coalesce at the common root of $(b_1,b_2)$. Hence a blowup of $\calS^2_1$ at this point might separate all of them. This is carried out in Section~\ref {Se:blow up sym point} by rescaling the new parameter $\varkappa=\frac{\lambda-1}{\mi(\lambda+1)}$. The elements of these blowups describe blowups of the spectral curve at $\varkappa=0$. The rescaled marked points at $\varkappa=\pm\mi$ coalesce at two unbranched points over $\varkappa=\infty$. So the blowup transforms the integrable system of the $\sinh$--Gordon equation into the nonlinear Schr\"odinger (NLS) equation. Consequently in Theorem~\ref{th:blow-up} the exceptional fibre is identified with a natural analogue of $\calS^2_1$ inside the elliptic spectral curves of NLS.

In Section~\ref{Se:wente} we consider the Wente family $\mathcal{W}:=\calS_1^2\cap\bbR[\lambda]$. The two--dimensional space $\calB_a$ has for such spectral curves $a\in\mathcal{W}$ a another natural basis $b_1$ and $b_2$. In Theorem~\ref{T:wente:wente} we give a rather explicit description of this family. This Theorem is used in Section~\ref{Se:proof} to prove that the restriction of the map $t\mapsto a_t$ to~\eqref{def:diagonal} is a diffeomorphism onto $\calW$.

The explicit description of the Wente family has the useful consequence that the roots of any $a\in\calS_1^2$ are never colinear. Consequently we can define globally on the universal covering of $\calS_1^2$ the two cycles of the spectral curve, which project in the $\lambda$--plane to straight lines connecting the two roots $\alpha_1,\,\alpha_2 \in B(0,1)\setminus\{0\}$ with their reflected roots $\Bar{\alpha}_1^{-1}$ and $\Bar{\alpha}_2^{-1}$, respectively. Moreover, in Theorem~\eqref{Th:global basis}~(i) we show that the orientation of these cycles is uniquely determined by the condition that both components of $\phi$~\eqref{def:phi} are positive. This allows us to define in Section~\ref{Se:global Whitham} on every simply connected subset of $\calS^2_1$ a section $b$ of $\calF$ such that each $\Theta(b_\ind(a))$ is the logarithmic derivative of a global function. The Whitham flow preserves this section. It is extended to a maximal domain $\Omega$ which is defined as the union of the intervals of the one--dimensional maximal flows along all directions of the two--dimensional commutative Whitham flow. For any sequence $(t_n)_{n\in\mathbb{N}}$ in $\Omega$ we describe the corresponding sequence $(a_n)$ in $\calS^2_1$ by its roots  $(\alpha_{1,n},\alpha_{2,n})_{n\in\mathbb{N}}$ in $B(0,1)\setminus\{0\}$. These sequences of roots are bounded and have convergent subsequences whose limits are denoted by $(\alpha_1,\alpha_2)$. We observe that if any accumulation point of $(t_n)_{n\in\mathbb{N}}$ is an endpoint of one of the maximal intervals in $\Omega$, then at least one limit $\alpha_1$ and $\alpha_2$ does not belong to $B(0,1)\setminus\{0\}$. In this case Theorem~\ref{Th:global basis} gives a list of 5 possible sets, to which the limits $(\alpha_1,\alpha_2)$ might belong:
\begin{enumerate}
\item[(A)] $\{(1,1)\}$.%\label{case one}
\item[(B)] $\bigr(\big(B(0,1)\setminus\{0\}\big)\times\{1\}\bigr) \;\cup\;\bigr(\{1\}\times\big(B(0,1)\setminus\{0\}\big)\bigr)$.%\label{case two}
\item[(C)] $\bigr(\{0\}\times\partial B(0,1)\bigr) \;\cup\;\bigr(\partial B(0,1)\times\{0\}\bigr)$.%\label{case three}
\item[(D)] $\bigr(\{0\}\times\big(B(0,1)\setminus\{0\}\big)\bigr)\;\cup\;\bigr(\big(B(0,1)\setminus\{0\}\big)\times\{0\}\bigr)$.%\label{case four}
\item[(E)] $\{(0,0)\}$.%\label{case five}
\end{enumerate}
In the subsequent Section~\ref{Se:limits a to e} we further restrict the sets in the cases~(B)--(D) and prove in all cases that the limits of $\phi$ take values in different parts of the boundary of the triangle $\triangle$.

Now the proof of the Theorem~\ref{Th:main} in Section~\ref{Se:proof} is a rather direct consequence: for any sequence in $\Omega$ the corresponding sequences $(\alpha_{1,n},\alpha_{2,n})_{n\in\mathbb{N}}$ have convergent subsequences. If both limits belong to $B(0,1)\setminus\{0\}$, then the corresponding sequence $a_n$ converges to the unique $a\in P_\bbR^4$ with the corresponding roots and the other two sequences converge to the values $b(a)$ of the section of $\calF$ constructed in the theorem. So in this case the limit belongs to $\calS_1^2$ and the sequence in $\Omega$ converges to the corresponding value of the map $\phi\circ b-\phi(a_0,b_{1,0},b_{2,0})$. In all other cases we prove in Section~\ref{Se:limits a to e} that the corresponding sequence in $\Omega$ has a convergent subsequence. Hence $\Omega$ is relatively compact. Furthermore, the boundary of $\Omega$ is a translated copy of the boundary $\partial\triangle$. This implies that $\Omega$ itself is a translated copy of the triangle $\triangle$, and the restriction of the map $\phi$ to any connected component of $\calS_1^2$ is a diffeomorphism onto $\triangle$. Every connected component contains a boundary point corresponding to the unique point in Theorem~\ref{th:blow-up} which belongs to the closure of the Wente family. Hence the uniqueness and connectedness of the Wente family imply that $\calS_1^2$ is connected.
\section{Preliminary results}
For the description of the boundary of $\calS^2$, the meromorphic function $f = b_2/b_1$, where $(b_1,b_2)$ is any basis of $\calB_a$, will turn out to be important. This function depends on the choice
of the basis $(b_1,b_2)$ only by a real M\"obius transformation in the range. Now suppose that $(b_1,b_2)$ is the normalised basis of $\calB_a$.
Then the corresponding function $f=\tfrac{b_2}{b_1}$ maps the unit circle onto the real line; the composition 
\begin{equation}
\label{eq:ftildedef}
\tilde{f} = \frac{1+\mi f}{1-\mi f} = \frac{\mi-f}{\mi+f} = \frac{b_1+\mi b_2}{b_1-\mi b_2}
\end{equation}
of $f$ with the Cayley transform maps the unit circle onto itself. We define the winding number $n(a)$ to be the winding number of $\tilde{f}|_{\bbS^1}$, see \cite[Section~3]{CS}; $n(a)$ depends only on $a$. We observe that the meromorphic differentials $\Theta(b_1) + \mi\Theta(b_2)$ and $\Theta(b_1) -  \mi\Theta(b_2)$ each have only one pole, namely at $\lambda=0$ respectively~at $\lambda=\infty$. 

As in \cite{CS}, we define for any $g \geq 0$ and $j\in \mathbb{Z}$
$$ V_j = \left\{ a\in \mathcal{H}^g \setminus \calS^g \mid n(a)=j \right\} \; . $$
As special cases of \cite[Thm~3.5]{CS} we then have
\begin{equation}\label{eq:Thm35CS}
\mathcal{H}^0 = V_1\;,\quad \mathcal{H}^1 = V_0 \quad\text{and}\quad 
\mathcal{H}^2 \setminus \calS^2 = V_1 \cup V_{-1} \; . 
\end{equation}
The sets $V_j$ mentioned in \eqref{eq:Thm35CS} are non--empty, open and in the case of the last equation, disjoint.

\begin{lemma}\label{L:intro:S2Vpm1}
We have $\calS^2 = \overline{V_1} \cap \overline{V_{-1}} \cap \mathcal{H}^2$.
\end{lemma}
\begin{proof}
This statement follows from Lemma~\ref{L:S2-smooth} and \cite[Thm~5.5(iii), (C) $\Rightarrow$ (B)]{CS}.
\end{proof}
We make repeated use of the following two results below (cf.\ \cite[Lemma~3.4]{KPS}).
\begin{lemma}\label{L:adouble}
Let $(a_n)_{n\in \N},\,(b_n)_{n\in\N}$ be sequences of holomorphic functions on an open, connected and simply-connected set $\Omega \subset \C$, which converge uniformly on compact subsets of $\Omega$ to $a\not\equiv 0$ respectively $b$. If all $a_n$ have only simple roots and all periods of $b_n(x)/y\,dx$ on the smooth curves $\Sigma_n=\{ (x,\,y) \in \Omega \times \C \mid y^2 = a_n (x) \}$ are purely imaginary, then the pull-back of $b(x)/y\,dx$ to the normalisation of  $\Sigma=\{ (x,\,y) \in \Omega \times \C \mid y^2 = a (x) \}$ is holomorphic.
\end{lemma}
\begin{proof}
Since the limit functions $a,\,b$ are holomorphic, then $b(x)/y\,dx$ is holomorphic at all $(x,\,y)$ where $a$ does not vanish. Further, the same holds at points where $a$ has a simple root, since the curve $\Sigma$ is smooth there. For any higher order root $x_0$ of $a$ choose a compact ball $\overline{B(x_0,\,\varepsilon)} \subset \Omega$. For sufficiently small $\varepsilon >0$ the function $a$ has no other roots in  $\overline{B(x_0,\,\varepsilon)}$.

For any $n \in \N$ the integral of $b_n(x)/y\,dx$ along a path between two roots of $a_n$ on the curve $\Sigma_n$ is imaginary as twice this integral is a period. Hence there exists a harmonic function $h_n$ on the curve $\Sigma_n$ with $dh_n = \RE\,(  b_n(x)/y\,dx )$. If in addition we assume that $h_n$ vanishes at all the roots of $a_n$, then this fixes the constants of integration and uniquely determines $h_n$.

On $\Sigma_n\cap\big(B(x_0,\,\varepsilon) \setminus \{x_0\} \times \C\big)$ the sequence $(h_n)_{n\in\N}$ of harmonic functions converges to a harmonic function $h$ on the smooth curve $\Sigma\cap\big(B(x_0,\,\varepsilon) \setminus \{x_0\} \times \C\big)$. By the maximum modulus principle each $h_n$ and $h$ are bounded by their values on $\Sigma_n\cap \big(\partial B(x_0,\,\varepsilon) \times \C\big)$ respectively $\Sigma\cap\big(\partial B(x_0,\,\varepsilon)\times\C\big)$. Hence $h$ extends to a harmonic function on $\Sigma\cap\big(B(x_0,\,\varepsilon) \times \C\big)$ by the theorem on removable singularities for bounded harmonic functions (cf.~ Theorem 2.3 in \cite{ABR}) This implies that that the imaginary part of the pull-back of $b(x)/y\,dx$ to the normalisation of $\Sigma$ is bounded. By complex linearity the same is true for the whole 1-form $b(x)/y\,dx$. This finishes the proof.
\end{proof}
\begin{corollary}\label{C:adouble}
The map $\C \times \calH^2 \to P^3_{\R}$ which maps $(z,\,a)$ to the unique $b \in \calB_a$ with $b(0) = z$ has a continuous extension to $\C \times \overline{\calH^2}$, where the closure of $\calH^2$ is taken in $P^4_\R$.
\end{corollary}
\begin{proof}
For any $z \in \C$ there exists $b \in P^3_\R$ with $b(0) = z$, which is unique up to addition with the product of $\lambda$ with an element of $P^1_\R$. Hence for $a \in \calH^2$ the 1-form $b(\lambda)/(\lambda \nu)\,d\lambda$ is unique up to addition with a holomorphic 1-form. The reality condition $\varrho^* \Theta(b) = - \Theta(b)$ implies that 
\[
\int_\gamma \Theta(b) \in \mi \R
\]
for all cycles $\gamma$ which are homologous to $\varrho \circ \gamma$. The $A$-cycles in figure \ref{fig:symmetric cycles}
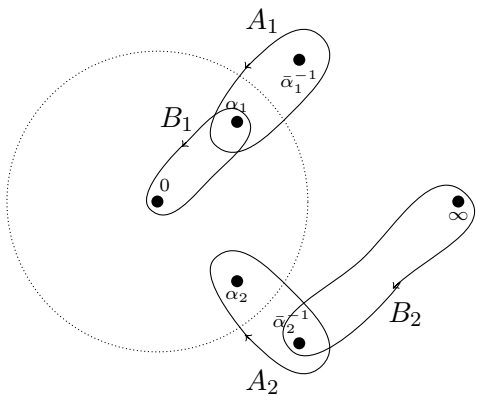
\begin{figure}
\centering
	\begin{tikzpicture} 
	\begin{scope}[scale=2.0]
	\draw[densely dotted] (0,0) circle[radius=1.0];
	\fill (0,0) circle[radius=0.04];
	\fill (0.530, 0.530) circle[radius=0.04];   % (3/4)*(1/sqrt(2),1/sqrt(2))
	\fill (0.530, -0.530) circle[radius=0.04];
	\fill (0.943, 0.943) circle[radius=0.04];   % (4/3)*(1/sqrt(2),1/sqrt(2))
	\fill (0.943, -0.943) circle[radius=0.04];
	\fill (2,0) circle[radius=0.04];
	\draw [->] (0.587, 0.887) to [out=45,in=135] (1.093,1.093) to [out=315,in=45] (0.887,0.587) to [out=225,in=315] (0.380,0.380) to [out=135,in=45] (0.582,0.887);
	\draw [<-] (0.587, -0.887) to [out=315,in=225] (1.093,-1.093) to [out=45,in=315] (0.887,-0.587) to [out=135,in=45] (0.380,-0.380) to [out=225,in=315] (0.582,-0.887);
	\draw[->] (0.165,0.365) to [out=45,in=135] (0.580,0.580) to [out=315,in=45] (0.365,0.165) to [out=225,in=315] (-0.05,-0.05) to [out=135,in=45] (0.165,0.365);
	\draw[->] (1.566,-0.578) to [out=45,in=315] (2.053,0.047) to [out=135,in=45] (1.378,-0.366) to [out=225,in=135] (0.890,-0.990) to [out=315,in=45] (1.566,-0.578);
	\draw (0.05,0) node[anchor=south] {\tiny $0$};
	\draw (0.530, 0.530) node[anchor=south] {\tiny $\alpha_1$};
	\draw (0.530, -0.530) node[anchor=north] {\tiny $\alpha_2$};
	\draw (0.943, 0.943) node[anchor=north] {\tiny $\bar{\alpha}_1^{-1}$};
	\draw (0.893, -0.943) node[anchor=south] {\tiny $\bar{\alpha}_2^{-1}$};
	\draw (2,0) node[anchor=north] {\tiny $\infty$};
	\draw (0.7,1.2) node {$A_1$};
	\draw (0.7,-1.2) node {$A_2$};
	\draw (0.13,0.55) node {$B_1$};
	\draw (1.65,-0.75) node {$B_2$};
	\end{scope} 
	\end{tikzpicture}
\caption{Symmetric and anti-symmetric cycles} \label{fig:symmetric cycles}
\end{figure}
are anti-symmetric with respect to $\varrho$, and the $B$-cycles are symmetric up to additive $A$-cycles. Thus $b \in P^3_\R$ belongs to $\calB_a$ if and only if $\int_{A_1} \Theta(b) = \int_{A_2} \Theta(b) = 0$. By Riemann's Bilinear Relations the map $P^1_\R \to \R^2$ defined by 
\[
	b \mapsto \left( \int_{A_1} \Theta(\lambda b) ,\, \int_{A_2} \Theta(\lambda b) \right)
\]
is an isomorphism. Hence for any $(z,\,a) \in \C \times \calH^2$ there exists a unique $b \in \calB_a$ with $b(0) = z$. Moreover, the map $(z,\,a) \mapsto b$ with $b(0)=z$ is continuous, since the following map is continuous: 
\[
	\calH^2 \times P^3_\R \to \R^2 \quad
	(a,\,b) \mapsto \left(  \int_{A_1} \Theta(b), \, \int_{A_2} \Theta(b) \right).
\]
Now let $(z_n,\,a_n)_{n \in \N} \subset \C \times \calH^2$ be a sequence which converges to $(z,\,a) \in \C \times P^4_\R$, and let $(b_n)_{n\in\N}$ be the sequence of images under the map $(z_n,\,a_n) \mapsto b_n$. By Lemma~\ref{L:adouble}, any accumulation point  $\tilde{b}$ of the sequence $(b_n/\| b_n \|)_{n\in \N}$ defines a 1-form $\tilde{b}/(\lambda \, \nu)\,d\lambda$ on the normalisation of $\{ (\lambda,\,\nu) \in \C^* \times \C \mid \nu^2 = \lambda \, a(\lambda) \}$, which is holomorphic at the higher order roots of $a$. We will show that $(\| b_n \|)_{n\in \N}$ is bounded. If we assume to the contrary that $(\| b_n \|)_{n\in \N}$ is not bounded, then there exists a divergent subsequence $(\| b_{n_k} \|)$ such that $z_n / \| b_n \|$ tends to $0$. Then $\tilde{b} (0) = 0$ and $\tilde{b}/(\lambda \, \nu)\,d\lambda$ is holomorphic on the normalisation with purely imaginary periods. By Riemann's Bilinear Relations it follows that $\tilde{b} \equiv 0$, contradicting that $\tilde{b}$ is an accumulation point of a sequence of polynomials of norm equal to $1$. Hence $(\| b_n \|)_{n\in \N}$ is bounded. This argument also shows that there exists a unique accumulation point $b$ of $(b_n)_{n\in\N}$ such that $b(\lambda)/(\lambda \, \nu)\,d\lambda$ is holomorphic on the normalisation at all higher order roots of $a$ with $b(0) = z$. The sequence $(b_n)_{n\in\N}$ converges to $b$ since it is bounded and $b$ is the unique accumulation point. This proves that the map $(z,\,a) \mapsto b$ has a continuous extension.
\end{proof}
%%%%%%%%%%%%%%%%%%%%
\section{The local Whitham flow}\label{Se:local whitham}
We begin with a general construction of the Whitham flows on $\calS^g_1$. More precisely, we construct for any $(a,b_1,b_2)\in\calF$ and any pair of values of $(\dot{q}_1,\dot{q}_2)\in(\mi\bbR)^2$ a tangent vector $(\dot{a},\dot{b}_{1},\dot{b}_2)\in T_{(a,b_1,b_2)}\calF$.
\begin{lemma}\label{L:para:whitham}
Let $g\geq 1$, $a\in \calS^g_{1}$ with $\deg \gcd(\calB_a)=1$ and $(b_1,\,b_2)$ a basis of $\calB_a$. Then the Whitham equations have for given values $(c_1(1),\,c_2(1)) \in \bbR^2$
\begin{align}\label{eq:para:whitham:1}
2\,a\,\dot{b}_\ind - \dot{a}\,b_\ind & = 2\mi\,\lambda\,a\,c_\ind' - \mi \,c_\ind\,(a+\lambda a') \quad\text{for $\ind=1,2$}\\		
\label{eq:para:whitham:2}
b_2\,c_1 - b_1\,c_2 & = Q\,a
\end{align}
a unique solution $(c_1,c_2,Q,\dot{a},\dot{b}_1,\dot{b}_2) \in P^{g+1}_\bbR \times P^{g+1}_\bbR \times P^2_\bbR \times T_a \calS_1^g \times P^{g+1}_\bbR \times P^{g+1}_\bbR $.

\noindent For $\ind=1,2$ the polynomials $c_\ind\in P^{g+1}_\bbR$ describe the deformations of the anti--derivative of $\Theta(b_\ind)$:
\begin{align}\label{eq:ck}
dq_\ind&=\Theta(b_\ind),&\dot{q}_\ind&=\mi c_\ind/\nu\,.
\end{align}
\end{lemma}
\begin{proof}
That the Whitham equations~\eqref{eq:para:whitham:1}--\eqref{eq:para:whitham:2} describe the Whitham flow was shown in \cite[Section~4]{CS}. 
In the setting of the lemma we have $a(1)\in \bbR_+$. For given $\dot{a}\in T_a\mathcal{H}^g$ the condition $\dot{a}\in T_a\calS^g_1$ is equivalent to  $\dot{b}_1(1)=0=\dot{b}_2(1)$. By equation~\eqref{eq:para:whitham:1} this in turn is equivalent to
\begin{equation} \label{eq:para:whitham:ck'}
c_\ind'(1) = c_\ind(1) \, \left(\tfrac{1}{2} + \tfrac{a'(1)}{2\,a(1)} \right) \quad\text{for $\ind=1,2$}\; .
\end{equation}
Due to equation~\eqref{eq:para:whitham:2} we know that $Q(1)=0$, so
\begin{equation}
\label{eq:para:whitham:Q'}
	a(1)\,Q'(1) = b_2'(1)\,c_1(1)-b_1'(1)\,c_2(1)\;,
\end{equation}
and
\begin{equation}
\label{eq:para:whitham:Q''}
	a(1)\,Q''(1) = b_2''(1)\,c_1(1) + 2\,b_2'(1)\,c_1'(1) - b_1''(1)\,c_2(1)-2\,b_1'(1)\,c_2'(1) - 2\,a'(1)\,Q'(1) \; .
\end{equation}
This shows that $c_1(1)$ and $c_2(1)$ uniquely determine the polynomial $Q \in \bbC^2[\lambda]$. Besides the common root of $b_1$, $b_2$ and $Q$ at $\lambda=1$, $b_1$ and $b_2$ have $g$ distinct roots. The values of $c_1$ at the roots of $\frac{b_1}{\lambda-1}$, and of $c_2$ at these roots of $\frac{b_2}{\lambda-1}$ are uniquely determined by equation~\eqref{eq:para:whitham:2}. Therefore $Q$ uniquely determines polynomials $\tilde{c}_1,\tilde{c}_2 \in\bbC^{g-1}[\lambda]$ and the divisor $q\in\bbC^1[\lambda]$ of the polynomial long division of $\frac{Qa}{\lambda-1}\in\bbC^{2g+1}[\lambda]$ divided by $\frac{b_1}{\lambda-1}\frac{b_2}{\lambda-1}\in\bbC^{2g}[\lambda]$ such that $(c_1,c_2)=(\Tilde{c}_1+q\frac{b_2}{\lambda-1},\Tilde{c}_2)$ solves equation~\eqref{eq:para:whitham:2}. The general solution $(c_1,\,c_2) \in (\bbC^{g+1}[\lambda])^2$ of equation~\eqref{eq:para:whitham:2} is then of the form $(c_1,c_2) = (\tilde{c}_1+q\frac{b_2}{\lambda-1},\tilde{c}_2) + p\,(\tfrac{b_1}{\lambda-1},\tfrac{b_2}{\lambda-1})$ with $p\in \bbC^1[\lambda]$. There exists $k\in \{1,2\}$ so that $b_\ind'(1)\neq 0$. Then there exists a unique polynomial $p\in \bbC^1[\lambda]$ so that $c_\ind(1)$ is the prescribed value and $c_\ind'(1)$ satisfies equation~\eqref{eq:para:whitham:ck'}. By definition of $Q$ and due to $b_\ind'(1)\ne0$, the other $c_{k}$ has the required value and derivative at $\lambda=1$. So $c_1$ and $c_2$ are uniquely determined in $\bbC^{g+1}[\lambda]$ by $c_1(1)$ and $c_2(1)$. Note that this is true even if for one $\ind\in\{1,2\}$ both terms $b_\ind'(1)$ and $b_\ind''(1)$ vanish and $Q$ does not depend on $c_\ind(1)$. In this case $Q$ alone does not uniquely determine the solutions $c_1$ and $c_2$ of~\eqref{eq:para:whitham:2}--\eqref{eq:para:whitham:ck'}. 

We next show $c_\ind \in P^{g+1}_\bbR$. For $p\in\bbC^m[\lambda]$ we have
\begin{gather}\begin{aligned}\label{derivative conjugate}
m\lambda^m\overline{p(\bar{\lambda}^{-1})}-2\lambda\tfrac{d}{d\lambda}\left(\lambda^m\overline{p(\bar{\lambda}^{-1})}\right)&=m\lambda^m\overline{p(\bar{\lambda}^{-1})}-2m\lambda^m\overline{p(\bar{\lambda}^{-1})}+\lambda^{m+1}\overline{\bar{\lambda}^{-2}\tfrac{d}{d\Bar{\lambda}^{-1}}p(\bar{\lambda}^{-1})}\\&=-\lambda^m\left(\overline{mp(\bar{\lambda}^{-1})-2\bar{\lambda}^{-1}p'(\bar{\lambda}^{-1})}\right).
\end{aligned}\end{gather}
For $m(mp-2\lambda p')-2\lambda\tfrac{d}{m\lambda}(dp-2\lambda p')=m^2p-4\lambda(m-1)p'+4\lambda^2p''$ we obtain
\begin{multline}\label{second derivative conjugate}
m^2\lambda^m\overline{p(\bar{\lambda}^{-1})}-4\lambda(m-1)\tfrac{d}{d\lambda}\left(\lambda^m\overline{p(\bar{\lambda}^{-1})}\right)+4\lambda^2\tfrac{d^2}{d\lambda^2}\left(\lambda^m\overline{p(\bar{\lambda}^{-1})}\right)\\=\lambda^m\left(\overline{m^2p(\bar{\lambda}^{-1})-4\bar{\lambda}^{-1}(m-1)p'(\bar{\lambda}^{-1})+4\lambda^2p''(\Bar{\lambda}^{-1})}\right).
\end{multline}
In order to show the invariance of the equations~\eqref{eq:para:whitham:ck'}--\eqref{eq:para:whitham:Q''} with respect to the transformation $(a,c_1,c_2,Q)\mapsto\left(\lambda^{2g}\overline{a(\Bar{\lambda}^{-1})},\lambda^{g+1}\overline{c_1(\Bar{\lambda}^{-1})},\lambda^{g+1}\overline{c_2(\Bar{\lambda}^{-1})},\lambda^2\overline{Q(\Bar{\lambda}^{-1})}\right)$, we rewrite these equations as
\begin{align*}
2a(1)\left((g+1)c_\ind(1)-2c_\ind'(1)\right)&=\left(2ga(1)-2a'(1)\right)c(1),\\
2a(1)Q'(1)=a(1)(2Q'(1)-2Q(1))&=(3b_1(1)-2b_1'(1))c_2(1)-(3b_2(1)-2b_2'(1))c_1(1),
\end{align*}
\vspace{-8mm}\begin{multline*}a(1)\left(4Q''(1)-4Q'(1)+4Q(1)\right)+2(2a'(1)-2ga(1))(2Q'(1)-2Q(1))\\=(4b_2''(1)-4gb_2'(1)+(g+1)^2b_2(1))c_1(1)-(4b_1''(1)-4gb_1'(1)+(g+1)^2b_1(1))c_2(1).
\end{multline*}
By equation~\eqref{derivative conjugate} this transformation replaces both sides of the first two equations by the negatives of their complex conjugates, and by~\eqref{derivative conjugate}--\eqref{second derivative conjugate} both sides of the third equations by their complex conjugates, respectively. So for given $(c_1(1),c_2(1))\in\bbR^2$ this transformation preserves the unique solution $(c_1,c_2,Q)$ of equations~\eqref{eq:para:whitham:2}--\eqref{eq:para:whitham:Q''}, which shows $c_\ind \in P_\bbR^{g+1}$ and $Q \in P_\bbR^2$. 

For every root $\alpha$ of $a$ there exists $k\in \{1,2\}$  so that $b_\ind(\alpha) \neq 0$. The corresponding equation in \eqref{eq:para:whitham:1} prescribes a value of $\dot{a}(\alpha)$, and in the case of $b_1(\alpha),b_2(\alpha) \neq 0$, equation~\eqref{eq:para:whitham:2} ensures that the two prescribed values for $\dot{a}(\alpha)$ are equal. The right hand side of~\eqref{eq:para:whitham:1} may be rewritten as
$$2\mi\,\lambda\,a\,c_\ind' - \mi \,c_\ind\,(a+\lambda a')=\mi\left(a\big(2\lambda c_\ind'-(g+1)c_\ind\big)-\big(\lambda a'-ga\big)c_\ind\right),$$
and belongs by equation~\eqref{derivative conjugate} to $P_\bbR^{3g+1}$. So $b_\ind\in P_\bbR^{g+1}$ implies $\alpha^{2g}\overline{\dot{a}(\Bar{\alpha}^{-1})}=\dot{a}(\alpha)$. Because the highest coefficient of $a\in\mathcal{H}^g$ is the square root of the product of all roots of $a$, which is unimodular, there exists a unique $\dot{a}\in T_a\mathcal{H}_g$ taking all these values $\dot{a}(\alpha)$ at the roots of $a$. By equation~\eqref{eq:para:whitham:1} we then have
$$ \dot{b}_\ind = \frac{2\mi\,\lambda\,a\,c_\ind' - \mi \,c_\ind\,(a+\lambda a') + \dot{a}\,b_\ind}{2a} \; . $$
The numerator vanishes at every root $\alpha$ of $a$ by the choice of $\dot{a}(\alpha)$ if $b_\ind(\alpha)\neq 0$, and due to equation~\eqref{eq:para:whitham:2} if $b_\ind(\alpha)=0$. So this $\dot{b}_\ind$ belongs to $\bbC^{g+1}[\lambda]$. Furthermore, it belongs to $P_\bbR^{g+1}$ since the numerator belongs to $P_\bbR^{3g+1}$ and $a\in \mathcal{H}^g\subset P_\bbR^{2g}$. In total the value $(c_1(1),c_2(1))\in\bbR^2$ and~\eqref{eq:para:whitham:1}--\eqref{eq:para:whitham:Q''} uniquely determine $(\dot{a},\dot{b}_1,\dot{b}_2)\in T_{a}\calS_1^g\times P^{g+1}_\bbR \times P^{g+1}_\bbR$ for given $(a,b_1,b_2)\in\calS^g\times\calB_a\times\calB_a$.
\end{proof}
%
%%%%%%%%%%%%%%%%%%%%%%%%
\section{The boundary of $\calS^2_1$ in $P_\bbR^4$}\label{Se:closure}
In Proposition~\ref{P:boundary-S21-C4lambda} we determine the closure of $\calS^2_1$ in $\bbC^4[\lambda]$. Given \,$a\in\calS^2_1$\, and a point on the boundary described in Proposition~\ref{P:boundary-S21-C4lambda}, we will show that one of the Whitham flows of Lemma~\ref{L:para:whitham} starts at \,$a$\, and terminates at the given boundary point. 
\begin{lemma}
\label{L:boundary-S21-C4lambda-pre}
For $a_0 = (\lambda-\lambda_0)^2(\lambda-\alpha)(\bar{\alpha}\lambda-1)/(\lambda_0\, |\alpha|)$ with $\lambda_0 \in \bbS^1$ and $0<|\alpha|<1$, the following three statements are equivalent:
\begin{enumerate}
\item[(i)] $a_0$ belongs to the closure of $\calS^2_{\lambda_0}$ in \,$\bbC^4[\lambda]$\,.
\item[(ii)] $a_0$ belongs to the closure of $\calS^2$ in \,$\bbC^4[\lambda]$\,.
\item[(iii)] $\mathrm{d}f(\lambda_0)=0$, where $f=\tfrac{b_2}{b_1}$ for a basis $b_1,b_2$ of $\calB_{a_0}$.
\end{enumerate}
\end{lemma}
Note that the condition $\mathrm{d}f(\lambda_0)=0$ is invariant under M\"obius transformations, so that it does not depend on the choice of the basis \,$b_1, b_2$\, and is well--defined even if $f(\lambda_0)=\infty$. 
\begin{proof}
Note that 
$$\overline{\mathcal{H}^2} = \left\{ a\in\C^4[\lambda] \left| \, |a(0)|=1, \,\overline{\lambda^4 a(1/\overline{\lambda})}=a(\lambda), \,\lambda^{-2}a(\lambda)\geq 0 \text{ for } \lambda\in \bbS^1 \right. \right\} \;.$$
If some polynomial $a \in \overline{\mathcal{H}^2}$ has roots of higher order, then there exists a unique $\tilde{a}\in\mathcal{H}^1\cup\mathcal{H}^0$ which has simple roots at the roots of $a$ of odd order and no others.
The quotient $a/\tilde{a}$ is the square of a polynomial $p$, which is unique up to sign. Therefore $\Sigma_{\tilde{a}}$ is the normalisation of $\Sigma_a$. It follows from Lemma \ref{L:adouble} for any $b\in\calB_a$, $\Theta(b)$ has no poles at higher order roots of $a$. Hence the polynomial $p$ divides $\gcd(\calB_a)$, and the map $b \mapsto p\, b$ is an isomorphism $\calB_{\tilde{a}} \to \calB_a$. This implies that $a$ and $\tilde{a}$ define the same function $f$, and therefore $n(a)=n(\tilde{a})$ holds. 

For  $a_0=(\lambda-\lambda_0)^2(\lambda-\alpha)(\bar{\alpha}\lambda-1)/(\lambda_0|\,\alpha|)$, we have $\tilde{a}_0 = (\lambda-\alpha)(\bar{\alpha}\lambda-1)/|\alpha| \in \mathcal{H}^1$ and therefore $n(a_0)=n(\tilde{a}_0)=0$ by equation~\eqref{eq:Thm35CS}.

(i) $\Rightarrow$ (iii): Now suppose $a_0\in\overline{\calS^2_{\lambda_0}}$.
By the implication (D) $\Rightarrow$ (B) in \cite[Theorem~5.8]{CS}, every neighbourhood $O$ of $a_0$ in $\C^4[\lambda]$ has non--empty intersection with $V_1$ and with $V_{-1}$.

Assume for a contradiction that $\mathrm{d}f(\lambda_0) \neq 0$ which is equivalent to $\mathrm{d}\tilde{f}(\lambda_0)\neq 0$. Here $\tilde{f}$ is defined in~\eqref{eq:ftildedef} with the normalised basis $(b_1,b_2)$ of $\calB_a$. We define $\mathrm{sign}(\mathrm{d}\tilde{f}(\lambda_0))$ to be either $+1$ or $-1$, according to whether the tangent map $\mathrm{d}\tilde{f}(\lambda_0)$ preserves or reverses the orientation of $\bbS^1$.
By the last formula in the proof of \cite[Theorem~3.5]{CS} there exists a neighbourhood $O$ of $a_0$ in $\C^4[\lambda]$ so that 
\begin{equation}
\label{eq:CS-Theorem35}
n(a)=n(a_0)-\mathrm{sign}(\mathrm{d}\tilde{f}(\lambda_0))
\end{equation}
for every $a \in O \cap (\mathcal{H}^2 \setminus \calS^2)$.  Therefore we have
\begin{align*}
	O \cap (\mathcal{H}^2\setminus \calS^2) \subset V_{-1} & \text{ for } \mathrm{sign}(\mathrm{d}\tilde{f}(\lambda_0))>0\;, \\
	O \cap (\mathcal{H}^2\setminus \calS^2) \subset V_{1} & \text { for } \mathrm{sign}(\mathrm{d}\tilde{f}(\lambda_0))<0\;.
\end{align*}
This implies $\mathrm{d}f(\lambda_0)=0$.

(iii) $\Rightarrow$ (ii): Let $\mathrm{d}f(\lambda_0)=0$. We now show that $\overline{\mathcal{H}^2}$ is near $a_0$ a manifold with boundary. Its elements near $a_0$ are of the form
$(\lambda-\beta)(\bar{\beta}\lambda-1)(\lambda-\alpha')(\bar{\alpha}'\lambda-1)/(|\alpha'|\,|\beta|)$
with $\beta \in B(\lambda_0,\epsilon)$ and $\alpha' \in B(\alpha,\epsilon)$ for some $\epsilon>0$. Inserting $\beta=re^{\mi \varphi}$ with $r>0,\,\varphi \in \R$ gives
$$ \frac{(\lambda-\beta)(\bar{\beta}\lambda-1)}{|\beta|} = \frac{\bar{\beta}}{|\beta|}\lambda^2 - \frac{1+\beta \bar{\beta}}{|\beta|}\lambda + \frac{\beta}{|\beta|} = e^{-\mi\varphi}\lambda^2 - (r+r^{-1})\lambda + e^{\mi\varphi} \; . $$
So $(r+r^{-1},\varphi,\alpha')\in[2,\infty)\times\bbR\times\mathbb{C}$ are local coordinates near $a_0$ of the manifold with boundary $\overline{\mathcal{H}_2}$ and the boundary near $a_0$ is given by 
$$ \partial \overline{\mathcal{H}^2} = \left\{ \left. a \in \overline{\mathcal{H}^2} \;\right|\; \text{$a$ has a double zero on $\bbS^1$} \right\} \; . $$
Thus every neighbourhood of $a_0 \in \partial \overline{\mathcal{H}^2}$ contains an open neighbourhood $O$ of $a_0$ such that $O \cap \mathcal{H}^2$ is connected. 
\begin{figure}\label{figure:delH2}
\centering
\begin{tikzpicture}
\definecolor{verylightgray}{gray}{0.9}
\begin{scope}[scale=1.0]
	\fill[verylightgray] (0,-2)--(4,-2)--(4,2)--(0,2)--(0,-2);
	\fill[gray] (0,-0.6) arc[start angle=-90, end angle=90, radius=0.6];
	\draw[black,very thick] (0,-2)--(0,2);
	\draw[black,very thick] (0,0)--(4,0);
	\fill[black] (0,0) circle[radius=0.1];
	\draw (2.0,1.0) node[anchor=east] {$V_1$};
	\draw (2.0,-1.0) node[anchor=east] {$V_{-1}$};
	\draw (4.0,0) node[anchor=west] {$\calS^2$};
	\draw (0,-2.0) node[anchor=north] {$\partial \mathcal{H}^2$};
	\draw (-0.1,0) node[anchor=east] {$a_0$};
	\draw[black,thin] (0.2,0.2)--(-0.5,0.7);
	\draw (-0.5,0.8) node[anchor=east] {$O$};
\end{scope}
\end{tikzpicture}
\caption{Sets in the proof of Lemma~\ref{L:boundary-S21-C4lambda-pre}}
\end{figure}
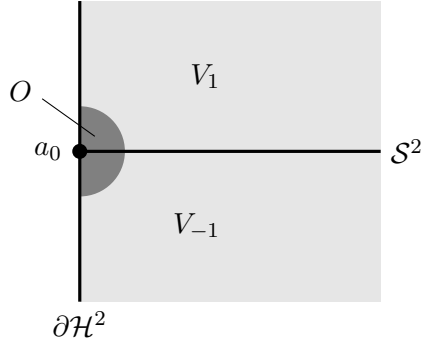
Since $b_1$ and $b_2$ have a common root at $\lambda=\lambda_0$, $\deg(f)$ is either $1$ or $2$. But \cite[Theorem~3.2]{CS} excludes $\deg(f)=1$, so $\deg(f)=2$ holds. By the Riemann--Hurwitz formula, $\mathrm{d}f$ has two roots. These must be distinct, because a double root would require the covering map $f$ to have at least three sheets. Therefore the root of $\mathrm{d}f$ at $\lambda=\lambda_0$ is simple. By equation~\eqref{eq:CS-Theorem35}, we have $O \cap V_1\neq \varnothing$ and $O \cap V_{-1} \neq \varnothing$. The sets $V_1$ and $V_{-1}$ are open and disjoint, but $O\cap\calH^2$ is connected, so $(O\cap V_1) \cup (O \cap V_{-1}) \subsetneq O\cap\calH^2$. This implies $O \cap \calS^2 \neq \varnothing$\ and $a_0\in\partial\calS^2$.  

(ii) $\Rightarrow$ (i): Let $(a_n)_{n\in\mathbb{N}}$ be a sequence in $\calS^2$ that converges to $a_0$. Due to Corollary \ref{C:adouble} the sequence $(\lambda_n)_{n\in\mathbb{N}}$ of common roots of the elements of $\calB_{a_n}$ depends continuously on $a_n$ and converges to a common root of the elements of $\calB_{a_0}$. By equation~\eqref{eq:Thm35CS}, $\deg \gcd \calB_{\tilde{a}_0}=0$. Therefore $\lim \lambda_n=\lambda_0$.
The rotation $\lambda \mapsto \lambda\,\lambda_0\,\lambda_n^{-1}$ transforms an element $a_n \in\calS^2$ into an element $\hat{a}_n$ of $\calS^2_{\lambda_0}$. Because of $\lim \lambda_n=\lambda_0$, also the sequence $(\hat{a}_n)$ converges to $a_0$. Hence $a_0$ is in the closure of $\calS^2_{\lambda_0}$.
\end{proof}
 
%\begin{lemma}
%	The closure of $\calS^2_{\lambda_0}$ in $\C^4[\lambda]$ does not contain
%	\begin{enumerate}
%		\item
%		$a(\lambda) = \tfrac{(\lambda-\beta)\, (\bar{\beta}\lambda-1)}{-|\beta|} \, \tfrac{(\lambda-\alpha)^2}{\alpha}$ with $\alpha \in \bbS^1\setminus \{1\}$ and $0<|\beta|<1$.
%		\item
%		$a(\lambda)=\tfrac{(\lambda-\alpha)^2\, (\bar{\alpha}\lambda-1)^2}{|\alpha|^2}$ with $\alpha \in \C^\times \setminus \bbS^1$.
%	\end{enumerate}
%\end{lemma}
\begin{proposition}\label{P:boundary-S21-C4lambda}
The boundary of $\calS^2_{\lambda_0}$ in $\C^4[\lambda]$ is 
\begin{equation} \label{eq:boundary-S21}
	\left\{ \left. \frac{(\lambda-\lambda_0)^2(\lambda-\alpha)(\bar{\alpha}\lambda-1)}{\lambda_0\,|\alpha|}\; \right|\; 0<|\alpha|<1, \,\mathrm{d}f(\lambda_0)=0 \right\} \;\cup\; \left\{ \frac{(\lambda-\lambda_0)^4}{\lambda_0^2}\right\} \; . 
\end{equation}
\end{proposition}
\begin{proof}
Because \,$\calS^2_{\lambda_0}$\, is a closed subvariety of \,$\calH^2$\,, 
the boundary of $\calS^2_{\lambda_0}$ is contained in $\partial \mathcal{H}^2 = \overline{\mathcal{H}^2} \setminus \mathcal{H}^2$. 
% $$ \partial \mathcal{H}^2 = \left\{ a\in\C[\lambda] \left| \; \deg(a)=4, |a(0)|=1, \lambda^{-4}\,\overline{a(1/\bar{\lambda})}=a(\lambda), \lambda^{-2}a(\lambda) \geq 0 \; \forall |\lambda|=1 \right. \right\} \;\setminus\; \mathcal{H}^2 \; . $$
For a general $a_0\in \partial \mathcal{H}^2$, the condition $|a_0(0)|=1$ excludes roots at $\lambda=0$. Thus any such $a_0$ has a higher order root at some $\lambda\in \C^\times$. From $\lambda^{-2} a(\lambda)\geq 0$ for $|\lambda|=1$ it follows that any unimodular root of $a_0$ has even order. Consequently, $a_0$ has either one ore two (possibly coinciding) double roots on $\bbS^1$, or two double roots away from $\bbS^1$ interchanged by $\lambda \mapsto \bar{\lambda}^{-1}$.

We now show for $a_0\in \partial\calS^2_{\lambda_0}$ that $a_0$ has an even order root at $\lambda=\lambda_0$. Thereby we exclude the possibility that $a_0$ has two even order roots on $\bbS^1\setminus \{\lambda_0\}$, or pairs of double roots off $\bbS^1$. Let $(b_1,b_2)$ be the normalised basis of $\calB_{a_0}$. By Corollary \ref{C:adouble} both $b_1$ and $b_2$ define continuous functions of $a_0\in \overline{\mathcal{H}^2}$. For $a_0\in \partial \calS^2_{\lambda_0}$ with higher order roots, we write $a_0(\lambda)=p^2(\lambda) \, \tilde{a}_0(\lambda)$ with $\tilde{a}_0 \in \mathcal{H}^{2-\deg{p}}$. The proof of Lemma \ref{L:adouble} shows that $p$ divides both $b_1$ and $b_2$, and $b_1/p,b_2/p \in \calB_{\tilde{a}}$. In particular, if $p(\lambda_0)\neq 0$, then $\tilde{a}_0\in \calS^{2-\deg{p}}$. Since $\calS^{g}=\varnothing$ for $g\in \{0,1\}$ by \cite{CS}, $p(\lambda_0)=0$ follows. 

We next exclude the case that $a_0$ has a double root at $\lambda=\lambda_0$ and another double root at some $\beta \in \bbS^1 \setminus \{\lambda_0\}$. Assume that this case occurs for some $a_0\in \partial \calS^2_{\lambda_0}$. Then $f$ has degree $1$, so $\mathrm{d}f$ has no zeros. We now use the notation of the proof of Lemma~\ref{L:boundary-S21-C4lambda-pre}. Let $\widehat{a}_0 = \bar{\beta}\,(\lambda-\beta)^2 \, \tilde{a}_0$. Then by \cite[Lemma~8]{CS1} and \cite[Theorem~3.2]{CS} there exists a neighbourhood $\widehat{O}$ of $\widehat{a}_0$ in $\C^2[\lambda]$ such that for $\widehat{a} \in \widehat{O}\cap \mathcal{H}^1$, $\mathrm{d}\tilde{f}_{\widehat{a}}(\lambda_0)$ is non--zero and $\mathrm{sign}(\mathrm{d}\tilde{f}_{\widehat{a}}(\lambda_0)) = \mathrm{sign}(\mathrm{d}\tilde{f}_{a_0}(\lambda_0)) = 1$. By equation~\eqref{eq:Thm35CS} we have $n(\widehat{a})=0$. We now choose a neighbourhood $O$ of $a_0$ in $\C^4[\lambda]$ whose pre--image with respect to the map $\widehat{a} \mapsto\bar{\lambda}_0(\lambda-\lambda_0)^2\,\widehat{a}$ is contained in $\widehat{O}$. Equation~\eqref{eq:CS-Theorem35} applies to $a \in O \cap \mathcal{H}^2$ and gives $n(a)=n(\widehat{a})-\mathrm{sign}(\mathrm{d}\tilde{f}_{\widehat{a}}(\lambda_0))=-1$. So $O\cap V_1 = \varnothing$. This contradicts Lemma~\ref{L:intro:S2Vpm1}. Now Lemma~\ref{L:boundary-S21-C4lambda-pre} shows that $\partial \calS^2_{\lambda_0}$ is contained in the set \eqref{eq:boundary-S21}.

We finally show that conversely the set \eqref{eq:boundary-S21} is contained in $\partial \calS^2_{\lambda_0}$. For the first set of the union, this is shown in Lemma~\ref{L:boundary-S21-C4lambda-pre}. Now we show that $\bar{\lambda}_0^2\,(\lambda-\lambda_0)^4$ is contained in the closure of the first set of the union. Consider $(\lambda-k)\, (\lambda- 1/k) \in \mathcal{H}^1$, with $k\in (0,1)$. Because of \cite[Proposition~2.2]{KSS} the corresponding $f$ is up to M\"obius transformations equal to $\frac{(\lambda-r)(\lambda-r^{-1})}{\lambda^2-1}$, where $r\in (k,1)$ is a root of a $b$ with a non--exact 1-form $\Theta(b)$. Then $df$ has roots at $\tfrac{2r \pm \mi (r^2-1)}{r^2+1} \in \bbS^1$. In the limit $k\to 1$ we have $r \to 1$. For $\alpha(k) = k\,\lambda_0\,\left( \tfrac{2r + \mi(r^2-1)}{r^2+1} \right)^{-1}$, we have $a(k):=\tfrac{(\lambda-\lambda_0)^2 \, (\lambda-\alpha(k))\, (\overline{\alpha(k)}\lambda-1)}{\lambda_0\,  |\alpha(k)|} \in \overline{\calS^2_{\lambda_0}}$ by Lemma~\ref{L:boundary-S21-C4lambda-pre}. In the limit $k\to 1$, $a(k)$ converges to $\bar{\lambda}_0^2\,(\lambda-\lambda_0)^4$, therefore $\bar{\lambda}_0^2\,(\lambda-\lambda_0)^4$ is in $\overline{\calS^2_{\lambda_0}}$. 
\end{proof}
\begin{lemma}
We have
$$ \overline{V_1} \cap \overline{V_{-1}} = \overline{\calS^2} = \calS^2 \cup \partial \calS^2
\quad \text{and} \quad
\partial \calS^2 = \bigcup_{\lambda_0\in \bbS^1} \partial \calS^2_{\lambda_0} \;, $$
where all boundaries and closures are taken in \,$\bbC^4[\lambda]$\, and $\partial \calS^2_{\lambda_0}$ is described in Proposition~\ref{P:boundary-S21-C4lambda}.
\end{lemma}
\begin{proof}
The second equality is obvious. For a sequence in $\calS^2$ with limit in $\partial \calS^2$, the sequence of the corresponding $\lambda_0 \in \bbS^1$ has a convergent subsequence, therefore $\bigcup_{\lambda_0\in \bbS^1} \partial \calS^2_{\lambda_0}$ is closed, whence the third equality follows.

Due to Lemma~\ref{L:intro:S2Vpm1} we have $\overline{\calS^2} \subset \overline{V_1} \cap \overline{V_{-1}} \subset \overline{\mathcal{H}^2}$. To prove the first equality, it therefore suffices to show that the points $a \in \partial \mathcal{H}^2 \setminus \partial \calS^2$ do not belong to $\overline{V_1} \cap \overline{V_{-1}}$. Any such $a$ either has two double roots off $\bbS^1$ or two different double roots on $\bbS^1$. In the first case, $a \mapsto f_a$ has a continuous extension to $\overline{\mathcal{H}^2}$ near $a$ by Corollary \ref{C:adouble}, so by \cite[equation~(2)]{CS}, the winding number $n(a)$ is locally constant on that neighbourhood. In the proof of Proposition~\ref{P:boundary-S21-C4lambda} we showed that in the second case, the winding number $n(a)$ is also locally constant near $a$. In either case, this implies $a\not\in \overline{V_1} \cap \overline{V_{-1}}$.
\end{proof}
%
%%%%%%%%%%%%%%%%%%%%%%%%%%%%%%
%
\section{A blowup at the Sym point}\label{Se:blow up sym point}
By Proposition~\ref{P:boundary-S21-C4lambda}, the  boundary of \,$\calS^2_1$\, in \,$\bbC^4[\lambda]$\, contains exactly one element $a(\lambda)=(\lambda-1)^4$ whose spectral curve has geometric genus zero. In this section we blowup this spectral curve to a one--dimensional family of elliptic curves. This family will become important in the proof of Theorem~\ref {Th:main}. More precisely, in Lemma~\ref{Le:case a} the composition $\phi\circ b$ is extended to this family and maps it to the boundary hypotenuse $\{\big(\varphi,\pi-\varphi)\mid\varphi\in[0,\pi]\}$ of the triangle $(\phi\circ b)[\calS_1^2]=\triangle$. The central element of this family is mapped by $\phi$ to $(\frac{\pi}{2},\frac{\pi}{2})$ and belongs to the Wente family which we study in section~\ref{Se:wente}.

After describing how we blowup $\calS^2_1$ at the boundary point $a(\lambda)=(\lambda-1)^4$, in Theorem~\ref{th:blow-up} we give an explicit description of the exceptional fibre of this blowup. The third statement of the theorem identifies the blowup with a set of spectral curves with properties similar to the defining properties of $\calS^2_1$. The second statement then explicitly describes this set of spectral curves.

Let us first present a different representation $\Hat{\calS}^2_0$ of  $\calS^2_1$. We replace the parameter $\lambda$ by a variant of its Cayley transform, namely the unique M\"obius transform which maps  $\lambda=0,\infty,1,\mi,-1,-\mi$ to  $\varkappa=\mi,-\mi,0,1,\infty, -1$, respectively. This transform is given by 
\begin{equation}
\label{eq:moebius-kappa}
\lambda \mapsto \varkappa=\frac{\lambda-1}{\mi(\lambda+1)}\,.
\end{equation}
The corresponding transformed objects are decorated by\; $\Hat{{}}$\;. Due to Proposition~\ref{P:boundary-S21-C4lambda} the elements of $\calS^2_1$ and their limits do not vanish at $\lambda=-1$ which corresponds to $\varkappa=\infty$. This is used in Lemma~\ref{L:boundary:S20-bounded} to show that $\Hat{\calS}^2_0$ is a bounded subset of $\bbR^4[\varkappa]$ with compact closure. 

\begin{definition}\label{def:hat objects}
Let $\Hat{\calH}^2$ denote the space of polynomials $a\in\bbR^4[\varkappa]$ with the following three properties: the highest coefficient is $1$, the values at real $\varkappa$ are positive and all roots are pairwise different. For any $a\in\Hat{\calH}^2$ let $\Hat{\calB}_a$ denote the  subspace of $b\in\bbR^3[\varkappa]$ such that $\Hat{\Theta}_b=\frac{b(\varkappa)}{\Hat{\nu}}\frac{d\varkappa}{\varkappa^2+1}$ is a meromorphic 1-form on the curve $\{(\varkappa,\Hat{\nu})\mid\Hat{\nu}^2=(\varkappa^2+1)a(\varkappa)\}$ with purely real periods. Finally, we denote
\begin{align*}
\Hat{\calS}^2_0&:=\Big\{a\in\Hat{\calH}^2\mid b(0)=0\text{ for all }b\in\Hat{\calB}_a\Big\}\quad\text{and}\\
\Hat{\calF}&:=\{(a,b_1,b_2)\in\Hat{\calS}^2_0\times\bbR^3[\varkappa]\times \bbR^3[\varkappa]\mid b_1, b_2\text{ is a basis of }\Hat{\calB}_a\}.
\end{align*}
\end{definition}
The differentials \,$\mi \Hat{\Theta}_{\Hat{b}}$\, with \,$\Hat{b} \in \Hat{\calB}_a$\, correspond under the M\"obius transformation to the differentials \,$\Theta(b)$\, with \,$b\in \calB_a$\,. 
Let us now construct a diffeomorphism $\calF\simeq\Hat{\calF}$:
\begin{lemma}\label{cayley transform}
The following map is a bundle diffeomorphism from $\calF$ onto $\Hat{\calF}$:
\begin{align*}
(a,b_1,b_2)\mapsto\left(\varkappa\mapsto\tfrac{(\mi+\varkappa)^4}{a(-1)}a(\tfrac{\mi-\varkappa}{\mi+\varkappa}),\varkappa\mapsto\tfrac{2\mi(\mi+\varkappa)^3}{\sqrt{a(-1)}}b_1(\tfrac{\mi-\varkappa}{\mi+\varkappa}),\varkappa\mapsto\tfrac{2\mi(\mi+\varkappa)^3}{\sqrt{a(-1)}}b_2(\tfrac{\mi-\varkappa}{\mi+\varkappa})\right).
\end{align*}
\end{lemma}
\begin{proof}
All three functions on the right hand side are polynomials with respect to $\varkappa$ and denoted by $(\Hat{a},\Hat{b}_1,\Hat{b}_2)$. The highest coefficient of $\Hat{a}$ is the alternating sum of the coefficients of $a$ divided by $a(-1)$, and hence \,$\Hat{a}$\, is monic. Since $\varkappa\mapsto\tfrac{\mi-\varkappa}{\mi+\varkappa}$ is the inverse of the Cayley transform $\lambda\mapsto\tfrac{\lambda-1}{\mi(\lambda+1)}$ and maps real $\varkappa$ onto unimodular $\lambda$, the following formula shows $\Hat{a}\in\Hat{\calH}^2$ for $a\in\calH^2$:
$$ \hat{a}(\varkappa) = \tfrac{(\mi+\varkappa)^4}{a(-1)}a(\tfrac{\mi-\varkappa}{\mi+\varkappa})=\tfrac{(\varkappa^2+1)^2}{a(-1)}(\tfrac{\mi-\varkappa}{\mi+\varkappa})^{-2}a(\tfrac{\mi-\varkappa}{\mi+\varkappa}) > 0 
\quad\text{for \,$a\in\calH^2$\, and \,$\varkappa\in\bbR$\,.} $$

For $\lambda=\tfrac{\mi-\varkappa}{\mi+\varkappa}$ we have
$$\lambda a(\lambda) = \tfrac{\mi-\varkappa}{\mi+\varkappa}a(\tfrac{\mi-\varkappa}{\mi+\varkappa})=-\tfrac{a(-1)}{(\mi+\varkappa)^6}(\varkappa^2+1)\tfrac{(\mi+\varkappa)^4}{a(-1)}a(\tfrac{\mi-\varkappa}{\mi+\varkappa}) \; . $$ 
Therefore
\begin{gather}\label{maps between curves}
(\lambda,\nu)\mapsto\left(\varkappa,\hat{\nu}\right)\quad\text{with}\quad\varkappa=\tfrac{\lambda-1}{\mi(\lambda+1)}\text{ and }\Hat{\nu}=\tfrac{\mi(\mi+\varkappa)^3}{\sqrt{a(-1)}}\nu
\end{gather}
maps the curve $\{(\lambda,\nu)\mid\nu^2=\lambda a(\lambda)\}$ onto the curve $\{ (\varkappa,\Hat{\nu}) \mid \Hat{\nu}^2=(\varkappa^2+1)\Hat{a}(\varkappa)\}$. Since $\tfrac{d\lambda}{\lambda}$ transforms into $\tfrac{\mi+\varkappa}{\mi-\varkappa}d(\tfrac{\mi-\varkappa}{\mi+\varkappa})=\tfrac{2\mi d\varkappa}{\varkappa^2+1}$ the following linear operator maps $\calB_a$ onto $\Hat{\calB}_{\Hat{a}}$:
\begin{align*}
b&\mapsto\hat{b},&\text{with}&&\hat{b}(\varkappa)&=\tfrac{2\mi(\mi+\varkappa)^3}{\sqrt{a(-1)}}b(\tfrac{\mi-\varkappa}{\mi+\varkappa})
\end{align*}
It is an isomorphism since $a(-1)>0$. The following calculation shows $\Hat{b}\in\bbR^3[\varkappa]$ for $b\in P_\bbR^3$:
$$\overline{\Hat{b}(\Bar{\varkappa})}=\tfrac{-2\mi(-\mi+\varkappa)^3}{\sqrt{a(-1)}}\overline{b\big(\tfrac{\mi-\Bar{\varkappa}}{\mi+\Bar{\varkappa}}\big)}=\tfrac{2\mi(\mi-\varkappa)^3}{\sqrt{a(-1)}}\overline{b\big(\overline{\tfrac{\mi+\varkappa}{\mi-\varkappa}}\big)}=\tfrac{2\mi(\mi-\varkappa)^3}{\sqrt{a(-1)}}(\tfrac{\mi-\varkappa}{\mi+\varkappa})^{-3}b(\tfrac{\mi-\varkappa}{\mi+\varkappa})=\tfrac{2\mi(\mi+\varkappa)^3}{\sqrt{a(-1)}}b(\tfrac{\mi-\varkappa}{\mi+\varkappa})=\Hat{b}(\varkappa).$$
Now $b$ vanishes at $\lambda=1$ if and only if $\Hat{b}$ vanishes at $\varkappa=0$, which finishes the proof.
\end{proof}
The element $a(\lambda)=(\lambda-1)^4$ of the family in Proposition~\ref{P:boundary-S21-C4lambda} corresponds to $\Hat{a}(\varkappa)=\varkappa^4$ in the closure of $\Hat{\calS}^2_0$ in $\bbR^4[\varkappa]$. It will turn out that this point in the closure of $\calS^2_1$ is the limit of all sequences in \,$\calS^2_1$\, whose image under \,$\phi \circ b$\, converges to the boundary hypotenuse $\{(\varphi,\pi-\varphi)\mid\varphi\in[0,\pi]\}$ of the triangle $\phi[\calS_1^2]=\triangle$. Therefore we shall blow up this point to a one--dimensional family of elliptic curves. In Lemma~\ref{Le:case a} we shall see that \,$\phi \circ b$\, continuously extends to a diffeomorphism from this family onto the boundary hypotenuse. The construction of the blowup is simplified by the fact that all the lower-order coefficients of \,$\hat{a}(\varkappa)=\varkappa^4$\, vanish. This is an additional advantage of using the coordinate \,$\varkappa$\, instead of \,$\lambda$\,. 
\begin{lemma}\label{map psi}
We define $\bbS^3[\varkappa]=\{\varkappa^4+a_1\varkappa^3+a_2\varkappa^2+a_3\varkappa+a_4\in\bbR^4[\varkappa]\mid a_1^{12}+a_2^6+a_3^4+|a_4|^3=1\}$ and the action
\begin{align}\label{action}
\bbR_+\times\bbR^4[\varkappa]&\to\bbR^4[\varkappa],&(C,a)&\mapsto C.a&\text{with}&&(C.a)(\varkappa)&=C^4a(\tfrac{\varkappa}{C}) \;,
\end{align}
which extends to an action of $[0,\infty)$ on $\bbR^4[\varkappa]$. Its restriction to $\bbR_+\times\bbS^3[\varkappa]$ is a diffeomorphism
\begin{gather}\label{diffeo psi}
\Psi:\bbR_+\times\bbS^3[\varkappa]\to\bbR^4[\varkappa]\setminus\{\varkappa^4\}.
\end{gather}
\end{lemma}
\begin{proof}
The set $\bbS^3[\varkappa]$ is the subset of monic $a\in\bbR^4[\varkappa]$ with $|a|=1$ with the following length:
$$a\mapsto|a|:=\big(a_1^{12}+a_2^6+a_3^4+|a_4|^3\big)^{\frac{1}{12}},\text{ which has the property }|C.a|=C|a|\text{ for all }C>0.$$
Due to this property the inverse of $\Psi$ is given by $a\mapsto\big(|a|,\frac{1}{|a|}.a\big)$.
\end{proof}
Now we are ready to define the blowup of $\Hat{\calS}^2_0$ at $\varkappa^4$:
\begin{definition}
The exceptional fibre of the blowup of $\Hat{\calS}^2_0$ at $\varkappa^4$ is defined as the subset $\Hat{\calE}^2_0\subset\bbS^3[\varkappa]$ such that $\{0\}\times\Hat{\calE}^2_0$ is the intersection of the closure of $\Psi^{-1}[\Hat{\calS}^2_0]$ in $[0,\infty)\times\bbS^3[\varkappa]$ with $\{0\}\times\bbS^3[\varkappa]$.
\end{definition}
This blowup is constructed in analogy to the real blowup $[0,\infty)\times\bbS^3$ of the point $0\in\bbR^4$. For the polynomials the usual scalar multiplication of the vector space $\bbR^4$ is replaced by the action~\eqref{action}. In order to state the subsequent theorem we now give the analogue to Definition~\ref{def:hat objects}.
\begin{definition}
For any $a\in\overline{\Hat{\calH}^2}$ let $\calB^\circ_a$ be the subspace of $b\in\bbR^3[\varkappa]$ such that the 1-form 
\begin{align}\label{limit 1-form blow up 3}
\Theta^\circ(b)&:=\frac{b(\varkappa)}{\nu}d\varkappa
\end{align}
is meromorphic without residues and has purely real periods on the curve $\{(\varkappa,\nu)\mid\nu^2=a(\varkappa)\}$. Furthermore, let $\calS^\circ_0$ denote $\{a\in\Hat{\calH}^2\mid b(0)=0\text{ for all }b\in\calB^\circ_a\}$ and $\overline{\calS^\circ_0}$ its closure in $\bbR^4[\varkappa]$.
\end{definition}
This definition emphasises the similarity between $\Hat{\calB}_a$ and $\calB^\circ_a$ and between $\Hat{\calS}^2_0$ and $\calS^\circ_0$. We shall see in the proof of the following theorem that this similarity extends to some statements of Lemma~\ref{L:boundary-S21-C4lambda-pre}. We have $\Theta^\circ(a')=2d\nu$, and hence the space $\calB^\circ_a$ contains the derivative $a'$, which is not contained in $\Hat{\calB}_a$.
Therefore $\calB^\circ_a$ will turn out to have a natural basis $(b,a')$, where $b$ is the unique monic element of $\calB^\circ_a\cap\bbR^2[\varkappa]$. In particular we have \,$a_3=0$\, for \,$a\in \overline{\calS_0^\circ}$\,. Note that the \,$\Theta(b)^\circ$\, with \,$b\in \calB_a^\circ$\, have up to third order poles at both points corresponding to \,$\varkappa = \infty$\,. 

Since all conditions on the elements of $\calS^\circ_0$ are invariant with respect to rescaling of the parameter $\varkappa$, this space is together with $\overline{\calS^\circ_0}\setminus\{\varkappa^4\}$ invariant with respect to the action~\eqref{action}. In the next theorem we shall show that the exceptional fibre $\Hat{\calE}^2_0$ is exactly the quotient of $\overline{\calS^\circ_0}\setminus\{\varkappa^4\}$ by this action. This just means that the condition on all elements of $\Hat{\calB}_a$ to vanish at $\varkappa=0$ is preserved by the blowup. In part~(iii) of the theorem we represent the elements of the quotient space of $\overline{\calS^\circ_0}\setminus\{\varkappa^4\}$ by the action~\eqref{action} in terms of a different normalisation, namely $a_2=1$. This normalisation is easier to preserve along the corresponding Whitham flow than the normalisation of $\bbS^3[\varkappa]$ and simplifies the description of the set of orbits of the action~\eqref{action} on $\overline{\calS^\circ_0}\setminus\{\varkappa^4\}$.
\begin{theorem}\label{th:blow-up}
\hspace{2em}
\begin{enumerate}
\item[(i)] The action~\eqref{action} preserves $\calS^\circ_0$ and $\overline{\calS^\circ_0}\setminus\{\varkappa^4\}$.
\item[(ii)] The map $\big({-}\tfrac{2}{\sqrt{3}},\tfrac{2}{\sqrt{3}}\big)\to\calS^\circ_0\cap\bbS^3[\varkappa],a_1\mapsto\big(a_1^{12}+1+g^3(a_1)\big)^{-\frac{1}{12}}.\big(\varkappa^4+a_1\varkappa^3+\varkappa^2+g(a_1)\big)$ is bijective, where $g:\big({-}\frac{2}{\sqrt{3}},\frac{2}{\sqrt{3}}\big)\to(0,\infty)$ is an analytic function with the following properties:
\begin{align}\label{def g}
g(-a_1)&=g(a_1)\text{ and }g(a_1)>\tfrac{(3a_1^2-4)(a_1^2-4)}{8(8-a_1^2)}\text{ for all }a_1\in\big({-}\tfrac{2}{\sqrt{3}},\tfrac{2}{\sqrt{3}}\big),&\lim_{a_1\to\pm\frac{2}{\sqrt{3}}}g(a_1)&=0.
\end{align}
\item[(iii)] The exceptional fibre $\Hat{\calE}^2_0$ is equal to $\{0\}\times\big(\overline{\calS^\circ_0}\cap\bbS^3[\varkappa]\big)$.
\end{enumerate}
\end{theorem}
The proof of this theorem uses arguments similar to those in the proof of Lemma~\ref{L:boundary-S21-C4lambda-pre} and is prepared in four lemmata. Let us now briefly describe how we use these lemmata. Lemma~\ref{lemma 1} is used twice. It contains an important part of the proof of Lemma~\ref{lemma 2} and also of the proof of Theorem~\ref{T:wente:wente}. Lemma~\ref{lemma 2} extends the inequality $\deg(\gcd(\calB_a))\le 1$ for $a\in\calH^2$ to $\deg(\gcd(\calB^\circ_a))\le 1$ for $a\in\Hat{\calH}^2$. Lemma~\ref{lemma 3} shows that given any $a$ in the exceptional fibre of the blowup of $\Hat{\calH}^2$ and any $b\in\calB^\circ_a$, there exist sequences \,$(a_n)$\, in \,$\Hat{\calH}^2$\, and $(b_n) \in\Hat{\calB}_{a_n}$ such that \,$(a_n)$\, converges to \,$a$\, and the blown up sequence of \,$(b_n)$\, converges to \,$b$\,.
Lemma~\ref{lemma 5} characterises the elements of $\overline{\calS^\circ_0}\setminus\Hat{\calH}^2$ and the elements of $\Hat{\calE}^2_0\setminus\Hat{\calH}^2$. The first characterisation is used in the proof of Theorem~\ref{th:blow-up}~(ii) and the second will imply together with Lemma~\ref{lemma 3} that $\Hat{\calE}^2_0\subset\overline{\calS^\circ_0}$ which is part of Theorem~\ref{th:blow-up}~(iii).
\begin{lemma}\label{lemma 1}
The polynomial $\varkappa^4+\alpha\varkappa^2+1$ belongs to $\Hat{\calH}^2$ if and only if $\alpha\in(-2,2)\cup(2,\infty)$. For such $\alpha$ there exists a unique $\beta\in\bbR$ for which $\frac{\varkappa^2+\beta}{\nu}d\varkappa$ has real periods on the curve $\nu^2=\varkappa^4+\alpha\varkappa^2+1$. Further, for \,$\alpha \in (-2,2)$\,, $\beta$ is a strictly increasing function of $\alpha$ with $\beta<\tfrac{\alpha}{2}$ which extends continuously to $\alpha=\pm2$ with values $\beta=\frac{\alpha}{2}$. Finally $(\alpha-2\beta)^2<\alpha^2-4$ holds for $\alpha\in(2,\infty)$.
\end{lemma}
\begin{remark}
The unique root \,$\alpha_0 \in (0,2)$\, of the function \,$\alpha \mapsto \beta$\, from Lemma~\ref{lemma 1} will be important in several places in the sequel, including Theorem~\ref{T:wente:wente} on the Wente family. This number can be described explicitly in terms of elliptic functions. Numerically, one obtains $\alpha_0 \approx 1.3044\dotsc$.%5.2178\dotsc$. 
\end{remark}
\begin{proof}
Set $a(\varkappa)=\varkappa^4+\alpha\varkappa^2+1$ and $b(\varkappa)=\varkappa^2+\beta$. If $\alpha\le-2$ or $\alpha=2$, then $a$ has real roots, which is excluded by $a(\varkappa)>0$ for real $\varkappa$. Next we consider $\alpha\in(2,\infty)$. In this case $a$ has four purely imaginary roots and is negative on two subintervals of $\mi\bbR\setminus\{0\}$. The 1-form $b(\varkappa/\nu\,d\varkappa$ is real on these subintervals. Hence there exists a unique $\beta\in\bbR$ such that $b$ has two roots in both subintervals and the integral of $b(\varkappa/\nu\,d\varkappa$ along these subintervals vanish. This is equivalent to all periods being purely real. An explicit calculation yields $(\alpha-2\beta)^2\le\alpha^2-4$ for such $\beta$. This shows the last statement.

If $\alpha\in(-2,2)$ then $a$ has for all combinations of signs of the real and the imaginary part a unique root and the real elliptic curve $\{(\varkappa,\nu)\in\bbC^2\mid\nu^2=a(\varkappa)\}$ is smooth and endowed with the anti--holomorphic involution $(\varkappa,\nu)\mapsto(\Bar{\varkappa},\Bar{\nu})$. Figure~\ref{fig:canonical basis} shows a canonical basis of the first homology group.
\begin{figure}
\centering
	\begin{tikzpicture}
	\begin{scope}[scale=1.0]
	\draw[densely dotted] (-4,0) to (4,0);
	\draw[densely dotted] (0,-2.5) to (0,2.5);
	\fill (2,1) circle[radius=0.04];
	\fill (2,-1) circle[radius=0.04];
	\fill (-2,1) circle[radius=0.04];
	\fill (-2,-1) circle[radius=0.04];
	\draw [->] (2.5,0) to [out=90,in=0] (2,1.2) to [out=180,in=90] (1.5,0) to [out=270,in=180] (2,-1.2) to [out=0,in=270] (2.5,0);
	\draw [->] (0,1.7) to [out=180,in=90] (-2.2,1) to [out=270,in=180] (0,0.3) to [out=0,in=270] (2.2,1) to [out=90,in=0] (0,1.7);
%	\draw (2,1) node[anchor=north] {\footnotesize $\alpha$};
%	\draw (2,-1) node[anchor=south] {\footnotesize $\bar{\alpha}$};
%	\draw (-2,1) node[anchor=north] {\footnotesize $-\bar{\alpha}$};
%	\draw (-2,-1) node[anchor=south] {\footnotesize $-\alpha$};
	\draw (2.45,0.5) node[anchor=west] {\footnotesize $A$};
	\draw (-1.2,1.7) node[anchor=south] {\footnotesize $B$};
	\draw (4,0) node[anchor=west] {\footnotesize $\bbR$};
	\draw (0,2.5) node[anchor=south] {\footnotesize $\mi\bbR$};
	\end{scope}
	\end{tikzpicture}
\caption{Canonical basis of the first homology group} \label{fig:canonical basis}
\end{figure}
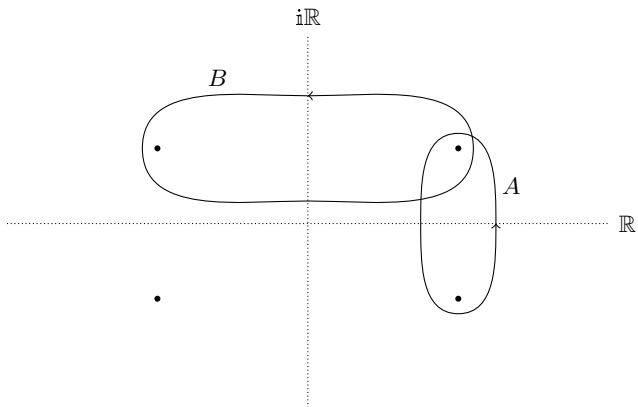
The fixed point set of $\rho$ consists of two cycles over $\varkappa\in\bbR$ which are both homologous to $\pm B$, and the cycle $A$ is anti--symmetric with respect to $\rho$, since $\rho$ maps $A$ into itself with two fixed points. As $\rho^\ast b(\varkappa)/\nu \,d\varkappa = \overline{b(\varkappa)/\nu \,d\varkappa}$, the integral of this 1-form along $A$ is imaginary and the integral along $B$ is real. Hence this 1-form has purely real periods, if and only if the integral along $A$ vanishes. Due to $\int_A\tfrac{d\varkappa}{\nu}\in\mi\bbR\setminus\{0\}$ there exists a unique $\beta\in\bbR$ such that $b(\varkappa)/\nu\,d\varkappa$ has purely real periods. This implies the second statement.

For $\alpha=-2$ the curve has two double points at $\varkappa=\pm 1$ and for $\alpha=-2$ at $\varkappa=\pm\mi$. Due to Lemma~\ref{L:adouble}, the pullbacks to the normalisation of the limits of the family of 1-forms $b(\varkappa)/\nu\,d\varkappa$ as \,$\alpha \to \pm 2$\, 
are holomorphic at the double points of the elliptic curve at $\varkappa^2=\pm 1$. In particular, the function $\alpha\mapsto\beta$ extends continuously to $\alpha=\pm2$ with the values $\beta=\frac{\alpha}{2}$.

It remains to prove the monotonicity of this function and $\alpha-2\beta>0$ for $\alpha\in(-2,2)$. By the theory of Whitham deformations, for any polynomial \,$c$\, there exists a flow preserving the periods of the meromorphic 1-form $\gamma\,b(\varkappa)/\nu\,d\varkappa$ such that the $t$-derivative of the local anti--derivative of the 1-form is equal to the global meromorphic function $c/\nu$. The Whitham equations then take the form
\begin{align}\label{eq:whitham hat}
	\frac{d}{dt}\gamma\frac{b(\varkappa)}{\nu}&=\frac{d}{d\varkappa}\frac{c(\varkappa)}{\nu}&\Longleftrightarrow&&2a(\varkappa)(\dot{\gamma}b(\varkappa)+\gamma\dot{b}(\varkappa))-\dot{a}(\varkappa)\gamma b(\varkappa)&=2a(\varkappa)c'(\varkappa)-a'(\varkappa)c(\varkappa).
\end{align}
Here we choose $c(\varkappa)=\gamma\varkappa(1-\alpha\beta+\beta^2)+\tfrac{1}{2}\gamma\varkappa(\alpha-2\beta)b(\varkappa)$. In our choice of \,$c$\,, the first summand eliminiates the poles in the resulting ode for \,$(\alpha, \beta, \gamma)$\, at the zeros of the resultant \,$1-\alpha\beta+\beta^2$\, of \,$a$\, and \,$b$\,, and the second summand serves as a rescaling to preserve the condition that \,$a(0)=1$\,. 

Comparing coefficients of $\varkappa^6$ in \eqref{eq:whitham hat} gives $\dot{\gamma}=\tfrac{1}{2}(\alpha-2\beta)\gamma$. Comparing coeffcients of $\varkappa^0$ gives
$$2(\dot{\gamma}\beta+\gamma\dot{\beta})=2\gamma(1-\alpha\beta+\beta^2+\tfrac{1}{2}(\alpha-2\beta)\beta)=\gamma(2-\alpha\beta),$$
which yields $\dot{\beta}=\tfrac{1}{2}(2-\alpha\beta-(\alpha-2\beta)\beta)=1-\alpha\beta+\beta^2$. We insert $\dot{\beta}$ and $\dot{\gamma}$ and calculate
$$\dot{a}(\varkappa)=\frac{a'(\varkappa)c(\varkappa)+2a(\varkappa)(\dot{\gamma}b(\varkappa)+\gamma\dot{b}(\varkappa)-c'(\varkappa))}{\gamma b(\varkappa)}=(4-\alpha^2)\varkappa^2.$$
In conclusion we obtain the following vector field on $\{(\alpha,\beta,\gamma)\in\bbR^3\mid\gamma\ne0\}$:
\begin{align*}
\dot{\alpha}&=4-\alpha^2,&\dot{\beta}&=1-\alpha\beta+\beta^2,&\dot{\gamma}=\tfrac{1}{2}(\alpha-2\beta)\gamma.
\end{align*}
Since $\dot{\alpha}$ depends only on $\alpha$, and $\dot{\beta}$ and $\tfrac{d}{dt}\ln\gamma$ depend only on $\alpha$ and $\beta$, this is also a vector field on $\alpha\in\bbR$ and on $(\alpha,\beta)\in\bbR^2$. The positivity of $\dot{a}$ and $\dot{b}$ for $\alpha\in(-2,2)$ implies that at the end points of any maximal trajectory in $(\alpha,\beta)\in(-2,2)\times\bbR$ either $\alpha\to\pm2$ or $\beta\to\pm\infty$ and that $\alpha\mapsto\beta$ is strictly increasing. We proved above that $\beta$ remains finite for $\alpha\in(-2,2)$ and converges at the end points $\alpha=\pm 2$ to $\beta\to\tfrac{1}{2}\alpha$. Hence $\alpha\mapsto\beta$ is strictly increasing on the unique maximal trajectory with purely real periods of $b(\varkappa)/\nu\,d\varkappa$ with the end points $(\alpha,\beta)=\pm(2,1)$.

For $\alpha\in(-2,2)$ and $\beta=\tfrac{1}{2}\alpha$ we have $\tfrac{d}{dt}(\alpha-2\beta)=2-\alpha^2+2\alpha\beta-2\beta^2=\tfrac{1}{2}(4-\alpha^2-(\alpha-2\beta)^2)>0$. Therefore the flow preserves $\alpha>2\beta$, and this inequality holds for all $\alpha\in(-2,2)$ if it holds for $\alpha\in(-2,-2+\epsilon)$ for some $\epsilon>0$. For the proof we show that $\alpha\mapsto\beta$ moves along an unstable manifold of the vector field at $(\alpha,\beta)=(-2,-1)$ on which $\alpha>2\beta$ holds. The root $(\alpha,\beta)=(-2,-1)$ is not hyperbolic, so we blow up the vector field by the coordinates $(x,y)$ for $\alpha=xy-2$ and $\beta=x-1$. Here $(x,y)\in\{0\}\times\bbR$ is the exceptional fibre of the blowup and the vector field transforms to
\begin{align*}
\dot{x}&=\dot{\beta}=1\!+\!(x\!-\!1)(x\!-\!xy\!+\!1)
=x(x+y-xy),&\dot{y}&=\tfrac{\dot{\alpha}-\dot{x}y}{x}=\tfrac{4-(xy-2)^2-x(x+y-xy)y}{x}=y(4\!-\!x\!-\!y).
\end{align*}
The flow can leave the exceptional fibre only at a root of $\dot{y}$ for which $x=0$. There are only two such roots: $(x,y)=(0,0)$ and $(x,y)=(0,4)$. The Jacobian of the vector field is $\big(\begin{smallmatrix}2x+y-2xy&x-x^2\\-y&4-x-2y\end{smallmatrix}\big)$. At the roots the Jacobian is equal to $\big(\begin{smallmatrix}0&0\\0&4\end{smallmatrix}\big)$ and $\big(\begin{smallmatrix}4&0\\-4&-4\end{smallmatrix}\big)$, respectively. The root at $(x,y)=(0,0)$ is not hyperbolic and the unique eigenvector of the unique non--zero eigenvalue corresponds to a trajectory which remains inside the exceptional fibre. The other root at $(x,y)=(0,4)$ is hyperbolic with a one--dimensional stable manifold and a one--dimensional unstable manifold. The stable manifold is also contained in the exceptional fibre $(x,y)\in\{0\}\times\bbR$, but the unstable manifold moves out of the exceptional fibre. For small positive $x$, and $y$ nearby $4$ we have $\alpha=xy-2>-2$, $\beta=x-1>-1$ and $\alpha-2\beta=x(y-2)>0$. The logarithmic derivative of $\gamma$ is $\tfrac{1}{2}\alpha-2\beta=\tfrac{1}{2}x(y-2)$ and nearby $(x,y)=(0,4)$ is approximately $2x$ and thus smaller than $\dot{x}$ which is approximately $4x$. Therefore there exists a solution of $\gamma$ on the unstable manifold with positive $x$, on which $(\alpha,\beta,\gamma)$ converges in the limit $t\to-\infty$ to $(-2,-1,1)$. In this limit the integral of $\gamma\,b(\varkappa)/\nu \,d\varkappa$ along the cycle $A$ vanishes, since $A$ collapses to the point at $\varkappa=1$. Since this integral is preserved along the flow, this integral vanishes on the entire unstable manifold. Hence this unstable manifold moves along the unique function $\alpha\mapsto\beta$ such that the periods of $b(\varkappa)/\nu\,d\varkappa$ are purely real. In particular, the inequality $\alpha>2\beta$ holds along this function, since $\alpha-2\beta=x(y-2)$ is positive for sufficiently negative $t$.
\end{proof}
\begin{lemma}\label{lemma 2}
Consider $a(\varkappa)=\varkappa^4+a_1\varkappa^3+\frac{1}{4}a_1^2\varkappa^2+a_4\in\bbR^4[\varkappa]$ and $b(\varkappa)=\varkappa(\varkappa+\tfrac{1}{2}a_1)\in\bbR^2[\varkappa]$ such that $a(\varkappa)\ge0$ for real $\varkappa$ and $\tfrac{b(\varkappa)}{\nu}d\varkappa$ has on the curve $\nu^2=a(\varkappa)$ real periods. Then $a_4$ vanishes.
\end{lemma}
\begin{proof}
The conditions on $a$ imply $a_4=a(0)\ge0$ which in turn implies $(\tfrac14 a_1)^4 + a_4 \ge0$. We now assume \,$a_4 > 0$\,, choose \,$C = \sqrt[4]{(\tfrac14 a_1)^4 + a_4} > \tfrac14|a_1| \geq 0$\,
and consider the transformed polynomials 
\begin{align*}
C^{-4}a(C\varkappa-\tfrac{1}{4}a_1)&=(\varkappa-\tfrac{a_1}{4C})^2(\varkappa-\tfrac{a_1}{4C}+\tfrac{a_1}{2C})^2+\tfrac{a_4}{C^4}=\varkappa^4+\alpha\varkappa^2+1&\text{with}&&\alpha&=-2(\tfrac{a_1}{4C})^2\,,\\
C^{-2}b(C\varkappa-\tfrac{1}{4}a_1)&=(\varkappa-\tfrac{a_1}{4C})(\varkappa-\tfrac{a_1}{4C}+\tfrac{a_1}{2C})=\varkappa^2+\beta&\text{with}&&\beta&=\tfrac12 \alpha.
\end{align*}
This transformation preserves the property that $\tfrac{b(\varkappa)}{\nu}d\varkappa$ has purely real periods on the curve $\nu^2 = a(\varkappa)$. Furthermore, the transformed polynomials satisfy the conditions of Lemma~\ref{lemma 1}. Since \,$C > \tfrac14 |a_1|$\, we have  \,$-2<\alpha\leq 0$\,, and therefore Lemma~\ref{lemma 1} implies \,$\beta < \tfrac{\alpha}{2}$\,, which is a contradiction. 
\end{proof}
\begin{lemma}\label{lemma 3}
Let $(a_n)_{n\in\bbN}$ be a sequence in $\bbS^3[\varkappa]\cap\calH^2$ which converges to $a \in \bbS^3[\varkappa]$, and let $(C_n)_{n\in\bbN}$ be a sequence of positive numbers with limit $0$. For any $b\in\calB^\circ_a$ there exists a sequence $(b_n)_{n\in\bbN}$ with $b_n\in\Hat{\calB}_{C_n.a_n}$ for all $n\in\bbN$ such that $(\Tilde{b}_n)_{n\in\bbN}$ with $\Tilde{b}_n(\varkappa)=C_n^{-1}b_n(C_n\varkappa)$ converges in $\bbR^3[\varkappa]$ to $b$.
\end{lemma}
\begin{remark}\label{rescaling b}
The map $(\varkappa,\nu) \mapsto (C_n\varkappa,C_n^2\nu)$ from $\Sigma_n = \{(\varkappa,\nu) \mid \nu^2=(C_n^2\varkappa^2+1)a_n(\varkappa)\}$ to $\Sigma_{C_n.a_n}=\{(\varkappa,\nu)\mid\nu^2=(\varkappa^2+1)(C_n.a_n)(\varkappa)\}$ pulls back the 1-form $b_n(\varkappa)/(\varkappa^2+1)\,\nu)\, d\varkappa$ on \,$\Sigma_{C_n.a_n}$\, to the 1-form $\Tilde{b}_n(\varkappa)/((C_n^2\varkappa^2+1)\nu) \,d\varkappa$ on \,$\Sigma_n$\,. The curves \,$\Sigma_n$\, converge for $n\to\infty$ to the curve $\Sigma_\infty = \{(\varkappa,\nu)\mid\nu^2=a(\varkappa)\}$. So Lemma~\ref{lemma 3} states that any 1-form $b(\varkappa)/\nu \,d\varkappa$ without residues and with purely real periods on \,$\Sigma_\infty$\, is the limit of a sequence of 1-forms on \,$\Sigma_n$\, which are the pullbacks of 1-forms $b_n(\varkappa)/((\varkappa^2+1)\nu)\,d\varkappa$ with purely real periods on the curves \,$\Sigma_{C_n.a_n}$\,.
\end{remark}
\begin{proof}
By Riemann's Bilinear Relations on any compact Riemann surface there exists  for any given basis of the first homology group a dual basis of the real space of holomorphic 1-forms with respect to the pairing given by the imaginary parts of the periods. This implies that for any Mittag Leffler distribution $f$ there exists a unique meromorphic 1-form $\omega$ such that $df-\omega$ is holomorphic and $\omega$ has purely real periods. Hence $\Hat{\calB}_{C_n.a_n}\to\bbC, b\mapsto b(\mi)$ is an isomorphism of real linear spaces. Similarly $\calB^\circ_a\to\bbR^2$, which maps $\Tilde{b}$ onto the two highest coefficients of $\Tilde{b}$ is an isomorphism. Now for given $a\in\bbS^3[\varkappa]\cap\calH^2$ the lemma follows from the combination of the following two statements:
\begin{enumerate}
\item[(i)] For sufficiently large $n$ the composition of the following maps is an \,$\bbR$-linear isomorphism: first map $z\in\bbC$ onto the unique $b_n\in\Hat{\calB}_{C_n.a_n}$ with $b_n(\mi)=z$, then map $b_n$ onto $\Tilde{b}_n$ with $\Tilde{b}_n(\varkappa)=C_n^{-1}b_n(C_n\varkappa)$ and finally map $\Tilde{b}_n$ onto the two highest coefficients of $\Tilde{b}_n$.
\item[(ii)] For any sequence $(b_n)_{n\in\bbN}$ with $b_n\in\Hat{\calB}_{C_n.a_n}$ the sequence $(\Tilde{b}_n)_{n\in\bbN}$ with $\Tilde{b}_n(\varkappa)=C_n^{-1}b(C_n\varkappa)$ converges in \,$\bbR^3[\varkappa]$\, if and only if the sequence of the two highest coefficients of $(\Tilde{b}_n)_{n\in\bbN}$ converges. Moreover, any such limit belongs to $\calB^\circ_a$.
\end{enumerate}
For the proof of (ii) we claim that the boundedness of the two highest coefficients of any such sequence $(\Tilde{b}_n)_{n\in\bbN}$ already implies that the sequence itself is bounded. In fact otherwise there exists such a sequence with $\|\Tilde{b}_n\|=1$ for some norm on $\bbR^3[\varkappa]$ such that the two highest coefficients converge to zero. After passing to a subsequence, \,$(\tilde{b}_n)$\, converges to an element $\Tilde{b}$ whose two highest coefficents vanish. Furthermore, all periods of $\Tilde{b}(\varkappa)/\nu\,d\varkappa$ on the limit curve $\nu^2= a(\varkappa)$ are limits of real numbers and hence real.
Choose a cycle on each of the curves $\nu^2=(C_n^2\varkappa^2+1)a_n(\varkappa)$ which is anti--symmetric with respect to the involution $(\varkappa,\nu)\mapsto(\Bar{\varkappa},\Bar{\nu})$ and which projects to a path on the \,$\varkappa$-plane which connects the branch points \,$\pm \mi/C_n$\, but passes on one side of all the roots of \,$a_n$\,. The periods of $\Tilde{b}_n(\varkappa)/\nu\,d\varkappa$ along these cycles vanish. These periods converge to the residue of $\Tilde{b}(\varkappa)/\nu\,d\varkappa$ at either of the points above \,$\varkappa=\infty$\,, and therefore both of these residues vanish. This implies $\Tilde{b}\in\bbR^1[\varkappa]\cap\calB^\circ_a$, and therefore $\Tilde{b}=0$ since the elements of $\calB^\circ_a$ are determined by their two highest coefficients. This contradicts $\|\Tilde{b}_n\|=1$ and shows the claim. 

Now suppose that the two highest coefficients of $(\Tilde{b}_n)_{n\in\bbN}$ converge. Then by the claim, this sequence is contained in a compact ball of $\bbR^3[\varkappa]$. The sequence $(\Tilde{b}_n)_{n\in\bbN}$ can have only one accumulation point, namely the unique element \,$\tilde{b} \in \calB_a^\circ$\, whose two highest coefficients are the limits of the highest coefficients of the sequence. Hence the sequence converges to \,$\tilde{b}$\,. 
Conversely, if $(\Tilde{b}_n)_{n\in\bbN}$ converges, the sequence of the two highest coefficents converges. This proves (ii).

For the proof of (i) we have to show that for sufficiently large $n$, the map from $b_n(\mi)$ to $(\Tilde{b}_{2,n},\Tilde{b}_{3,n})$ with $\Tilde{b}_n(\varkappa)=\Tilde{b}_{3,n}\varkappa^3+\Tilde{b}_{2,n}\varkappa^2+\Tilde{b}_{1,n}\varkappa+\Tilde{b}_{0,n}$ is an isomorphism.  Due to $b_n(\mi)=C_n\Tilde{b}_n(\mi/C_n)$ we have
\begin{align*}
\RE(b_n(\mi))&=-C_n^{-1}\Tilde{b}_{2,n}+C_n\Tilde{b}_{0,n},&\IM(b_n(\mi))&=-C_n^{-2}\Tilde{b}_{3,n}+\Tilde{b}_{1,n}.
\end{align*}
By the above claim, $|\Tilde{b}_{0,n}+\mi\Tilde{b}_{0,n}|\le M |\Tilde{b}_{2,n}+\mi\Tilde{b}_{2,n}|$ for some $M>0$. Hence for $C_n\le 1/\sqrt{3M}$, the composition of $x+\mi y\mapsto -C_n^{-1}x-C_n^{-2}\mi y$ with $b(\mi)\mapsto(\Tilde{b}_{2,n},\Tilde{b}_{3,n})$ obeys
$$|\Tilde{b}_{2,n}+\mi\Tilde{b}_{n,3}-(x+\mi y)|=C_n^2|\Tilde{b}_{0,n}+\mi\Tilde{b}_{1,n}|\le\tfrac{1}{3}|\Tilde{b}_{2,n}+\mi\Tilde{b}_{n,3}|.$$
We insert this into $|\Tilde{b}_{2,n}+\mi\Tilde{b}_{3,n}|\le|x+\mi y|+|\Tilde{b}_{2,n}+\mi\Tilde{b}_{n,3}-(x+\mi y)|$ and obtain
$$|\Tilde{b}_{2,n}+\mi\Tilde{b}_{n,3}-(x+\mi y)|\le\tfrac{1}{3}|\Tilde{b}_{2,n}+\mi\Tilde{b}_{3,n}|\le\tfrac{1}{2}|x+\mi y|.$$
Therefore $(x,y)\mapsto(\Tilde{b}_{2,n},\Tilde{b}_{3,n})$ is invertible, and hence the composition of $z\mapsto-(C_n\RE(z),C_n^2\IM(z))$ with this map is also invertible. This composition is the map in (i). This proves (i).
\end{proof}
\begin{lemma}\label{lemma 5}
Any $\varkappa^4+a_1\varkappa^3+a_2\varkappa^2+a_3\varkappa+a_4\in\big(\overline{\calS^\circ_0}\cup\Hat{\calE}^2_0\big)\setminus\Hat{\calH}^2$ satisfies $a_3=0=a_4$ and $a_2=\frac{3}{4}a_1^2$.
\end{lemma}
\begin{proof}
For any \,$a \in \overline{\calS_0^\circ}$\, we have \,$a' \in \calB_a^\circ$\, and therefore \,$a_3=a'(0)=0$\,. By Lemma~\ref{lemma 3} the same holds for \,$a \in \Hat{\calE}^2_0$\,. Moreover, any \,$a \in \overline{\Hat{\calH}^2}\setminus \Hat{\calH}^2$\, has a higher order root \,$\varkappa^*$\,.
By Lemma~\ref{L:adouble}, the unique monic $b\in\calB^\circ_a\cap\bbR^2[\varkappa]$ vanishes at \,$\varkappa^*$\,. We have either \,$\varkappa^*=0$\, and then \,$a_4=0$\,, or \,$\varkappa^*\neq 0$\, and then  $\deg(\gcd(b,a'))\ge 2$. In the latter case we are in the situation of Lemma~\ref{lemma 2}, and therefore \,$a_4=0$\, holds in both cases. 
Because of \,$a \in \overline{\Hat{\calH}^2}$\,, we have \,$a(\varkappa) \geq 0$\, for real \,$\varkappa$\,. Therefore \,$a_2=0$\, implies \,$a_1=0$\,, which means  $a_2=\tfrac{3}{4}a_1^2$. Thus we suppose \,$a_2 \neq 0$\, in the sequel.
The lemma follows if in analogy to Lemma~\ref{L:boundary-S21-C4lambda-pre} for such $a$ and the corresponding $b=\varkappa(\varkappa+\tfrac{1}{2}a_1)$ the derivative
$$\frac{d}{d\varkappa}\frac{b(\varkappa)}{a'(\varkappa)}=\frac{\varkappa+\frac{1}{2}a_1}{4\varkappa^2+3a_1\varkappa+2a_2}=\frac{(4\varkappa^2+3a_1\varkappa+2a_2)-(8\varkappa+3a_1)(\varkappa+\tfrac{1}{2}a_1)}{(4\varkappa^2+3a_1\varkappa+2a_2)^2}$$
vanishes at $\varkappa=0$. In fact, this is equivalent to $a_2=\tfrac{3}{4}a_1^2$. 
We now proceed by arguments similar to those in the proof of \cite[Lemma~8]{CS1}. 

We first consider the case \,$a \in \overline{\calS^\circ_0}$\,.  
Then \,$a$\, is a limit of a sequence \,$(a_n)_{n\in \N}$\, in \,$\calS^\circ_0$\,. By Lemma~\ref{L:adouble}, the corresponding sequence \,$(b_n)_{n\in \N}$\, of the unique monic elements \,$b_n \in \calB^\circ_{a_n} \cap \R^2[\varkappa]$\, converges to \,$b$\,. Any \,$a_n$\, has two distinct roots near \,$\varkappa=0$\, and they are complex conjugates of each other. The path on the curve $\nu_n^2=a_n(\varkappa)$ that projects to the straight line in the \,$\varkappa$-plane joining these two roots of \,$a_n$\, is anti--symmetric with respect to the involution $(\varkappa,\nu_n)\mapsto(\Bar{\varkappa},\Bar{\nu}_n)$. Hence the integral of $b_n/\nu_n\,d\varkappa$ along this path vanishes. Consequently on a tubular neighbourhood of this path, there exist a unique anti-derivative \,$q_n$\, of $b_n/\nu_n\,d\varkappa$ which vanishes at both ramification points. This \,$q_n$\, is  anti-symmetric with respect to $\sigma:(\varkappa,\nu_n)\mapsto(\varkappa,-\nu_n)$ and hence the quotient $q_n/\nu_n$ is symmetric with respect to $\sigma$ and therefore a local function of $\varkappa$. Since \,$\nu_n$\, has only simple roots at the ramification points, \,$\varkappa \mapsto q_n/\nu_n$\, is holomorphic. Because of \,$b_n(0)=0=a_n'(0)$\,, we have  \,$\left.\tfrac{d}{d\varkappa}\right|_{\varkappa = 0} q_n/\nu_n = 0$\,. Since \,$a_2\neq 0$\, and \,$b(0)=0$\,, there exists a regular anti-derivative \,$q$\, of \,$b/\nu\,d\varkappa$\, on the curve $\nu^2=a(\varkappa)$ which vanishes at the singularity at \,$\varkappa=0$\,.
The corresponding local function \,$\varkappa \mapsto q/\nu$\, is holomorphic on a neighbourhood of \,$\varkappa=0$\,. By construction, it is the limit of the local holomorphic functions \,$q_n/\nu_n$\,. In particular \,$\left.\tfrac{d}{d\varkappa}\right|_{\varkappa = 0} q/\nu = 0$\,.
Because \,$a$\, has at \,$\varkappa=0$\, a zero of the order exactly \,$2$\,, \,$q$\, and \,$\nu$\, are local functions in \,$\varkappa$\, there. Moreover \,$\nu$\, has a zero of the order exactly \,$1$\,, and \,$q$\, has a zero of order at least \,$1$\,. Therefore there exists a unique \,$t\in \bbR$\, such that \,$q+t\nu$\, has a zero of order at least \,$2$\,. Then the equation  \,$\left.\tfrac{d}{d\varkappa}\right|_{\varkappa = 0} q/\nu = 0$\, implies that \,$q+t\nu$\, has in fact a zero of order at least \,$3$\,, and by the definition of \,$q$\,, the same holds for the polynomial \,$b+\tfrac{t}{2}a'$\,. This shows \,$\left.\tfrac{d}{d\varkappa}\right|
_{\varkappa = 0} b/a' = 0$\,. 

In a second step we extend this argument to the case  $a\in\Hat{\calE}^2_0$. Then the pair $(0,a)$ is the limit of a sequence $(C_n,a_n)_{n\in\mathbb{N}}$ in $(0,\infty)\times\bbS^3[\varkappa]$ with $C_n.a_n\in\Hat{\calS}^2_0$. Let $b$ be the unique monic element of $\calB^\circ_a\cap\bbR^2[\varkappa]$. By Lemma~\ref{lemma 3}, $(b,a')$ is the limit of a sequence $(\Tilde{b}_{1,n},\Tilde{b}_{2,n})_{n\in\mathbb{N}}$ such that $(b_{1,n},b_{2,n})$ with $b_{i,n}(\varkappa)=C_n\Tilde{b}_{i,n}(\varkappa/C_n)$ is a basis of $\Hat{\calB}_{C_n.a_n}$. As above, for sufficiently large $n$ the 1-forms $(\Tilde{b}_{1,n}/((C_n^2\varkappa^2+1)\,\Tilde{\nu}_n)\,d\varkappa,\,\Tilde{b}_{2,n}/((C_n^2\varkappa^2+1)\,\Tilde{\nu}_n)\,d\varkappa)$ have anti--derivatives $(q_{1,n},q_{2,n})$ on an analogous tubular neighbourhood of both ramification points nearby $\varkappa=0$ of the curve $\Tilde{\nu}_n^2=(C_n^2\varkappa^2+1)a_n(\varkappa)$. Again the assumption that they vanish at these ramification points makes them unique. For sufficiently large $n$ the ramification points are simple roots of $q_{2,n}$, since $a_2\neq 0$. The functions \,$\varkappa \mapsto q_{1,n}/q_{2,n}$\, converge to \,$\varkappa \mapsto 2q/\nu$\, by the same arguments as in the first case. Now $\left.\tfrac{d}{d\varkappa}\right|_{\varkappa = 0} \, q_{1,n}/q_{2,n} = 0$ again yields \,$\left.\tfrac{d}{d\varkappa}\right|
_{\varkappa = 0} b/a' = 0$\,. 
\end{proof}
\noindent{\it Proof of Theorem~\ref{th:blow-up}.} For $C>0$, any $a\in\Hat{\calH}^2$ and $b\in\bbR^3[\varkappa]$ the biholomorphic map $(\varkappa,\nu)\mapsto(C\varkappa,C^2\nu)$ from the curve $\{(\varkappa,\nu)\mid\nu^2=a(\varkappa)\}$ onto the curve $\{(\varkappa,\nu)\mid\nu^2=C.a(\varkappa)\}$ pulls back the 1-form $b/\nu\,d\varkappa$ to the 1-form $\Tilde{b}/\nu \,d\varkappa$ with $\Tilde{b}(\varkappa)=C^{-1}b(C\varkappa)$. Hence $b\in\calB^\circ_{C.a}$ is equivalent to $\Tilde{b}\in\calB^\circ_a$. Now the action $\Psi$~\eqref{action} preserves $\calS^\circ_0$ and by continuity also $\overline{\calS^\circ_0}\setminus\{\varkappa^4\}$. This finishes the proof of~(i).

For the proof of~(ii) we note that $a\in\Hat{\calH}^2$ belongs to $\calS^\circ_0$ if and only if $a$ obeys two conditions:
\begin{enumerate}
\item[(a)] $a'(0)=0$;
\item[(b)] There exists a monic $b\in\bbR^2[\varkappa]$ so that the 1-form $b/\nu \,d\varkappa$ on the curve $\{ (\varkappa,\nu) \mid \nu^2=a(\varkappa) \}$ has zero residue at the poles over $\varkappa=\infty$, has purely real periods and vanishes at $\varkappa=0$.
\end{enumerate}
Condition~(a) is equivalent to $a_3=0$, while (b) is equivalent to the requirement that the following meromorphic 1-form has purely real periods on the curve $\{(\varkappa,\nu)\mid\nu^2=a(\varkappa)\}$:
\begin{gather}\label{blow up form 1}
b/\nu \, d\varkappa=\varkappa(\varkappa+\frac{a_1}{2})/\nu \,d\varkappa.
\end{gather}
Now we use the Whitham equation to deform such pairs $(a,b)$. For any polynomial $c\in\bbR^3[\varkappa]$ the following ODE on the coefficients of $a$ and $b$ describes flows which preserve all periods of~\eqref{blow up form 1}:
\begin{align}\label{blow up whitham}
2a\dot{b}-\dot{a}b&=2ac'-a'c.
\end{align}
For spectral curves without singularities $a$ does not vanish on $\varkappa\in\bbR$. In this case the ODE preserves the root of $b$ at $\varkappa=0$ if and only if the right hand side vanishes at $\varkappa=0$. Since $a'(0)$ vanishes, this is equivalent to $c'(0)=0$. With the choice $c(\varkappa)=c_3+c_1\varkappa^2$, equation~\eqref{blow up whitham} is equivalent to:
\begin{multline*}
2(\varkappa^4+a_1\varkappa^3+a_2\varkappa^2+a_4)\tfrac{1}{2}\dot{a}_1\varkappa-(\dot{a}_1\varkappa^3+\dot{a}_2\varkappa^2+\dot{a}_4)(\varkappa^2+\tfrac{1}{2}a_1\varkappa)\\=2(\varkappa^4+a_1\varkappa^3+a_2\varkappa^2+a_4)2\varkappa c_1-(4\varkappa^3+3a_1\varkappa^2+2a_2\varkappa)(c_1\varkappa^2+c_3).
\end{multline*}
Comparing coefficients of \,$\varkappa$\, in this equation gives:
\begin{align*}
\varkappa^5&: 0=0,&\varkappa^4&:\tfrac{1}{2}a_1\dot{a}_1-\dot{a}_2=a_1c_1,&\varkappa^3&:\dot{a}_1a_2-\tfrac{1}{2}a_1\dot{a}_2=2a_2c_1-4c_3,\\
\varkappa^2&:-\dot{a}_4=-3a_1c_3,&\varkappa^1&:\dot{a}_1a_4-\tfrac{1}{2}a_1\dot{a}_4=4a_4c_1-2a_2c_3,&\varkappa^0&:0=0.
\end{align*}
Inserting $\dot{a}_2=\tfrac{1}{2}a_1\dot{a}_1-a_1c_1$ and $\dot{a}_4=3a_1c_3$ in the other equations gives
\begin{align*}
\dot{a}_1\big(a_2-\tfrac{1}{4}a_1^2\big)&=2\big(a_2-\tfrac{1}{4}a_1^2\big)c_1-4c_3,&\dot{a}_1a_4&=4a_4c_1+\big(\tfrac{3}{2}a_1^2-2a_2\big)c_3.
\end{align*}
This implies
$$2\big(a_2-\tfrac{1}{4}a_1^2\big)a_4c_1=-\left(4a_4+\big(\tfrac{3}{2}a_1^2-2a_2\big)\big(a_2-\tfrac{1}{4}a_1^2\big)\right)c_3.$$
With the following choice for the coefficients $c_1$ and $c_3$, we obtain a smooth vector field:
\begin{gather}\begin{aligned}\label{eq:whitham local}
c_1&=\big(2a_2-\tfrac{3}{2}a_1^2\big)\big(a_2-\tfrac{1}{4}a_1^2\big)-4a_4,\qquad c_3=2\big(a_2-\tfrac{1}{4}a_1^2\big)a_4,\\
\dot{a}_1&=2\left(\big(2a_2-\tfrac{3}{2}a_1^2\big)\big(a_2-\tfrac{1}{4}a_1^2\big)-4a_4\right)-8a_4=\big(4a_2-3a_1^2\big)\big(a_2-\tfrac{1}{4}a_1^2\big)-16a_4,\\
\dot{a}_2&=a_1\left(\big(2a_2-\tfrac{3}{2}a_1^2\big)\big(a_2-\tfrac{1}{4}a_1^2\big)-8a_4-\big(2a_2-\tfrac{3}{2}a_1^2\big)\big(a_2-\tfrac{1}{4}a_1^2\big)+4a_4\right)=-4a_1a_4,\\
\dot{a}_4&=6a_1\big(a_2-\tfrac{1}{4}a_1^2\big)a_4.
\end{aligned}\end{gather}
Due to part~(i) the action $(a,b)\mapsto(C.a,\Tilde{b})$ with $\Tilde{b}(\varkappa)=Cb(\frac{\varkappa}{C})$ (compare Lemma~\ref{lemma 3}) preserves $\calS^\circ_0$. This group action is induced by the following solution of the Whitham equations~\eqref{blow up whitham}:
\begin{align}\label{eq:rescaling}
\dot{a}(\varkappa)&=\varkappa a'(\varkappa)-4a(\varkappa), &\dot{b}(\varkappa)&=\varkappa b'(\varkappa)-b(\varkappa), &c(\varkappa)&=\varkappa b(\varkappa).
\end{align}
The highest coefficient of $b$ is not preserved by~\eqref{eq:rescaling}, in contrast to~\eqref{eq:whitham local}.

Now we show four claims. For any $a\in\Hat{\calH}^2$ and $b\in\calB^\circ_a$ the pair $a^-(\varkappa):=a(-\varkappa)$ and $b^-(\varkappa):=b(-\varkappa)$ obeys $a^-\in\Hat{\calH}^2$ and $b^-\in\calB^\circ_{a^-}$. Since $(a,b)\mapsto(a^-,b^-)$ preserves the conditions~(a)--(b), this shows

\noindent{\it Claim~1: The involution $(a_1,a_2,a_3,a_4)\!\mapsto\!(-a_1,a_2,-a_3,a_4)$ preserves the set of coefficients of $a\!\in\!\calS^\circ_0$.}

\noindent{\it Claim~2: Any $a\in\overline{\calS^\circ_0}\setminus\{\varkappa^4\}$ with $a_1=0$ obeys $0<a_2$ and $\frac{1}{4}a_2^2<a_4$.}

Let $a\in\overline{\calS^\circ_0}$ with $a_1=0$. If $a\not\in\Hat{\calH}^2$, then Lemma~\ref{lemma 5} gives $a=\varkappa^4$. Otherweise Lemma~\ref{lemma 1} applies to $C.a\in\overline{\calS^\circ_0}$ for some $C>0$. The continuity of the periods of the 1-form~\eqref{blow up form 1} in dependence of $a\in\overline{\calS^\circ_0}$ gives $\beta=0$ and $\alpha\in(0,2)$ with $(a_2/\alpha)^2=a_4$. This proves Claim~2.

\noindent{\it Claim~3: For $a\in\overline{\calS^\circ_0}\setminus\{\varkappa^4\}$ we have $a_2\ne \frac{1}{4}a_1^2$. Furthermore, any path--connected component of $\calS^\circ_0$ contains an element $a\in\calS^\circ_0$ with $a_2>0$.}

Lemma~\ref{lemma 2} implies that if the coefficients of any $a\in\overline{\calS^\circ_0}$ obey $a_2=\frac{1}{4}a_1^2$ then they also obey $a_4=0$, which means $a\in\overline{\calS^\circ_0}\setminus\Hat{\calH}^2$. Now Lemma~\ref{lemma 5} yields $a_2=\frac{3}{4}a_1^2$ and $a(\varkappa)=\varkappa^4$. This shows the first statement of Claim~3. To prove the second statement we present a vector field along $\calS^\circ_0$ whose maximal integral curves each contain an $a\in\calS^\circ_0$ with $a_2>0$. The vector field is constructed so that it preserves \,$a_4$\,, and it is given by the sum of~\eqref{eq:whitham local} and $\frac{3}{2}a_1(a_2-\frac{1}{4}a_1^2)$ times~\eqref{eq:rescaling}:
\begin{align}\label{eq:whitham local 3}
\dot{a}_1&=(a_2-\tfrac{1}{4}a_1^2)(4a_2-\tfrac{9}{2}a_1^2)-16a_4,&\dot{a}_2&=-a_1(3a_2(a_2-\tfrac{1}{4}a_1^2)+4a_4),&\dot{a}_4&=0.
\end{align}
It suffices to consider $a\in\calS^\circ_0$ with $a_2\le0$. Such $a$ obey \,$0<a_4$\,. Due to Claim~2, the coefficient $a_1$ has no root along an integral curve of~\eqref{eq:whitham local 3} as long as \,$a_2 \leq 0$\,. Furthermore, for $a_2\le 0$ we have $\dot{a}_1\ge -16a_4$. Therefore a maximal integral curve starting with $a_2(0)\le0$ and $a_1(0)<0$ obeys $a_1(0)-16a_4t\le a_1(t)<0$, and $a_2(0)\le a_2(t)\le 0$  by $0<\dot{a}_2(t)$ for $0\leq t\le t_0:=\inf\{s>0\mid a_2(s)>0\}<\infty$. In particular, $0<a_2(t)$ for $t>t_0$. For $a_2(0)\le0$ and $a_1(0)>0$ a similar argument gives $-\infty<t_0 :=\sup\{s<0\mid a_2(s)>0\}\le 0$ with $a_1(t),a_2(t_0)>0$ for $t<t_0$. This proves Claim~3.

\noindent{\it Claim 4: Any $a\in\calS^\circ_0$ obeys $\frac{3}{4}a_1^2<a_2$.}

We consider a vector field along $\calS^\circ_0$ which preserves \,$a_2$\,. It is given by the sum of $-a_2$ times~\eqref{eq:whitham local} and $2a_1a_4$ times~\eqref{eq:rescaling}. Since $a_2>0$ in all applications, by Theorem~\ref{th:blow-up}~(i) we may set $a_2=1$.
\begin{align}\label{eq:whitham local 2}
\dot{a}_1&=2a_4(8-a_1^2)-(4-3a_1^2)(1-\tfrac{1}{4}a_1^2),&\dot{a}_2&=0,&\dot{a}_4&=a_1a_4\big(\tfrac{3}{2}a_1^2-6-8a_4\big).
\end{align}
The set of roots of $\dot{a}_1$ and $\dot{a}_4$ are $\graph(h_1)$ and $\{(a_1,a_4)\mid a_1a_4=0\}\cup\graph(h_4)$ with
\begin{align*}
h_1: a_1 \mapsto a_4 &=\frac{(3a_1^2-4)(a_1^2-4)}{8(8-a_1^2)},&h_4: a_1 \mapsto a_4&=\tfrac{3}{32}(a_1^2-4),
\end{align*}
respectively. We first focus on the stable manifold of the critical point $s_1$\, at \,$(a_1,a_4) = (2,0)$. On
$$
R_1:=\left\{(a_1,a_4)\mid a_1\in\big(2,\sqrt{8}),h_4(a_1)\!<\!a_4\!<\!h_1(a_1)\right\}\cup\left\{(a_1,a_4)\mid a_1\in\big[\sqrt{8},\infty\big), h_4(a_1)\!<\!a_4\right\}
$$
we have $\dot{a}_1<0$ with equality on $\graph(h_1)$ and $\dot{a}_4<0$ with equality on $\graph(h_4)$. Hence the vector field~\eqref{eq:whitham local 2} points inwards to \,$\overline{R_1}$\, both on \,$\graph(h_1)$\, and on \,$\graph(h_4)$\,, and the integral curves are trapped in $\overline{R_1}$. Furthermore, $\overline{R_1}$ belongs to the stable manifold of $s_1$. Next we consider
$$
R_2=\big\{(a_1,a_4)\mid a_1\in\big(2,\infty\big), 0<a_4<h_4(a_1)\big\}.
$$
In $R_2$ we have $\dot{a}_1<0$ and $0<\dot{a}_4$ with equality on $\{(a_1,a_4)\mid a_4=0\}\cup\graph(h_4)$. Hence after finite positive time, any integral curve in $R_2$ crosses $\graph(h_4)$ and enters $R_1$. In particular, $\overline{R_2}$ also belongs to the stable manifold of $s_1$. Finally we consider, see Figure~\ref{fig:plots}, 
$$
R_3=\big\{ (a_1,a_4) \mid a_1\in\big( 2/\sqrt{3}, 2 \big), 0<a_4\big\}\cup\big\{(a_1,a_4)\mid a_1\in\big[2,\sqrt{8}\big),h_1(a_1)<a_4\big\}.
$$
In $R_3$ we have $0<\dot{a}_1$ with equality on $\graph(h_1)$ and $\dot{a}_4<0$ with equality on $\{(a_1,a_4)\mid a_4=0\}$. Hence after finite positive time any integral curve in this region crosses $\graph(h_1)$ and enters $R_1$. Let \,$s_2$\, denote the critical point at \,$(a_1,a_4)= (\frac{2}{\sqrt{3}},0)$\,. Then $\overline{R_3}\setminus\{s_2\}$ and the union $\overline{R_1}\cup\overline{R_2}\cup\overline{R_3}\setminus\{s_2\}$ also belong to the stable manifold of $s_1$.
\begin{figure}[t]
  \centering
  \includegraphics[scale=0.6]{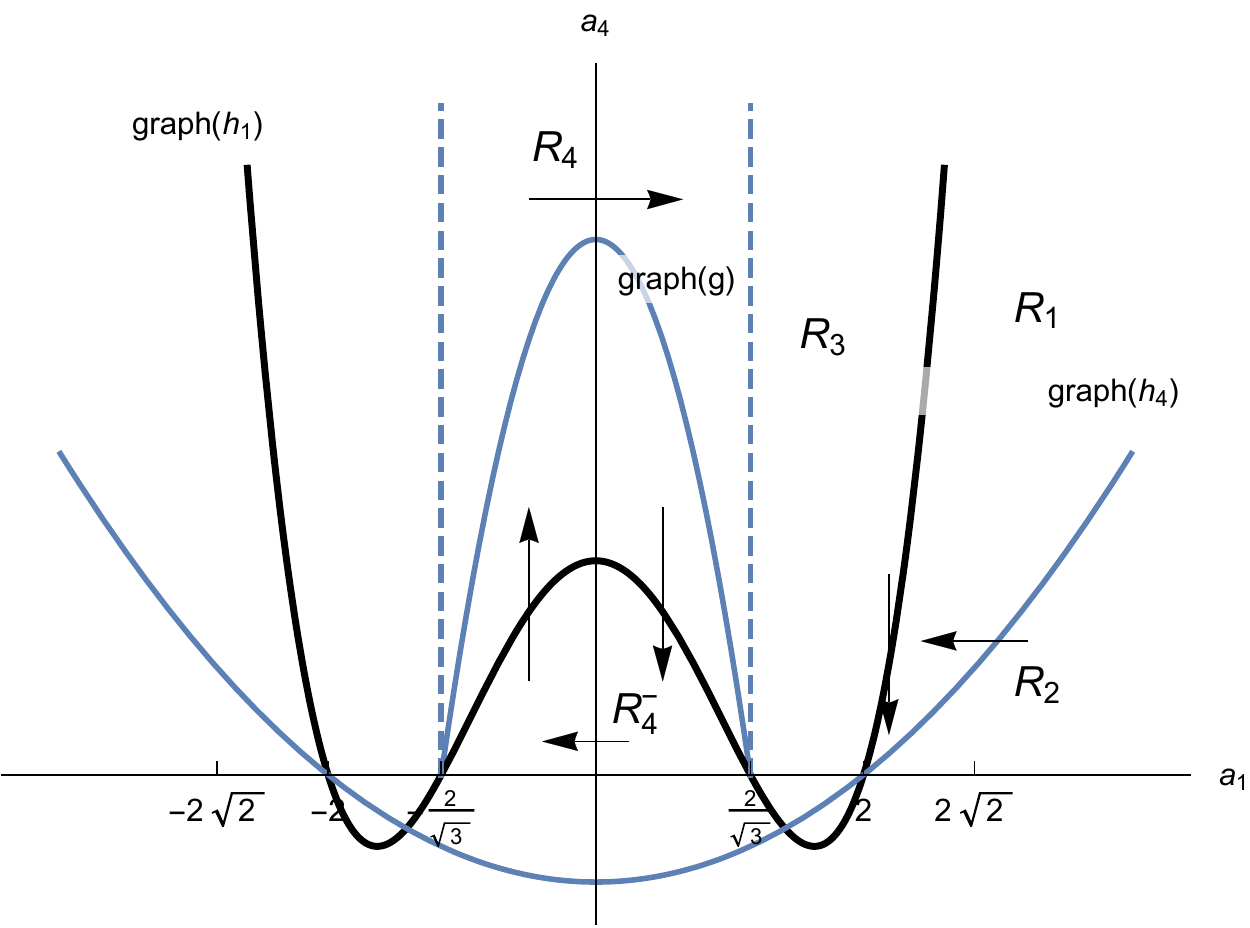}
  \caption{Sets and graphs in the proof of Theorem~\ref{th:blow-up}. The arrows indicate the direction of the Whitham flow.}
    \label{fig:plots}
\end{figure}
By Lemma~\ref{lemma 2} and~\ref{lemma 5} the coefficients of $a\in\overline{\calS^\circ_0}\setminus\{\varkappa^4\}$ obey $\frac{1}{4}a_1^2\ne a_2$. Now we prove $\frac{1}{4}a_1^2<a_2$ by contradiction. Suppose the coeffcients of $a\in\calS^\circ_0$ obey $a_2<\frac{1}{4}a_1^2$ and $a_2\le0$.
The proof of Claim~3 presents a vector field whose integral curve with such an initial point flows to $a\in\calS^\circ_0$ with $0<a_2<\frac{1}{4}a_1^2$. Hence the existence of $a\in\calS^\circ_0$ with $a_2<\frac{1}{4}a_1^2$ implies the existence of such an $a\in\calS^\circ_0$ with $0<a_1<\frac{1}{4}a_1^2$. By Theorem~\ref{th:blow-up}~(i), $\Tilde{a}= (1/\sqrt{a_2}).a$ belongs to $\calS^\circ_0$ with coefficients $\Tilde{a}_2=1<\frac{1}{4}\Tilde{a}_1^2$. By Claim~1 we may in addition assume $\Tilde{a}_1>0$. Then $(\Tilde{a}_1,\Tilde{a}_4)\in\overline{R_1}\cup\overline{R_2}\cup\overline{R_3}\setminus\{s_2\}$ belongs to the stable manifold of $s_1$. Since the flow stays in $\calS^\circ_0$, the limit $(a_1,a_2,a_3,a_4) \to (2,1,0,0)$ belongs to $\overline{\calS^\circ_0}$. This contradicts Lemma~\ref{lemma 5} and shows $\frac{1}{4}a_1^2<a_2$ for any $a\in\calS^\circ_0$. Furthermore, if $a\in\calS^\circ_0$ obeys $\frac{1}{4}a_1^2<a_2\le\frac{3}{4}a_1^2$, then the coefficients $(\Tilde{a}_1,\Tilde{a}_4)$ or $(-\Tilde{a}_1,\Tilde{a}_4)$ of $\Tilde{a}=(1/\sqrt{a_2}).a$ belong to \,$\overline{R_3} \setminus\{s_2\}$\,, which leads to the same contradiction. This proves Claim~4.

To finish the proof of Theorem~\ref{th:blow-up}~(ii) we consider the region
$$
R_4:=\big({-} 2/\sqrt{3},2/\sqrt{3}\big) \times (0,\infty).
$$
The derivative $\dot{a}_1$ is negative for $a_4<h_1(a_1)$ and positive for $h(a_1)<a_4$, and the derivative $\dot{a}_4$ is positive for $a_1<0$ and negative for $a_1>0$. The line segment $(-2/\sqrt{3}, 2/\sqrt{3})\times \{0\}$ is an integral curve and equal to the intersection of the unstable manifold of $s_2$ with the stable manifold of $-s_2$. By Lemma~\ref{lemma 5} only its end points might belong to $\overline{\calS^\circ_0}$. At $\pm s_2$ the Jacobi matrix of~\eqref{eq:whitham local 2} is
\begin{align*}
\begin{pmatrix}
a_1(8-3a_1^2-4a_4)&16-2a_1^2\\a_4(\tfrac{9}{2}a_1^2-6-8a_4)&a_1(\tfrac{3}{2}a_1^2-6-16a_4)\end{pmatrix}&=\begin{pmatrix}\pm\tfrac{8}{\sqrt{3}}&-\tfrac{40}{3}\\0&\mp\tfrac{8}{\sqrt{3}}\end{pmatrix}.
\end{align*}
Therefore these singularities are hyperbolic and only one eigenvalue has a non--zero $a_4$--component. By the Hartman--Grobman linearisation theorem these singularities have one--dimen\-sional stable and unstable manifolds. At $-s_2$ the stable manifold and at $s_2$ the unstable manifold are completely contained in the invariant line $(a_1,a_4)\in\bbR\times\{0\}$. Since $\dot{a}_1$ is negative on
$$
R_4^-:=\big\{(a_1,a_4)\in R_4\mid 0<a_4<h_1(a_1)\big\},
$$
and since $\dot{a}_4$ is positive on $\{(a_1,h_1(a))\mid a_1\in({-}2/\sqrt{3},0)\}$ and negative on $\{(a_1,h_1(a))\mid a_1\in(0,2/\sqrt{3})\}$, any integral curve in $R_4^-$ enters $R_4^-$ at some point of $\{(a_1,h_1(a))\mid a_1\in(0,2/\sqrt{3})\}$, crosses the line $\{(0,a_4)\mid a_4\in (0,1/4)\}$ in a unique point, and leaves $R_4^-$ at some point of $\{(a_1,h_1(a))\mid a_1\in({-}2/\sqrt{3},0)\}$. Lemma~\ref{lemma 1} implies that any $(a_1,a_4)\in\overline{R_4^-}\setminus\{\pm s_2\}$ together with $a_2=1$ cannot be the coefficients of any \,$a \in \overline{\calS^\circ_0}$. This argument also applies to the root $(a_1,a_4)=(0,\frac{1}{4})$ of~\eqref{eq:whitham local 2}.

Any integral curve in $\big\{(a_1,a_4)\in R_4\mid 0<a_1,h_1(a_1)<a_4\big\}$ obeys $\dot{a}_4<0<\dot{a}_1$. Therefore either
\begin{enumerate}
\item after finite positive time it crosses the boundary in $\graph(h_1)$ and enters $R_4^-$, or
\item after finite positive time it crosses the boundary in \,$\{2/\sqrt{3}\} \times (0,\infty)$\,  and enters $R_3$, or
\item it is entirely contained in the stable manifold of $s_2$.
\end{enumerate}
In the first two cases the pairs \,$(a_1, a_4)$\, of the integral curve together with \,$a_2=1$\, cannot be the coefficients of any $a \in \calS^\circ_0$. If $\alpha_0\in(0,2)$ denotes the unique root of $\alpha\mapsto \beta$ in Lemma~\ref{lemma 1}, then \,$s_3$\, with  $(a_1,a_4)=(0,\alpha_0^{-2})$ is the unique element of $\{0\}\times(0,\infty)\subset R_4$ whose entries together with $a_2=1$ are the coefficients of some $a\in\calS^\circ_0$. Therefore the integral curve through $s_3$ is the intersection of $R_4$ with the stable manifold of $s_2$. By the $(a_1,a_4)\mapsto(-a_1,a_4)$ symmetry this integral curve is also the intersection of $R_4$ with the unstable manifold of $-s_2$. Furthermore, it is the graph of a unique analytic function $g: ({-}2/\sqrt{3},2/\sqrt{3}) \to (0,\infty)$ obeying~\eqref{def g}. Hence $\graph(g)$ is the set of those $(a_1,a_4)\in R_4$ whose entries together with $a_2=1$ are the coefficients of some $a\in\calS^\circ_0$. By Claim~4 we conclude that the set $a\in\calS^\circ_0$ is characterised by $a_2>0$ and the condition that the coefficients $(\Tilde{a}_1,\Tilde{a}_4)$ of $\Tilde{a}=(1/\sqrt{a_2}).a\in\calS^\circ_0$ belong to $\graph(g)$. This finishes the proof of Theorem~\ref{th:blow-up}~(ii).

First we observe $\Hat{\calE}^2_0\subset\calS^\circ_0\cup\{\varkappa^4+a_1\varkappa^3+\frac{3}{4}a_1^2\mid a_1\in\bbR\}$ from Lemmata~\ref{lemma 3} and~\ref{lemma 5}. The last property of \eqref{def g} in Theorem~\ref{th:blow-up}~(ii) gives $\{\varkappa^4+a_1\varkappa^3+\frac{3}{4}a_1^2\mid a_1\in\bbR\}\subset\overline{\calS^\circ_0}$. This implies $\Hat{\calE}^2_0\subset\overline{\calS^\circ_0}$.

To prove $\big(\calS^\circ_0\cap\bbS^3[\varkappa]\big)\subset\Hat{\calE}^2_0$ we transfer the arguments of Lemma~\ref{L:boundary-S21-C4lambda-pre} to the present situation. In a neighbourhood of $\{0\}\times\big(\Hat{\calH}^2\cap\bbS^3[\varkappa]\big)$ the closure of $\Psi^{-1}[\Hat{\calH}^2]$ in $[0,\infty)\times\bbR^4[\varkappa]$ is a manifold with boundary, whose boundary is equal to $\{0\}\times\big(\Hat{\calH}^2\cap\bbS^3[\varkappa]\big)$. Given $a\in\calS^\circ_0\cap\bbS^3[\varkappa]$, a neighbourhood of $(0,a)$ in $[0,\infty)\times\bbS^3[\varkappa]$ contains an open neighbourhood $O$ such that $O\cap\Psi^{-1}[\Hat{\calH}^2]$ is connected and is neither disjoint from $\Psi^{-1}[\Hat{V}_{-1}]$ nor from $\Psi^{-1}[\Hat{V}_1]$, cf.~Figure~\ref{figure:delH2} in the proof of Lemma~\ref{L:boundary-S21-C4lambda-pre}. Hence $O\cap\Psi^{-1}[\Hat{\calS}^2]\ne\emptyset$. Therefore $(0,a)$ is the limit of a sequence $(C_n,a_n)_{n\in\mathbb{N}}$ in $(0,\infty)\times\bbS^3[\varkappa]$ with $C_n.a_n\in\Hat{\calS}^2$. Due to Lemma~\ref{lemma 3} the corresponding basis $(b,a')$ of $\calB_a^\circ$ is the limit of a sequence of pairs of polynomials, such that each pair has a unique common root $\varkappa_n$. The roots \,$\varkappa_n$\, converge to the unique common root $\varkappa=0$ of $(b,a')$. Note that $C_n.a_n\in\Hat{\calH}^2$ implies $a_n\in\Hat{\calH}^2$. Let $(\alpha_{1,n},\alpha_{2,n})_{n\in\mathbb{N}}$ denote the roots of $a_n$ in the upper half plane. We may label them in such a way that they converge to roots of $a$. For any $n\in\mathbb{N}$, let $g_n$ denote the real M\"obius transformation which fixes the values $\varkappa=\pm C_n^{-1}\mi$ and maps $\varkappa_n$ to $0$. The sequence $(g_n)_{n\in\mathbb{N}}$ converges uniformly to the identity map on any compact subset of $\varkappa\in\bbC$. Therefore the following sequence $(\Tilde{a}_n)_{n\in\mathbb{N}}$converges to $a$:
$$\Tilde{a}_n(\varkappa)=(\varkappa-g_n(\alpha_{1,n}))(\varkappa-g_n(\Bar{\alpha}_{1,n}))(\varkappa-g_n(\alpha_{2,n}))(\varkappa-g_n(\Bar{\alpha}_{2,n})).$$
By the definition of $\varkappa_n$ and $g_n$, we have $C_n.\Tilde{a}_n \in \Hat{\calS}^2_0$. Hence $a\in\Hat{\calE}^2_0$, $\big(\calS^\circ_0\cap\bbS^3[\varkappa]\big)\subset\Hat{\calE}^2_0$, and finally $\big(\overline{\calS^\circ_0}\cap\bbS^3[\varkappa]\big)\subset\Hat{\calE}^2_0$ since $\Hat{\calE}^2_0$ is closed in $\bbS^3[\varkappa]$. This concludes the proof of Theorem~\ref{th:blow-up}~(iii).\qed
%
%%%%%%%%%%%%%%%
\section{The Wente family}\label{Se:wente}
We next study the \emph{Wente family} $\mathcal{W}:=\calS^2_1\cap\bbR^4[\lambda]$. A cmc torus in $\R^3$ of spectral genus $2$ is called a \emph{Wente torus} \cite{Wen, Ab} if the corresponding polynomial $a\in \calS^2_1$ is a member of $\mathcal{W}$. We call $\calH^2\cap\bbR^4[\lambda]$ the \emph{Abresch family}, since Abresch \cite{Ab} was the first to construct the corresponding solutions of the $\sinh$--Gordon equation. It has two connected components which we introduce now:
\begin{definition}
\label{D:wente:A+-A-}
Let $\calA_+$ denote the connected component of $\calH^2\cap\bbR^4[\lambda]$ whose elements have four different real roots, and $\calA_-$ the connected component of $\calH^2\cap\bbR^4[\lambda]$ whose elements have four different roots of the form $\alpha,\Bar{\alpha},\alpha^{-1}$ and $\Bar{\alpha}^{-1}$ with \,$\alpha \in \bbC \setminus (\bbR \cup \bbS^1)$\,. 
\end{definition}
  
\begin{theorem}
\label{T:wente:wente}
Let $\alpha_0\in(0,2)$ be the unique root of the strictly increasing function $\alpha\mapsto\beta$ in Lemma~\ref{lemma 1}. Then $\calW$ is the following connected non--compact 1--dimensional submanifold of $\calA_-$:
\begin{equation}
\label{eq:wente:wente:W}
\mathcal{W} = \; \bigr\{ \lambda^4-(4+\alpha_0\alpha_+)\lambda^3+(6+2\alpha_0\alpha_++\alpha_+^2)\lambda^2-(4+\alpha_0\alpha_+)\lambda+1 \,\bigr|\,\alpha_+\in\bbR_+\bigr\}
\end{equation}
% The polynomails $a\in\calW$ have no roots in $\bbR\cup\mi\bbR\cup\bbS^1$
%
%maximal integral curve of a suitable Whitham flow on $\calS^2_1$, in particular it is a connected, 1--dimensional submanifold of $\calS^2_1$. $\Phi$ is a diffeomorphism that maps $\mathcal{W}$ bijectively onto an open interval of length $\pi$. 
%
%$\mathcal{W}$ is given explicitly as in terms of a transcendental constant $4\alpha_0$ with $-8 < 4\alpha_0 <8$. The polynomials $a\in\mathcal{W}$ have no zeros on $\bbR \cup \bbS^1$.
\end{theorem}
The objective of the remainder of the section is the proof of this theorem. The proof consists of several lemmata. First we show that a certain type of  Whitham flow stays in $\mathcal{W}$. 
\begin{lemma}\label{L:wente:whitham-wente}
In the setting of Lemma~\ref{L:para:whitham} suppose that $a\in\mathcal{W}$.
\begin{itemize}
\item[(a)] $a$ can be written in the form
\begin{equation} \label{eq:wente:whitham-wente:a}
	a(\lambda) = \lambda^4 + \tfrac14(a_+-a_-)\lambda^3+\tfrac12(a_++a_--4)\lambda^2 + \tfrac14(a_+-a_-)\lambda+1 \;,
\end{equation}
where $a_\pm = a(\pm 1) >0$. Moreover, there exists a basis $(b_1,b_2)$ of $\calB_a$ so that
\begin{equation}\label{eq:para:whitham-real:b-real}
\overline{b_1(\bar{\lambda})} = b_1(\lambda) \quad\text{and}\quad \overline{b_2(\bar{\lambda})} = -b_2(\lambda)
\end{equation}
holds, and then there exist $\beta_1,\beta_2,\beta_3 \in \bbR$ with
\begin{equation}\label{eq:wente:whitham-wente:b}
b_1(\lambda) = \beta_1\,(\lambda-1)^2\,(\lambda+1) \quad\text{and}\quad b_2(\lambda)=\mi\,\beta_2\,(\lambda-1)\,(\lambda^2+(\beta_3+2)\lambda+1) \;.
\end{equation}
\item[(b)]
The Whitham vector field $(\dot{a},\dot{b}_1,\dot{b}_2,c_1,c_2,Q)$ with $c_1(1)=0$ and $c_2(1) \in \bbR$ then satisfies 
\begin{gather}\label{eq:para:whitham-real:dot-real}
\begin{aligned}
\overline{\dot{a}(\bar{\lambda})} & = \dot{a}(\lambda)\;, &  
\overline{\dot{b}_1(\bar{\lambda})} & = \dot{b}_1(\lambda)\;, &  \overline{\dot{b}_2(\bar{\lambda})} & = -\dot{b}_2(\lambda)\;, & \\
\overline{c_1(\bar{\lambda})} & = -c_1(\lambda)\;, &  \overline{c_2(\bar{\lambda})} & = c_2(\lambda)\;, &  \overline{Q(\bar{\lambda})} & = Q(\lambda) \; . 
\end{aligned}
\end{gather}
Such a Whitham flow is tangential to $\mathcal{W}$. 

\item[(c)] We write
\begin{gather}\label{eq:wente:whitham-wente:dotabk}
\begin{aligned}
\dot{a}(\lambda)& =\tfrac14(\dot{a}_+-\dot{a}_-)\lambda^3 + \tfrac12(\dot{a}_++\dot{a}_-)\lambda^2 + \tfrac14(\dot{a}_+-\dot{a}_-)\lambda\;, \\ \dot{b}_1(\lambda) & = \dot{\beta}_1\,(\lambda-1)^2\,(\lambda+1) \quad\text{and} \\ \dot{b}_2(\lambda) & = \mi\,\dot{\beta}_2\,(\lambda-1)\,(\lambda^2+(\beta_3+2)\lambda+1) + \mi\,\beta_2\,\dot{\beta}_3\,(\lambda-1)\,\lambda 
\end{aligned}
\end{gather}
with $\dot{a}_\pm,\dot{\beta}_\ind \in \bbR$. For $c_1(1)=0$ and $c_2(1)=-4\,a_+\,\beta_2\,\beta_3 \in \bbR$ we then have
\begin{align}\label{eq:wente:whitham-wente:dgl-apm}
\dot{a}_+ & = 4\,a_+\,a_-,&
%\label{eq:wente:whitham-wente:dgl-am}
\dot{a}_- & = 2\,(a_++a_--16)\,a_-\\
\label{eq:wente:whitham-wente:dgl-beta1}
\dot{\beta}_1 & = -a_-\,\beta_1,\\
\label{eq:wente:whitham-wente:dgl-beta23}
\dot{\beta}_2 & = -(a_-+4\beta_3)\,\beta_2,&
%\label{eq:wente:whitham-wente:dgl-beta3}
\dot{\beta}_3 & = 4\,\beta_3^2+2\,a_-\,\beta_3+4\,a_- \; .
\end{align}
\end{itemize}
\end{lemma}
\begin{proof}
For (a) we note that because of $a\in\mathcal{W}$, all coefficients of $a$ are real. Thus the reality condition and the normalisation for $a\in\mathcal{H}^2$ imply that the highest and the lowest coefficient of $a$ are equal to one, whereas the $\lambda^3$--coefficient and the $\lambda$--coefficient are equal. This implies equation~\eqref{eq:wente:whitham-wente:a}. The statement $a(\pm 1) > 0$ follows from the positivity condition for $a\in \mathcal{H}^2$. 
		
By definition $\calB_a$ is preserved under $b(\lambda) \mapsto \overline{b(\bar{\lambda})}$. Since $b(0)$ uniquely determines $b\in\calB_a$, we conclude
$$b(0)\in\bbR \;\Longleftrightarrow\; b\in\bbR[\lambda]\quad\text{and}\quad  b(0) \in \mi\bbR \;\Longleftrightarrow\; b\in\mi\bbR[\lambda]\; . $$
This shows the existence of $b_1$ and $b_2$. Then $b_1$ vanishes at $\lambda=1$ by the definition of $\calS^2_1$ and at $\lambda=-1$ because it is both real and purely imaginary there. Therefore $b_1$ is of the form $b_1(\lambda) = (\lambda-1)\,(\lambda+1)\,p(\lambda)$ with $p \in \mi P^1_\bbR$ and $\overline{p(\bar{\lambda})} = p(\lambda)$. Thus $p(\lambda) = \beta_1\,(\lambda-1)$ with $\beta_1 \in \bbR$, and hence $b_1$ is as in equation~\eqref{eq:wente:whitham-wente:b}. Further $b_2$ is also zero at $\lambda=1$, and therefore of the form $b_2(\lambda)=\mi\,(\lambda-1)\,p(\lambda)$ with $p \in P^2_\bbR\cap\bbR^2[\lambda]$. Now $p(\lambda)=\beta_2\,(\lambda^2+(\beta_3+2)\lambda+1)$ with $\beta_2,\beta_3\in\bbR$, which shows \eqref{eq:wente:whitham-wente:b} for $b_2$.

For (b), if the $b_\ind$ satisfy the conditions of equation~\eqref{eq:para:whitham-real:b-real}, then some anti--derivative of $\Theta(b_1)$ is real on the real line, and some anti--derivative of $\Theta(b_2)$ is purely imaginary on the real line. If we want to preserve this property under the Whitham flow, then the equations for $c_\ind$ in \eqref{eq:para:whitham-real:dot-real} must hold. In this case we have $c_1(1) \in \bbR\cap \mi\bbR= \{0\}$ and $c_2(1) \in \bbR$. If this is the case, then $Q'(1), Q''(1) \in \bbR$ by equations~\eqref{eq:para:whitham:Q'} and \eqref{eq:para:whitham:Q''}. Therefore $Q$ satisfies \eqref{eq:para:whitham-real:dot-real}. Because of the uniqueness of solutions of Lemma~\ref{L:para:whitham}, the corresponding solutions satisfy \eqref{eq:para:whitham-real:dot-real}. 
	
In (c), the representations of $\dot{a}$ and $\dot{b}_\ind$ in \eqref{eq:wente:whitham-wente:dotabk} follow immediately from (a), because the flow in question is tangential to $\mathcal{W}$. We have $c_1(0)=0$ and therefore also $c_1'(0)=0$ by equation~\eqref{eq:para:whitham:ck'}. It follows that $c_1(\lambda)=\mi\,(\lambda-1)^2\,p(\lambda)$ holds, where $p\in \mi P^1_\bbR$ satisfies $\overline{p(\bar{\lambda})}=p(\lambda)$. Thus $\lambda-1$ divides $p$, and hence $c_1(\lambda)=\mi\,\gamma_1\,(\lambda-1)^3$ holds with some $\gamma_1 \in \bbR$. Also $c_2$ vanishes at $\lambda=-1$ because it is both real and purely imaginary there, and thus  of the form $c_2(\lambda) = \gamma_2\,(\lambda+1)\,(\lambda^2+(\gamma_3+2)\lambda+1)$ with some $\gamma_2,\gamma_3\in \bbR$. By equations~\eqref{eq:para:whitham:Q'} and \eqref{eq:para:whitham:Q''} we now obtain $Q(1)=Q'(1)=0$ and 
$$ Q''(1) = -\frac{b_1''(1)\,c_2(1)}{a_+} = -\frac{4\beta_1 \, 2\gamma_2(\gamma_3+4)}{a_+} $$
and hence
$$ Q(\lambda)=-\frac{4\,\beta_1\,\gamma_2\,(\gamma_3+4)}{a_+}\,(\lambda-1)^2 \; . $$
By inserting the expansions of the polynomials $b_\ind,\,c_\ind,\,Q$ and $a$ into equation~\eqref{eq:para:whitham:2} and comparing coefficients of $\lambda$ we obtain a system of linear equations in $\gamma_\ind$. Under the additional condition $c_2(1)=-4\,a_+\,\beta_2\,\beta_3$, it has the unique solution
$$ \gamma_1 = 2\,a_-\,\beta_1\;,\quad \gamma_2 = -2\,\beta_2\,(a_-+4\beta_3)\quad\text{and}\quad \gamma_3 =  \frac{a_+\,\beta_3}{a_-+4\beta_3}-4 \; . $$
Inserting the representations of the various polynomials and the equations for $\gamma_\ind$ into equation~\eqref{eq:para:whitham:1} for $\ind=1$ and collecting like powers of $\lambda$ yields another system of linear equations in $\dot{a}_+$, $\dot{a}_-$ and $\dot{\beta}_1$, which has the unique solution given by equations~\eqref{eq:wente:whitham-wente:dgl-apm} and \eqref{eq:wente:whitham-wente:dgl-beta1}. Treating equation~\eqref{eq:para:whitham:1} for $\ind=2$ in the same way, one obtains a system of linear equations in $\dot{a}_+$, $\dot{a}_-$, $\dot{\beta}_2$ and $\dot{\beta}_3$ with the unique solution given by equations~\eqref{eq:wente:whitham-wente:dgl-apm} and \eqref{eq:wente:whitham-wente:dgl-beta23}.
\end{proof}
Wente proved \,$\calW \neq \varnothing$\, in \cite{Wen}. In Corollary~\ref{wente existence} we provide an independent proof for this. 
\begin{lemma}\label{L:wente:submfd}
The Wente family $\mathcal{W}$ is a smooth, 1--dimensional submanifold of $\calS^2_1$. The connected components of $\mathcal{W}$ are images of maximal integral curves of the Whitham flow described in Lemma~\ref{L:wente:whitham-wente}. On every such integral curve, $a_+$ is strictly increasing, and we have $a_+\to 0$ at the lower boundary and $a_+\to\infty$ at the upper boundary of the curve. %The map $\Phi:\mathcal{W}\to\bbR$ (equation~\eqref{eq:wente:Phi}) is a local diffeomorphism that maps each connected component of $\mathcal{W}$ bijectively onto an open interval of length $\pi$. 
\end{lemma}
\begin{proof}
Recall that $\calS^2_1$ is a smooth, real 2--dimensional manifold by Lemma~\ref{L:S2-smooth}. For $a\in \calS^2_1$ and a basis $(b_1,b_2)$ of $\calB_a$ we consider the Whitham flows constructed in Lemma~\ref{L:para:whitham}. For $(c_1(1),c_2(1))\in\bbR^2$, $\dot{a}=0$ first implies that $(c_1,c_2)$ vanishes at all roots of $a$, and then $(c_1,c_2)=(0,0)$. This shows that the linear map $g: \bbR^2 \to T_a\calS^2_1$, which associates to $(c_1(1),c_2(1))\in \bbR^2$ the $\dot{a}$ defined in Lemma~\ref{L:para:whitham}, is injective. Because of $\dim \calS^2_1=2$, $g$ is in fact a linear isomorphism. By the inverse function theorem, the flow of these vector fields defines a map $\widetilde{g}$ of a neighbourhood of $(0,0)\in\bbR^2$ into $\calS^2_1$ so that $\widetilde{g}(0,0)=a$ and  $\mathrm{d}_{(0,0)}\widetilde{g}=g$ holds. This map is a submersion.

Now suppose $a\in \mathcal{W}$ and let the basis $(b_1,b_2)$ of $\calB_a$ be chosen as in Lemma~\ref{L:wente:whitham-wente}(a). Then $g(1,0)$ is transversal to $\mathcal{W}$, whereas $g(0,1)$ is tangential to $\mathcal{W}$. This shows that $\mathcal{W}$ is near $a$ a 1--dimensional submanifold of $\calS^2_1$. 
It also follows that the connected components of $\mathcal{W}$ are images of maximal integral curves of the Whitham flow defined by $c_1(1)=0$, $c_2(1)=1$. 

On such an integral curve we have $a_\pm = a(\pm 1)>0$ because of $a \in \calS^2_1$, and therefore the first equation in~\eqref{eq:wente:whitham-wente:dgl-apm} shows that $a_+$ is strictly increasing. At its boundary points, the integral curve either converges to a boundary point of $\calS^2_1$, which implies that $a_+ \to 0$ by Proposition~\ref{P:boundary-S21-C4lambda}, or else $a_+ \to \infty$. Because $a_+$ is strictly increasing, $a_+\to 0$ occurs at the lower boundary and $a_+\to\infty$ occurs at the upper boundary of the curve. 
\end{proof}
In the sequel we consider the vector field defined by equations~\eqref{eq:wente:whitham-wente:dgl-apm}--\eqref{eq:wente:whitham-wente:dgl-beta23} also as a vector field on $(a_+,a_-,\beta_1,\beta_2,\beta_3) \in \bbR^5$. It is a remarkable fact that the differential equations for $a_+$ and $a_-$ do not depend on the $\beta_\ind$ in our situation, and thus split off to give a vector field on $(a_+,a_-)\in \bbR^2$ defined by equations~\eqref{eq:wente:whitham-wente:dgl-apm}. Similarly the differential equations for $(a_+,a_-,\beta_2,\beta_3)$ do not depend on $\beta_1$, and thus define a vector field on $(a_+,a_-,\beta_2,\beta_3)\in \bbR^4$. Any smooth integral curve of the latter differential equation can be supplemented by $\beta_1=\exp\bigr( -\int a_-(s)\,\mathrm{d}s \bigr)$ to produce an integral curve of the full system of differential equations \eqref{eq:wente:whitham-wente:dgl-apm}--\eqref{eq:wente:whitham-wente:dgl-beta23} with the same domain of definition as before. 

Lemma~\ref{L:wente:submfd} shows that the connected components of $\mathcal{W}$ correspond to maximal integral curves of the vector field on $(a_+,a_-)\in\bbR^2$ given by equations~\eqref{eq:wente:whitham-wente:dgl-apm} with $a_+>0$. However, not all integral curves of that vector field flow along $\mathcal{W}$; in fact we will see that only a single integral curve has this property. The reason is that whereas the periods of the differential $\Theta(b_1)$ with $b_1$ given by equation~\eqref{eq:wente:whitham-wente:b} on the spectral curve defined by $\nu^2=\lambda\,a(\lambda)$ with $a$ given by equation~\eqref{eq:wente:whitham-wente:a} are constant along integral curves, there is no reason in general why these periods should be purely imaginary. Hence the corresponding $a$ will generally not be in $\calS^2_1$ and in particular not in $\mathcal{W}$. 
\begin{lemma}
\label{L:wente:delta}
For every maximal integral curve of the vector field on $(a_+,a_-)$ defined by equations~\eqref{eq:wente:whitham-wente:dgl-apm} with $a_+ > 0$ there exists $\alpha \in \bbR$ so that $a_- = a_+ + 4\alpha\,\sqrt{a_+} + 16$ holds.  
\end{lemma}

\begin{proof}
We consider a maximal integral curve, and regard $y = a_--16$ as a function of the strictly increasing variable $a_+$. Then it follows from equations~\eqref{eq:wente:whitham-wente:dgl-apm} that $y$ satisfies the linear ODE
\begin{equation}\label{eq:wente:delta:ode}
\frac{\mathrm{d}y}{\mathrm{d}a_+} = \frac{\dot{a}_-}{\dot{a}_+} = \frac{1}{a_+}y + \frac{1}{2} \; .
\end{equation}
The general solution of $\tfrac{\mathrm{d}y_h}{\mathrm{d}a_+} = \tfrac{1}{a_+}y_h$ is given by $y_h(x)=C\,\sqrt{a_+}$ with a constant $C$. By variation of parameters it follows that the general solution of \eqref{eq:wente:delta:ode} is $y(a_+) = f(a_+)\,\sqrt{a_+}$, where $f'(a_+)=\tfrac{1}{2\sqrt{a_+}}$. We thus have $f(a_+)=\sqrt{a_+}+4\alpha$ with a constant $\alpha \in \bbR$ and hence $y(a_+)=a_++4\alpha\,\sqrt{a_+}$. 
\end{proof}
%\begin{proof}
%We consider a maximal integral curve, where we choose $\eta=1/a_-$ in  equations~\eqref{eq:wente:whitham-wente:dgl-apm}.
%Then equation~\eqref{eq:wente:whitham-wente:dgl-ap} becomes $\dot{a}_+ = 4\,a_+$, which shows that $a_+(t)=C_1\,e^{4t}$ holds with some $C_1>0$, where $t$ is the parameter of the integral curve. equation~\eqref{eq:wente:whitham-wente:dgl-am} thus takes the form
%\begin{equation*}
%%\label{eq:wente:delta:dgl-am}
%\dot{a}_-(t) = 2(a_+(t)+a_-(t)-16) = 2\,a_-(t)+2\,C_1\,e^{4t}-32 \; .
%\end{equation*}
%%The general solution of the homogeneous part $\dot{y} = 2y$ of this linear differential equation is $y(t)=C_2\,e^{2t}$ with $C_2\in \bbR$. 
%By the principle of variation of the constant, the general solution of this inhomogeneous linear differential equation is given by $a_-(t)=f(t)\,e^{2t}$, where the function $f$ solves the differential equation $\dot{f}\,e^{2t} = 2\,C_1\,e^{4t}-32$, that is ~$f(t)=C_1\,e^{2t} + 16e^{-2t}+C_2$ with a constant $C_2 \in \bbR$. We thus obtain
%$$ a_-(t)= C_1\,e^{4t} + 16 + C_2\,e^{2t} = a_+(t) + 4\alpha\,\sqrt{a_+(t)} + 16 \quad\text{with}\quad 4\alpha = \frac{C_2}{\sqrt{C_1}} \; . $$
%\end{proof}
\begin{lemma}
\label{L:wente:flowout}
Every maximal integral curve of the vector field on $(a_+,a_-)$ given by equations~\eqref{eq:wente:whitham-wente:dgl-apm} with $a_+\downarrow0$ at the lower boundary extends to a unique maximal integral curve of the vector field on $(a_+,a_-,\beta_1,\beta_2,\beta_3)$ given by equations~\eqref{eq:wente:whitham-wente:dgl-apm}--\eqref{eq:wente:whitham-wente:dgl-beta23} with the same lower boundary (but a possibly smaller upper boundary), such that $\beta_1\, a_+^{1/4} \to 1$, $\beta_2 \to 1$ and $\beta_3\to-4$ at the lower boundary. For this integral curve, $\Theta(b_2)$ with $b_2$ defined by equation~\eqref{eq:wente:whitham-wente:b} has purely imaginary periods. 
\end{lemma}
\begin{remark}\label{convergency of 1-forms}
It is to be expected that $\beta_1$ is unbounded, but $\beta_2$, $\beta_3$ are bounded in the situation of the lemma, cf.\ \cite[the proof of Theorem~3.5]{KPS}. Our proof of Lemma~\ref{L:wente:flowout} does not depend on this statement.
%Indeed it was shown in \cite[the proof of Theorem~3.5]{KPS} that for a limit of $a(\lambda) \to (\lambda-1)^4$ in $\calS^2_1$, the coefficients of some corresponding $b \in \calB_a$ remain bounded if and only if the integral of $\Theta(b)$ along all straight lines connecting the 4 roots of $a$ vanishes. Since the integrals along the lines connecting two roots interchanged by $\lambda\mapsto\Bar{\lambda}^{-1}$ vanish, this is equivalent that the integral along the cycle which surrounds both roots of $a$ inside $B(0,1)$ vanish. Since this cycle is anti--symmetric with respect to $(\lambda,\nu)\mapsto(\Bar{\lambda},\Bar{\nu})$ this is in $\calW$ true for $\Theta(b_2)$ but not for $\Theta(b_1)$. 
Also note that in $\mathcal{W}$, $a_+\to 0$ implies that $b_2(1)\to0$ and $\beta_3\to-4$. 
\end{remark}
\begin{proof}[Proof of Lemma~\ref{L:wente:flowout}.]
By Lemma~\ref{L:wente:delta} there exists $\alpha\in \bbR$ so that $a_- = a_+ + 4\alpha\sqrt{a_+}+16$ holds. Thus $a_+\to 0$ implies $a_-\to 16$. We now consider the vector field on $(a_+,a_-,\beta_2,\beta_3) \in \R^4$ given by equations~\eqref{eq:wente:whitham-wente:dgl-apm} and \eqref{eq:wente:whitham-wente:dgl-beta23}. We show that this vector field has a suitable integral curve flowing out of $(a_+,a_-,\beta_2,\beta_3)=(0,16,1,-4)$. This point is a root of the vector field with non--hyperbolic Jacobi matrix. In order to study the integral curves nearby this point we blow it up.

We introduce the new variable $x=a_+^{1/4}$, which provides a local coordinate of the time domain because $a_+>0$ is strictly monotonic, and $a_+=x^4$. Moreover, the other variables we replace by
\begin{equation}\label{eq:wente:flowout:blowup}
a_--16 = 4\alpha\,x^2+y_1\,x^4\;,\quad \beta_2-1=y_2\,x^2 \quad\text{and}\quad \beta_3+4=y_3\,x^2 \; .
\end{equation}
With respect to $(x,y_1,y_2,y_3)$, the vector field defined by equations~\eqref{eq:wente:whitham-wente:dgl-apm} and \eqref{eq:wente:whitham-wente:dgl-beta23} is given by
\begin{gather*}\begin{aligned}
\dot{x} & = x\, (16 + 4\alpha\,x^2+y_1\,x^4) \\
\dot{y}_1 & = -2\,(y_1-1)\,(16+4\alpha\,x^2+y_1\,x^4) \\
\dot{y}_2 & = -4\alpha - 32y_2-4y_3-x^2 \, (y_1+12\alpha\,y_2+4\,y_2\,y_3+3y_1\,y_2\,x^2) \\
\dot{y}_3 & = -16\alpha - 32y_3 -x^2\, 4 (y_1-y_3^2) \; . 
\end{aligned}\end{gather*}
We will regard $x$ as the blowup variable, by which we blowup the functions $a_--16-4\alpha x^2$, $\beta_2-1$ and $\beta_3+4$ to give the blownup functions $y_1, y_2, y_3$. The exceptional fibre of this blowup is $\{x=0\}$. The vector field given above has exactly one zero on the exceptional fibre, namely at $(x,y_1,y_2,y_3)=(0,1,-\tfrac{1}{16}\alpha,-\tfrac{1}{2}\alpha)$. The Jacobi matrix of the vector field at this zero is
\[
\left( \begin{smallmatrix}
16 & 0 & 0 & 0 \\ 0 & -32 & 0 & 0 \\ 0 & 0 & -32 & -4 \\ 0 & 0 & 0 & -32 
\end{smallmatrix} \right) \; . 
\]
This matrix is hyperbolic. Its eigenvalues are $16$ and $-32$. The eigenvectors for the negative eigenvalue $-32$ are tangential to the exceptional fibre, but the eigenspace for the eigenvalue $16$, which is spanned by $(1,0,0,0)$, is transversal to the exceptional fibre. The Hartman--Grobman linearisation theorem therefore shows that there exists a unique integral curve of the blownup vector field that starts at $(x,y_1,y_2,y_3)=(0,1,-\tfrac{1}{16}\alpha,-\tfrac{1}{2}\alpha)$ into the direction $x>0$. Because of $y_1=1$ at the starting point, the $(x,y_1)$--component of this integral curve has the same tangent vector as the correspondingly blownup originally given integral curve $(a_+,a_-)$ of \eqref{eq:wente:whitham-wente:dgl-apm}, and thus these two curves are equal on the intersection of their domains of definition. If we now take equations~\eqref{eq:wente:flowout:blowup} to define functions $\beta_2$ and $\beta_3$ in terms of $x,y_1,y_2,y_3$, we obtain an integral curve $(a_+,a_-,\beta_2,\beta_3)$ of equations~\eqref{eq:wente:whitham-wente:dgl-apm} and \eqref{eq:wente:whitham-wente:dgl-beta23}. Because $x \to 0$ at the lower boundary, we have $\beta_2 \to 1$ and $\beta_3 \to -4$ there. Thus we have $a(\lambda) \to (\lambda-1)^4$ and $b_2(\lambda) \to \mi\,(\lambda-1)^3$, and therefore $\Theta(b_2) \to d\big(2\mi (1+\lambda)/\sqrt{\lambda}\big)$. In this limit all periods of $\Theta(b_2)$ are purely imaginary, since they converge to integer multiples of the difference $4\mi$ of the values of the function $2\mi (1+\lambda)/\sqrt{\lambda}$ at both points over $\lambda=1$. Because these periods are constant, the same holds along the maximal integral curve. 

Finally we supplement the integral curve $(a_+,a_-,\beta_2,\beta_3)$ with a function $\beta_1$ to an integral curve of the full system of differential equations \eqref{eq:wente:whitham-wente:dgl-apm}--\eqref{eq:wente:whitham-wente:dgl-beta23}. It follows from 
the first equation in~\eqref{eq:wente:whitham-wente:dgl-apm} and \eqref{eq:wente:whitham-wente:dgl-beta1} that  $\dot{\beta}_1/\dot{a}_+ = -\tfrac14\,(\beta_1/a_+)$ holds. Thus $\beta_1 = C a_+^{-1/4}$ with any constant $C\neq 0$ satisfies the requirements.
We may choose $C=1$, and then $\beta_1 \, a_+^{1/4} \to 1$ holds. 
\end{proof}
We saw in the preceeding lemma that on an integral curve of the Whitham flow that preserves $\mathcal{W}$, the differential $\Theta(b_2)$ always has purely imaginary periods. For the polynomial $a$ to lie in $\calS^2_1$, also the other corresponding differential form $\Theta(b_1)$ needs to have purely imaginary periods. The question of when this is the case is discussed in the next
\begin{lemma}\label{L:wente:delta-unique}
On the maximal integral curves of the vector field given by equations~\eqref{eq:wente:whitham-wente:dgl-apm}--\eqref{eq:wente:whitham-wente:dgl-beta23} with $a_+>0$ and $a_+\to 0$, $a_-\to 16$, $\beta_1\, a_+^{1/4}\to 1$, $\beta_2\to 1$, $\beta_3\to -4$ at a boundary of the curve, the corresponding differential form $\Theta(b_1)$ has purely imaginary periods if and only if the constant $\alpha$ from Lemma~\ref{L:wente:delta} is equal to the unique root $\alpha_0\in(0,2)$ of the map $\alpha\mapsto\beta$ from Lemma~\ref{lemma 1}.
\end{lemma}
\begin{proof}
As in Section~\ref{Se:blow up sym point}, we first replace the parameter $\lambda$ by $\varkappa=\tfrac{\lambda-1}{\mi(\lambda+1)}$ and then blowup $\varkappa$ at $\varkappa=0$. The transformation in Lemma~\ref{cayley transform} transforms the pair $(a,b_1)$ defined in~\eqref{eq:wente:whitham-wente:a} and~\eqref{eq:wente:whitham-wente:b} into
\begin{align*}
  \Hat{a}(\varkappa)&=\frac{(\mi+\varkappa)^4}{a(-1)}a\big(\tfrac{\mi-\varkappa}{\mi+\varkappa}\big)
  %=\frac{((\mi-\varkappa)^2-(\mi+\varkappa)^2)^2-\frac{1}{4}(1+\varkappa^2)(a_+(2\mi)^2-a_-(-2\varkappa)^2)}{a_-}\\&
  =\frac{a_-\varkappa^4+(a_++a_--16)\varkappa^2+a_+}{a_-},\\\Hat{b}_1(\varkappa)
   &=\frac{2\mi(\mi+\varkappa)^3}{\sqrt{a(-1)}}b\big(\tfrac{\mi-\varkappa}{\mi+\varkappa}\big)
  % =\frac{2\mi\beta_1(2\varkappa)^22\mi}{\sqrt{a_-}}
  =-\frac{16\beta_1\varkappa^2}{\sqrt{a_-}}.
\end{align*}
Consequently the 1-form $\Theta(b_1)$ on the spectral curve $\{(\lambda,\nu)\mid\nu^2=\lambda a(\lambda)\}$ is transformed into the 1-form $\tfrac{\Hat{b}_1(\varkappa)}{\Hat{\nu}}\tfrac{d\varkappa}{\varkappa^2+1}$ on the spectral curve $\{(\varkappa,\Hat{\nu})\mid\Hat{\nu}^2=(\varkappa^2+1)\Hat{a}(\varkappa)\}$. In the proof of Lemma~\ref{L:wente:flowout} we derived the following asymptotics for the blowup~\eqref{eq:wente:flowout:blowup}:
\begin{align*}
a_+&=x^4,&a_-&=16+4\alpha x^2+\mathrm{O}(x^4),&\beta_1&=x^{-1}+\mathrm{O}(x^0)
& \text{for \,$x=a_+^{1/4} \to 0$\,.}
\end{align*}
This implies that in this limit $\tfrac{2}{x}.\Hat{a}$ converges to $\varkappa^4+\alpha\varkappa^2+1$ and $\Tilde{b}(\varkappa)=\tfrac{2}{x}b\big(\tfrac{x}{2}\varkappa\big)$ to $-8\varkappa^2$. Now Remark~\ref{rescaling b} implies that $\Theta(b_1)$ has in this limit purely imaginary periods if and only if $\alpha$ is the unique $\alpha_0\in(0,2)$ in Lemma~\ref{lemma 1} with $\beta=0$. Because the periods of $\Theta(b_1)$ are constant along the Whitham integral curve, the same holds on the whole maximal integral curve.
\end{proof}
\begin{proof}[Proof of Theorem~\ref{T:wente:wente}]
We saw in Lemma~\ref{L:wente:submfd} that $\mathcal{W}$ is the union of all maximal integral curves of the Whitham flow from Lemma~\ref{L:wente:whitham-wente} contained in $\mathcal{W}$. But we do not yet know that there actually exists such an integral curve, nor that it is unique. To prove these two points, we begin by constructing an integral curve of the vector field given by equations~\eqref{eq:wente:whitham-wente:dgl-apm} so that the corresponding polynomials $a$ defined by equation~\eqref{eq:wente:whitham-wente:a} are members of $\mathcal{W}$. Let $\alpha_0$ be as in Lemma~\ref{L:wente:delta-unique}, and consider the autonomous differential equation
\begin{equation}\label{eq:wente:wente:phiode}
\dot{u} = 2\,u^3+8\,\alpha_0\,u^2 + 32\,u
\end{equation}
that is obtained by substituting $a_+=u^2$ and $a_-=a_++4\alpha_0\,\sqrt{a_+}+16$ in the first equation~\eqref{eq:wente:whitham-wente:dgl-apm}. We obtain a solution of this differential equation by separation of variables. Indeed, due to $|\alpha_0|<2$, and $u>0$ we have $2u^3+8\alpha_0u^2+32u>0$, and therefore the elliptic integral
$$ u \mapsto t = \int_{1}^{u}\frac{1}{2v^3+8\alpha_0v^2+32v}\,\mathrm{d}v$$
is strictly increasing. It tends to $-\infty$ for $u \to 0$ and to some finite $t_0>0$ for $u \to \infty$. Its inverse function $t \mapsto u(t)$ is a solution of \eqref{eq:wente:wente:phiode} that is defined for all $t < t_0$. Therefore 
\begin{equation}\label{eq:wente:wente:apam}
a_+(t) = u^2(t) \quad\text{and}\quad a_-(t) = u^2(t) + 4\alpha_0\,u(t)+16
\end{equation}
defines a maximal integral curve of the vector field given by equations~\eqref{eq:wente:whitham-wente:dgl-apm} with $a_+(t)\to 0$ for $t\to-\infty$ and $a_+(t) \to \infty$ for $t\to t_0$. By Lemma~\ref{L:wente:flowout} there exist functions $\beta_\ind(t)$ ($\ind=1,2,3$) defined at least for $t<t_1$ with some $t_1 \leq t_0$ so that $a_\pm(t)$ and $\beta_\ind(t)$ give a maximal integral curve of the vector field given by equations~\eqref{eq:wente:whitham-wente:dgl-apm}--\eqref{eq:wente:whitham-wente:dgl-beta23} and with $\beta_1\, a_+^{1/4}\to 1$, $\beta_2 \to 1$ and $\beta_3\to -4$ for $t\to -\infty$. By Lemma~\ref{L:wente:delta-unique} and Lemma~\ref{L:wente:flowout}, the periods of $\Theta(b_1)$ and $\Theta(b_2)$ are purely imaginary on this integral curve. Therefore $a$ defined by equation~\eqref{eq:wente:whitham-wente:a} for this integral curve is in $\calS^2_1$, and therefore in $\mathcal{W}$. Because the polynomials \,$a$\, are contained in $\calS^2_1$, the $\beta_\ind$ are in fact defined for all times $t<t_0$.

Conversely, if there were another maximal integral curve of the vector field given by equations~\eqref{eq:wente:whitham-wente:dgl-apm}--\eqref{eq:wente:whitham-wente:dgl-beta23} (up to scaling of $\beta_1$, $\beta_2$) whose polynomials \,$a$\, stay in $\calS^2_1$, then the equation of Lemma~\ref{L:wente:delta} would hold with some $\alpha\neq \alpha_0$. But then the periods of $\Theta(b_1)$ would not be purely imaginary by Lemma~\ref{L:wente:delta-unique}, which is a contradiction.

Therefore $\mathcal{W}$ is the image of the single maximal integral curve with $\alpha=\alpha_0$ constructed above. In particular $\mathcal{W}\neq \varnothing$ and $\mathcal{W}$ is connected and by Lemma~\ref{L:wente:submfd} a 1--dimensional submanifold of $\calS^2_1$. %and that $\Phi: \mathcal{W}\to\bbR$ is a diffeomorphism onto an open interval of length $\pi$.
One obtains the explicit representation \eqref{eq:wente:wente:W} of $\mathcal{W}$ by substituting \eqref{eq:wente:wente:apam} into equation~\eqref{eq:wente:whitham-wente:a}.

Finally we saw in Lemma~\ref{lemma 1} that the blownup spectral curve for $a_+\to 0$ has no branch points on $\bbR \cup \mi\bbR$. Therefore the polynomial $a$ has no zeros where $\varkappa =\tfrac{\lambda-1}{\mi(\lambda+1)} \in \bbR \cup \mi\bbR$, hence $\lambda \in \bbS^1 \cup \bbR$, at least for $t$ near $-\infty$. Along the integral curve, zeroes of $a$ on $\bbS^1$ or $\bbR$ can only occur when two roots  coalesce. However, the polynomials \,$a$\, stay in $\calS^2_1$, so this cannot happen. Hence $a$ does not have any zeroes on $\bbR\cup\bbS^1$ for all times $t$, which means $\calW\subset\calA_-$ by Definition~\ref{D:wente:A+-A-}.

We saw that as the integral curve covers all of $\mathcal{W}$, the function $u$ covers $\bbR_+$, so equation~\eqref{eq:wente:wente:W} now follows from equations~\eqref{eq:wente:wente:apam} and~\eqref{eq:wente:whitham-wente:a}.
\end{proof}
\section{The global Whitham flow}\label{Se:global Whitham}
In this section we define for every $a\in\calS^2_1$ a canonical basis of the first homology group of the spectral curve. This will allow us to construct a basis of \,$\calB_a$\, that is associated to the canonical homology basis in a sense detailed below.  From the following lemma we will deduce the fact that this basis of $\calB_a$ is preserved along the corresponding Whitham flows introduced in Section~\ref{Se:local whitham}. 
%This will allow us to construct a global section of the frame bundle \,$\calF$\,. 
%
\begin{lemma}\label{global cycles}
The space $\calS^2_1$ contains no polynomial whose four roots are all colinear.
\end{lemma}
\begin{proof}
Suppose to the contrary that $a\in\calS^2_1$ has four colinear roots. After some rotation of $\lambda$, all roots of \,$a\in \calS^2_{\lambda_0}$\, are real and $(\lambda,\nu)\mapsto(\Bar{\lambda},\Bar{\nu})$ defines an involution of $\Sigma_a$. Consequently $\calB_a$ is invariant with respect to the map $b(\lambda)\mapsto\overline{b(\Bar{\lambda})}$. Due to \cite[Theorem~3.2]{CS}, $\deg(\gcd(\calB_a))=1$, so the common root $\lambda_0$ has to be a fixed point of both involutions $\lambda\mapsto\Bar{\lambda}$ and $\lambda\mapsto\Bar{\lambda}^{-1}$. This implies $\lambda_0=1$ or $\lambda_0=-1$. In the second case we replace $\lambda$ by $-\lambda$ and obtain in both cases $a\in\mathcal{W}$. Now by Theorem~\ref{T:wente:wente} we have $\calW\subset\calA_-$, which contradicts the assumption of four colinear roots of $a$.
\end{proof}
In Lemma~\ref{L:para:whitham} we constructed Whitham flows on the frame bundle $\calF$. Let us now describe how we construct the global section \,$a \mapsto (a, b_1,b_2)$\, of \,$\calF$\, whose existence is asserted in Theorem~\ref{Th:main}. The basis \,$(b_1,b_2)$\, of \,$\calB_a$\, depends on a choice of two oriented cycles $A_1$ and $A_2$ in $H_1(\Sigma_a,\bbZ)$. Up to orientation they are represented by the two unparameterised curves which are mapped by $\Sigma_a\to\CPone$, $(\lambda,\nu)\mapsto\lambda$ to the straight lines in the $\lambda$--plane which connect $\alpha_1$ with $\Bar{\alpha}^{-1}_1$ and $\alpha_2$ with $\Bar{\alpha}^{-1}_2$, respectively. Altogether, for given $a\in\calS^2_1$ there are 8 possible combinations of the choice of the labelling of the roots of $a$ in $B(0,1)$ and the choice of the orientations of the cycles $A_1$ and $A_2$.

We have that $\rho(A_\ind)$ is homologous to $-A_\ind$ because the corresponding submanifolds are invariant with respect to the anti--holomorphic involution $\rho$ of equation~\eqref{eq:rho} and intersect the fixed point set $\bbS^1$ of $\rho$ twice. For any $b\in \calB_a$ the following calculation using $\rho^\ast\Theta(b)=-\overline{\Theta(b)}$ shows $\int_{A_\ind} \Theta(b) = 0$:
$$ \int_{A_\ind} \Theta(b) = -\overline{\int_{A_\ind} \Theta(b)} = -\int_{A_\ind} \overline{\Theta(b)} = \int_{A_\ind} \rho^* \Theta(b) = \int_{\rho(A_\ind)} \Theta(b) = -\int_{A_\ind} \Theta(b) \; . $$
Moreover let $B_1$ be the cycle that encircles the branch points $\lambda=0$ and $\alpha_1$, and let $B_2$ be the cycle that encircles the branch points $\bar{\alpha}_2^{-1}$ and $\lambda=\infty$. Then the $A$--cycles and the $B$--cycles together form a canonical basis of $H_1(\Sigma_a,\bbZ)$. Because the anti--holomorphic involution $\rho$ reverses both the orientation and the intersection number of cycles, it follows that $\rho(B_\ind)-B_\ind$ is homologous to a linear combination of $A$--cycles. Figure~\ref{fig:symmetric cycles} shows the projections of these cycles to the $\lambda$--plane, each encircling a pair of fixed points of the hyperelliptic involution. This shows that given the oriented cycles $(A_1,A_2)$, there exists a unique basis $(b_1,b_2)$ of $\calB_a$ such that for $k,\ind=1,2$ we have
\begin{align*}
\int_{A_k}\Theta(b_\ind)&=0,&\int_{B_k}\Theta(b_\ind)&=2\pi\mi\delta_{kl}.
\end{align*}
This is equivalent to the following condition, which does not involve $B_1$ and $B_2$:
\begin{align}\label{period map}
\frac{1}{2\pi\mi}\int_C\Theta(b_\ind)&=A_\ind\cdot C&\text{for all}&&C&\in H_1(\Sigma_a,\bbZ).
\end{align}
Here $\cdot$ denotes the intersection form on $H_1(\Sigma_a,\bbZ)$ (compare \cite[Chapter~III Section~1]{FaKr}). Any such triple $(a,b_1,b_2)$ defines holomorphic functions $\mu_\ind:\Sigma_a^\circ\to\bbC$ with $d\ln\mu_\ind=\Theta(b_\ind)$. These functions are uniquely determined if we additionally assume that $\ln\mu_\ind$ has in a neighbourhood of $\lambda=0$ an anti--symmetric branch with respect to the hyperelliptic involution $\sigma$~\eqref{eq:sigma}. In this case the same holds for the branch $-\overline{\rho^\ast\ln\mu_\ind}$ in a neighbourhood of $\lambda=\infty$. Since all elements of $\calS^2_1$ have four distinct simple roots, the labelling of the roots \,$\alpha_1,\alpha_2$\, and the choices of orientations of $(A_1,A_2)$ extend uniquely along continuous paths in $\calS^2_1$. Note that if the straight lines connecting $\alpha_1$ with $\Bar{\alpha}_1^{-1}$ and $\alpha_2$ with $\Bar{\alpha}_2^{-1}$ could pass through each other, then the larger of the two \,$A$--cycles would be transfomed to a sum of both \,$A$--cycles, and the condition~\eqref{period map} would not be preserved. Since this is impossible by Lemma~\ref{global cycles}, the Whitham flows defined in Lemma~\ref{L:para:whitham} preserves the condition \eqref{period map} on \,$\calF$\,. In the following theorem we prove that the Whitham flows corresponding to $(c_1(1),c_2(1))=(\sqrt{a(1)},0)$ and $(c_1(1),c_2(1))=(0,\sqrt{a(1)})$ commute. We call them the first and second Whitham flow, respectively. Recall that $\phi$~\eqref{def:phi} maps $(a,b_1,b_2)$ to the value of the pair of functions $(\theta_{b_1}(y(a)),\theta_{b_2}(y(a)))$ defined in Lemma~\ref{L:theta} at the point $y(a)=(1,\nu)\in\Sigma_a$ with $\nu>0$.
\begin{remark}\label{choice of labelling and signs}
The 8 possibilities of the labelling $(\alpha_1,\alpha_2)$ of the roots of $a$ in $B(0,1)$ and the choice of orientations of $(A_1,A_2)$ are all permuted by the dihedral group $D_4$ which is the subgroup of the invertible real $2\times2$--matrices generated by $\big(\begin{smallmatrix}0&-1\\1&0\end{smallmatrix}\big)$ and $\big(\begin{smallmatrix}-1&0\\0&1\end{smallmatrix}\big)$. This group acts on the column vectors $\big(\begin{smallmatrix}b_1\\b_2\end{smallmatrix}\big)$, $\big(\begin{smallmatrix}A_1\\A_2\end{smallmatrix}\big)$ and the components $\big(\begin{smallmatrix}\phi_1\\\phi_2\end{smallmatrix}\big)$ of the map $\phi$~\eqref{def:phi}%and the components $\big(\begin{smallmatrix}t_1\\t_2\end{smallmatrix}\big)$ of $t\in\Omega$
as left multiplication. For all non--diagonal elements it permutes the labels of the roots $\alpha_1,\alpha_2 \in B(0,1)$ of $a$.%In this way the group acts on the maximal flows~\eqref{eq:whitham flow}. We will use this group action to modify the labelling of $(\alpha_1,\alpha_2)$ and the signs of the solutions $(\nu_1,\nu_2)$ of~\eqref{labelling and orientation} globally so that some further condition is satisfied in case~(A) in Theorem~\ref{Th:global basis}~(iv).
\end{remark}
\begin{theorem}\label{Th:global basis}
Fix $a_0\in\calS_1^2$ together with a labelling of its roots $(\alpha_{1,0},\alpha_{2,0})$ in $B(0,1)$.
\begin{enumerate}
\item[(i)] There exists a unique choice of orientations of $A_1$ and $A_2$ such that with the corresponding basis $(b_{1,0},b_{2,0})$ of $\calB_{a_0}$ obeying~\eqref{period map}, both components of $\phi(a_0,b_{1,0},b_{2,0})$~\eqref{def:phi} are positive.
\item[(ii)] The two Whitham flows on \,$\calF$\, defined in Lemma~\ref{L:para:whitham} by  $(c_1(1),c_2(1))=(\sqrt{a_0(1)},0)$ and $(c_1(1),c_2(1))=(0,\sqrt{a(1)})$ commute, and preserve both~\eqref{period map} and the labelling of \,$(\alpha_1, \alpha_2)$\,.
\item[(iii)] Let the Whitham flows from (ii) have the initial value $(a_0,b_{1,0},b_{2,0})$ at \,$t=0$\,. 
There exists a unique maximal open domain $0\in\Omega\subset\bbR^2$ of the corresponding two-dimensional flow
\begin{align}\label{eq:whitham flow}
\Omega&\to P_\bbR^4\times P_\bbR^3\times P_\bbR^3,&t=(t_1,t_2)\to(a_t,b_{1,t},b_{2,t})
\end{align}
whose restriction to any line $\{(sx,sy)\mid s\in\bbR\}\cap\Omega$ with $(x,y) \in\bbR^2\setminus\{0\}$ is the unique maximal integral curve of the Whitham flow for $(c_1(1),c_2(1))=\sqrt{a_t(1)}(x,y)$.
\item[(iv)] The composition of~\eqref{eq:whitham flow} with the map $\phi$~\eqref{def:phi} is equal to $t\mapsto t+\phi(a_0,b_{1,0},b_{2,0})$. It maps $\Omega$ into $\{\phi\in\bbR^2\mid\phi_1>0,\phi_2>0\}$ and the vector fields of the flow onto $(\frac{\partial}{\partial\phi_1},\frac{\partial}{\partial\phi_2})$.
\item[(v)] Let $(t_n)_{n\in\mathbb{N}}$ be a sequence in $\Omega$, all of whose accumulation points in $\bbR^2$ are end points $(x,y)$ of the maximal intervals of the Whitham flow for $(c_1(1),c_2(1))=\sqrt{a_t(1)}(x,y)$ in~(iii). Any accumulation point $(\alpha_1,\alpha_2)$ of the corresponding sequence $(\alpha_{1,n},\alpha_{2,n})_{n\in\mathbb{N}}$ in $B(0,1)$ of $(a_{t_n})_{n\in\mathbb{N}}$ belongs to one of the following sets:
\begin{enumerate}
\item[(A)] $\{(1,1)\}$.%\label{case one}
\item[(B)] $\bigr(\big(B(0,1)\setminus\{0\}\big)\times\{1\}\bigr) \;\cup\;\bigr(\{1\}\times\big(B(0,1)\setminus\{0\}\big)\bigr)$.%\label{case two}
\item[(C)] $\bigr(\{0\}\times\partial B(0,1)\bigr) \;\cup\;\bigr(\partial B(0,1)\times\{0\}\bigr)$.%\label{case three}
\item[(D)] $\bigr(\{0\}\times\big(B(0,1)\setminus\{0\}\big)\bigr)\;\cup\;\bigr(\big(B(0,1)\setminus\{0\}\big)\times\{0\}\bigr)$.%\label{case four}
\item[(E)] $\{(0,0)\}$.%\label{case five}
\end{enumerate}\end{enumerate}
\end{theorem}
\begin{proof}
We first prove by contradiction that for any given $a_0\in\calS^2_1$ together with a labelling $(\alpha_{1,0},\alpha_{2,0})$ of its roots in $B(0,1)$ and any choice of orientations of $A_1$ and $A_2$ the corresponding unique basis $(b_{1,0},b_{2,0})$ of $\calB_{a_0}$ obeys $\phi_1(a_0,b_{1,0},b_{2,0})\ne0\ne\phi_2(a_0,b_{1,0},b_{2,0})$. Due to Remark~\ref{choice of labelling and signs} this implies (i).

So let us assume $\phi_\ind(a_0,b_{1,0},b_{2,0})=0$ for some $\ind\in\{1,2\}$. Consequently, by definition of $\phi$~\eqref{def:phi} and the function $q_\ind$ in Lemma~\ref{L:theta}, $b_{\ind,0}$ has besides the root at $\lambda=1$ at least one other root in $\bbS^1$. Hence due to the reality condition for \,$b_{\ind,0} \in P_\bbR^3$\,, all three roots of this polynomial are unimodular. Because all periods of $\Theta(b_{\ind,0})$ are imaginary, the real part $\RE q_\ind$ of the anti-derivative $q_\ind$ is single valued on $\Sigma_{a_0}$ with an integration constant uniquely determined by the condition $\sigma^\ast\RE q_\ind=-\RE q_\ind$. We claim that the subset $I_\ind\subset\Sigma_{a_0}$ where $\RE q_\ind$ vanishes is a union of the embedded copy of $\bbS^1$ given by the fixed point set $\Sigma_{a_0,\bbR}$ of $\rho$ together with three disjoint embbeded copies of $\bbS^1$. Each of the three latter copies passes through a pair of the fixed points of the hyperelliptic involution which are interchanged by the anti-holomorphic involution $\rho$ and intersects $\Sigma_{a_0,\bbR}$ in the two points over one root of $b_{\ind,0}$. This follows as in \cite[Lemma~9.5]{HKS3} from four observations: First $I_\ind$ is away from the roots of $b_{\ind,0}$ a submanifold, and is at any root of order $l$ contained in $I_\ind$ an intersection point of $l$-many local submanifolds. Secondly, $I_\ind$ contains the submanifold $\Sigma_{a_0,\bbR}$. Thirdly, $I_\ind$ contains all six fixed points of the hyperelliptic involution $\sigma$. Finally the boundary of any connected component of \,$\Sigma_{a_0} \setminus I_\ind$\, contains either the point \,$\lambda=0$\, or \,$\lambda=\infty$\,, since due to the maximum principle the harmonic function $\RE q_\ind$ is unbounded on any such component.

Hence $\Sigma_{a_0}\setminus I_\ind$ has four connected components. Let $D_\ind$ be the unique connected component over $\lambda\in B(0,1)$, whose boundary contains the segment of $\Sigma_{a_0,\bbR}$ which starts at the root $y(a_0)$ of $\Theta(b_{\ind,0})$ over $\lambda=1$ with positive $\nu$ and moves in anti-clockwise order around $\lambda\in\bbS^1$ until it reaches the next root of $\Theta(b_{\ind,0})$ over another root of $b_{\ind,0}$. The boundary $\partial D_\ind$ is divided into six segments: three segments of $\Sigma_{a_0,\bbR}$ between two roots of $\Theta(b_{\ind,0})$, and three half--circles passing through exactly one of the three fixed points of $\sigma$ over $\lambda\in B(0,1)$. We denote the three segments of $\Sigma_{a_0,\bbR}$ by $S_1$, $S_2$ and $S_3$ such that they cover $\partial B(0,1)$ in anti-clockwise order along $\partial D_\ind$ starting with $S_1$ at $\sigma(y(a_0))$ and ending with $S_3$ at $y(a_0)$.
%% FIGURE 1: THE ROOTS OF B_\IND.
%
\begin{figure} \label{fig:HC-cycles}
\begin{tikzpicture}[scale = .8]
  \draw (0,0) circle (3);

    \node[below] at (0,-0.15) {$0$};
    \node at (0,-1.5) {$D_\ind$};

    \draw (0,0) -- (190:4);

    %draw two arcs
    \draw ({3+sqrt(2)},{-2+sqrt(2)}) arc (45:135:2) node[very near start, below right, inner sep = 8pt] {$\overline{\alpha}_1^{-1}$};
    \draw ({(3/sqrt(2))-sqrt(2)+2},{(3/sqrt(2)+sqrt(2)}) arc (0:-90:2) node[very near start, above right] {$\overline{\alpha}_2^{-1}$};

    % set sep to be 0.5?
    % draw dashed arcs
    \draw[densely dashed] (3:2.85) arc (3:42:2.85) node[midway, left] {$S_1$};
    \draw[densely dashed] (48:2.85) arc (48:187:2.85) node[midway, below right] {$S_2$};
    \draw[densely dashed] (193:2.85) arc (193:357:2.85) node[midway, above] {$S_3$};

    % draw dashed arcs around solution lines
    \draw[densely dashed] (3:2.85) arc (93:135:2.1547) node[near end, above, inner sep = 7pt] {$HC_1$};
    \draw[densely dashed] (357:2.85) arc (96:135:1.8572);

    \draw[densely dashed] (42:2.85) arc (-48:-87:2.1547) node[near end, below, inner sep = 6pt] {$HC_2$};
    \draw[densely dashed] (48:2.85) arc (-51:-90:1.8572);

    \draw[densely dashed] (187:2.85) -- (0,0.1515);
    \draw[densely dashed] (193:2.85) -- (0,-0.1515);

    % add caps to lines
    \draw[densely dashed] (0,-0.1515) arc (282.744:462.744:0.1515);

    \draw[densely dashed] (0.7382,1.7041) arc (-249.856:-69.856:0.1714) node[midway, left] {$\alpha_2$};

    \draw[densely dashed] (1.4353,-0.4790) arc (141.01:321.01:0.184) node[midway, below] {$\alpha_1$};

    % draw circles for points
    \draw (3,0) circle (3pt) node[above right] {$\lambda = 1$};
    \draw (190:3) circle (3pt);
    \draw (45:3) circle (3pt);

    % info
    \draw (4.5,2) circle (2pt) node[right, inner sep = 7pt] {roots of $b_\ind$};
    \draw[densely dashed] (4.3,1.5) -- (4.6,1.5) node[right, inner sep = 7pt] {$\partial D_\ind$};
\end{tikzpicture}
\caption{The segments $S_1$, $S_2$, $S_3$ and half--circles $HC_1$ and $HC_2$.}
\end{figure}
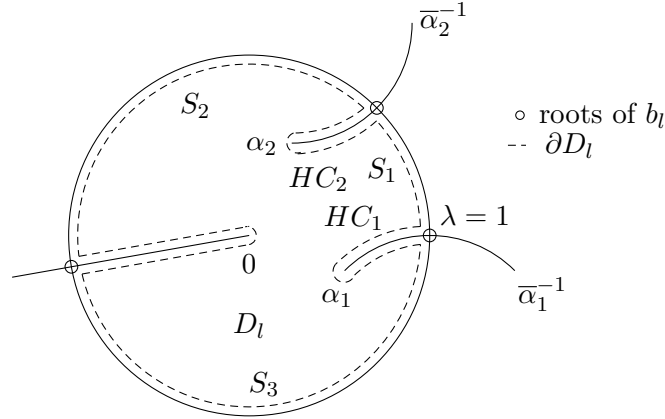
For a higher order root of $b_{\ind,0}$ one or two of these segments shrink to a point. The path from $\sigma(y(a))$ to $y(a)$ along $\Sigma_{a_0,\bbR}$ is equal to the concatenation $S_1+\sigma(S_2)+S_3$ where $\sigma(S_l)$ denotes the image of $S_l$ under $\sigma$. This gives
\begin{align}\label{vanishing phi}
\phi_\ind(a_0,b_{1,0},b_{2,0})&=\tfrac{1}{2}\IM\left(\int_{S_1}\Theta(b_{\ind,0})-\int_{S_2}\Theta(b_{\ind,0})+\int_{S_3}\Theta(b_{\ind,0})\right).
\end{align}
Let $HC_\ind$ denote the unique half--circle in $\partial D_\ind$ which passes through $\alpha_\ind$. In order to write the cycle $C_\ind$ which encircles both fixed points of $\sigma$ over $\lambda\in B(0,1)$ besides $(\lambda,\nu)=(\alpha_\ind,0)$ as a combination of the segments of $\partial D_\ind$ we distinguish between the following mutually exclusive cases:
\begin{enumerate}
\item $HC_\ind$ connects $S_1$ and $S_2$: $C_\ind=S_1+HC_\ind+S_2+\sigma(S_3)$.
\item $HC_\ind$ connects $S_2$ and $S_3$: $C_\ind=\sigma(S_1)+S_2+HC_\ind+S_3$.
\item $HC_\ind$ connects $S_3$ and $S_1$: $C_\ind=S_1+\sigma(S_2)+S_3+HC_\ind$.
\end{enumerate}
Hence we may calculate $\IM\int_{C_\ind}\Theta(b_{\ind,0})$ as linear combination of the imaginary parts of integrals of $\Theta(b_{\ind,0})$ along the segments of $\partial D_\ind$. Due to~\eqref{period map} this linear combination vanishes, which yields
\begin{gather}\label{def value phi}\begin{aligned}
\phi_\ind(a_0,b_{1,0},b_{2,0})&=\IM\int_{S_1}\Theta(b_{\ind,0})+\tfrac{1}{2}\IM\int_{HC_\ind}\Theta(b_{\ind,0})&\text{in case~(1)},\\
\phi_\ind(a_0,b_{1,0},b_{2,0})&=\tfrac{1}{2}\IM\int_{HC_\ind}\Theta(b_{\ind,0})+\IM\int_{S_3}\Theta(b_{\ind,0})&\text{ in case~(2)},\\
\phi_\ind(a_0,b_{1,0},b_{2,0})&=-\tfrac{1}{2}\IM\int_{HC_\ind}\Theta(b_{\ind,0})&\text{ in case~(3)}.
\end{aligned}\end{gather}
By definiton of $D_\ind$ the imaginary part of the integral of $\Theta(b_{\ind,0})$ is monotonic along $\partial D_\ind$. Hence $\phi_\ind(a_0,b_{1,0},b_{2,0})=0$ implies $\int_{HC_\ind}\Theta(b_{\ind,0})=0$ in case~(1)-(3), which is impossible for $a_0\in\calS^2_1$. This finishes the proof of~(i).

Now we consider the Whitham flows locally and prove (ii) and (iv). The labelling of the roots $(\alpha_{1,0},\alpha_{2,0})$ and the choices of orientations of $A_1$ and $A_2$ have a unique continuous extension to any simply connected open neighbourhood $O\subset\calS^2_1$ of $a_0$. Consequently the corresponding basis $(b_1,b_2)$ which satisfies~\eqref{period map} defines a smooth section of the frame bundle $\calF$ on $O$. The image of this section is a two--dimensional submanifold \,$\widehat{O}$\, of $\calF$ which is diffeomorphic to $O$. Both Whitham flows define vector fields along \,$\widehat{O}$\, since they preserve all periods of the differentials $\Theta(b_1)$ and $\Theta(b_2)$, and therefore also condition~\eqref{period map} on $(b_1,b_2)$. The restriction of $\phi$~\eqref{def:phi} to \,$\widehat{O}$\, is an immersion into $\bbR^2$. For sufficiently small $O$, this restriction of $\phi$ is a diffeomorphism. By \eqref{eq:ck} and the choice of \,$(c_1(1),c_2(1))$\,, the two Whitham flows obey \,$\tfrac{\partial \phi_i}{\partial t_j}(a_t, b_{1,t}, b_{2,t}) = \delta_{ij}$\,. In particular, the two Whitham vector fields are the pullbacks to \,$\widehat{O}$ of the vector fields $\frac{\partial}{\partial\phi_1}$ and $\frac{\partial}{\partial\phi_2}$ under $\phi$\,. Due to the proof of~(i) the components of $\phi$ can never vanish along the flow, and their positivity is preserved along the connected subset $\Omega$ in part~(iii). This proves (iv). Since $\frac{\partial}{\partial\phi_1}$ and $\frac{\partial}{\partial\phi_2}$ commute, part~(ii) follows.
  
(iii) For each $(x,y)\in\bbR^2 \setminus \{0\}$ the integral curve of the corresponding Whitham vector field is defined on a unique maximal interval $I_{(x,y)}$. There exists a unique subset $\Omega\subset\bbR^2$ with $\Omega\cap\{(sx,sy)\mid s\in\bbR\}=\{(sx,sy)\mid  s\in I_{(x,y)}\}$ for all $(x,y)\in\bbR^2\setminus\{0\}$.  Because solutions of ODEs depend continuously on their parameters,
% the vector field and on initial values 
% (see e.g.\ Teschl: Theorem~2.11) 
$\Omega$ is open. We can uniquely define the map \eqref{eq:whitham flow} by the condition that its restriction to any line $\{(sx,sy)\mid s\in\bbR\}\cap\Omega$ with $(x,y) \in\bbR^2\setminus\{0\}$ is the unique maximal integral curve of the corresponding Whitham flow. Because the two Whitham vector fields from (ii) commute, this map is indeed the two-dimensional flow of these vector fields.

(v): The bounded sequence $(\alpha_{1,n}, \alpha_{2,n})_{n\in\mathbb{N}}$ has a  convergent subsequence in \,$\overline{B(0,1)}^2$\,. First we assume $\alpha_1\ne0$ and $\alpha_2\ne0$. Then the corresponding subsequence of $(a_{t_n})_{n\in\mathbb{N}}$ converges to some $a\in P_\bbR^4$. For any norm \,$\|\,\cdot\,\|$\, on $P_\bbR^3$, a subsequence of the renormalised basis $(\frac{b_{1,n}}{\|b_{1,n}\|},\frac{b_{2,n}}{\|b_{2,n}\|})_{n\in\mathbb{N}}$ converges to a basis of $\calB_a$, and $a$ belongs to the closure of \,$\calS^2_1$\,. We now show by contradiction that the assumption on \,$(t_n)_{n\in\mathbb{N}}$\, implies \,$a\in\partial\calS^2_1$\,. Indeed, otherwise the sequence \,$(a_{t_n},b_{1,t_n},b_{2,t_n})_{n\in\mathbb{N}}$\, has an accumulation point in \,$\calF$\, and the sequence $(\phi(a_{t_n},b_{1,t_n},b_{2,t_n}))_{n\in\mathbb{N}}$ has an accumulation point in \,$\R^2$\,.
%together with $(a_{t_n})_{n\in\mathbb{N}}$ the corresponding data $(\nu_{1,n},\nu_{2,n})_{n\in\mathbb{N}}$, the basis $(b_{1,t_n},b_{2,t_n})_{n\in\mathbb{N}}$ and the values $(\phi(a_{t_n},b_{1,t_n},b_{2,t_n}))_{n\in\mathbb{N}}$ have an accumulation point. 
So by~(iv) also $(t_n)_{n\in\mathbb{N}}$ has an accumulation point $(x,y)\in\bbR^2$, which is assumed to be an end point of the maximal interval of the Whitham flow for $(c_1(1),c_2(1))=\sqrt{a_t(1)}(x,y)$. If $(t_n)_{n\in\mathbb{N}}$ has such a limit $(x,y)$, and $(a_{t_n})_{n\in\mathbb{N}}$ has an accumulation point $a\in\calS^2_1$, then we consider the 2-dimensional Whitham flow from (iii) with initial point $a$. The flow covers an open neighbourhood of \,$a$\, in \,$\calS^2_1$\, which contains a subsequence of \,$(a_{t_n})$\,. The uniqueness of the Whitham flows implies that $(x,y)$ belongs to \,$\Omega$\,, and not to $\partial\Omega$, whenever the corresponding subsequence $(a_{t_n})_{n\in\mathbb{N}}$ has an accumulation point in $\calS^2_1$. Therefore, due to our assumption on $\,(t_n)_{n\in\mathbb{N}}$\,, any accumulation point of \,$(a_{t_n})_{n\in\mathbb{N}}$\, belongs to \,$\partial\calS^2_1$\,. Now by Proposition~\ref{P:boundary-S21-C4lambda} either (A) or (B) is true.

In the other case either $\alpha_1$ or $\alpha_2$ is zero. Then the other root belongs to the disjoint union $\overline{B(0,1)}=\partial B(0,1)\cup\big(B(0,1)\setminus\{0\}\big)\cup\{0\}$. So either (C), (D) or (E) is true.
\end{proof}
\section{The limits of the map $\phi$ in the cases A to E}\label{Se:limits a to e}
In this section we prove for all five cases of Theorem \ref{Th:global basis}~(v) that all accumulation points of the corresponding values of $\phi$~\eqref{def:phi} belong to \,$\partial \triangle$\,, compare Theorem~\ref{Th:main}. In Figure~\ref{fig:triangle} it is shown which part of \,$\partial \triangle$\, is covered by the accumulation points in each case. We will refine the five cases of Theorem~\ref{Th:global basis}~(v). The following table summarises our findings: 

\begin{figure*}[h]
	\begin{tabularx}{\textwidth}{|c|l|l|l|l|X|}
		\hline
		\textbf{case} & \textbf{final $(\alpha_1,\alpha_2)$} & \textbf{values of \,$\phi\in\partial \triangle$\,} & \textbf{\#} & \textbf{blowup} & \textbf{integrable system} \\
		\hline
		(A) & $\{(1,1)\}$ & hypothenuse & 1 & of $\Hat{\calS}^2_0$ to 1-d manifold & defocussing NLS \\
		\hline
		(B) & 1-d submanifold & endpts of hypothen. & 2 & not performed & $\sinh$--Gordon\\
		\hline
		(C) & $\{(0,1),(1,0)\}$ & endpts of short sides & 2 & not performed & non--conf.\ har.\ maps\\
		\hline
		(D) & $\alpha_\ind\!\in\!(0,1),\alpha_k=0$ & two short sides & 2 & not performed & non--conf.\ har.\ maps\\
		\hline
		(E) & $\{(0,0)\}$ & vertex at right angle & 1 & of $\calS^2_1$ to 1-d manifold & non--real KdV\\
		\hline
	\end{tabularx}
\end{figure*}

The second column lists the final values of $(\alpha_1,\alpha_2)$ pertaining to each case; the cases~(B), (C) and (D) have been fuerther restricted. The third column describes what part of $\partial\triangle$ each case is mapped to under $\phi$, whilst the fourth column states the number of components of each image. The fifth states whether a blowup is utilised and if so, which space is blown up. The sixth column lists the relevant integrable system. It is identified by the choices of the marked points and the Mittag--Leffler distributions, which generate the space--like flows in the Jacobian (cf. \cite[Section 7]{K-L-S-S}). The spectral curves of all four appearing systems are hyperelliptic. The marked points are one branch point (KdV) \cite{GriS1}, two unbranched points over one point of $\bbP^1$ (NLS), two branch points ($\sinh$--Gordon) and two pairs of unbranched points over two points of $\bbP^1$ (non--conformal harmonic maps to $\bbS^2$ cf.\ \cite{hit:tor}) of the corresponding two--sheeted covering over $\bbP^1$, respectively. The Mittag Leffler distributions are in all cases anti--symmetric with respect to hyperelliptic involution and have first order poles at all marked points. For KdV and NLS the space of such Mittag--Leffler distributions is one--dimensional and in the two other cases two--dimensional.

\noindent We make some further remarks on the blowups. In case~(A), at the unique limit of $(a_{t_n})_{n\in\mathbb{N}}$ the map $\phi$~\eqref{def:phi} takes all values in the hypothenuse of $\partial\triangle$. Thus we need to blow up in order to determine the corresponding values of $\phi$. By contrast, in case~(E) the limit of $\phi$ takes only one value. Hence for our present purposes we do not need a blowup in case~(E). However as indicated in the table and as we shall describe in  forthcoming work, such a blowup is interesting as it yields a solution of a different integrable system. In each case~(B) and (D), the restricted set of the positions of $(\alpha_1,\alpha_2)$ is one--dimensional and contains no distinguished points at which we might perform a blowup. Indeed, the endpoints belong to the cases~(A), (C) or (E). In case~(C) there are two distinguished points with non--coalescing roots $(\alpha_1,\alpha_2)$. In the boundary of $\Hat{\calS}^2_0$ these points separate the cases~(B) and (D), in each of which one of the roots $(\alpha_1,\alpha_2)$ is fixed whilst the other moves. A blowup of (C) should interpolate between cases~(B) and (D) such that both blown up roots move. However, since these roots at $\lambda=0,1$ are separated, this should result in a relationship between the blowups of the spectral curves at $\lambda=0$ and $\lambda=1$. This seems to be less interesting from the point of view of switching from one integrable system to another.

% The global flow~\eqref{eq:whitham flow} in  Theorem~\ref{Th:global basis} depends for some given $a_0\in\calS_1^2$ on the labelling of the roots $(\alpha_{1,0},\alpha_{2,0})$ in $B(0,1)$ of $a_0$. The orientations of $(A_1,A_2)$ are determined in Theorem~\ref{Th:global basis}~(i) such that both components of $\phi$ are positive along the two-dimensional flow.
In the proofs of this section we shall use the simplified notation $(a_n,b_{1,n},b_{2,n}):=(a_{t_n},b_{1,t_n},b_{2,t_n})$.

\subsection{The case (A)} We first consider case (A) in Theorem~\ref{Th:global basis}~(v). In Lemma~\ref{cayley transform} we identify $\calS^2_1$  with $\Hat{\calS}^2_0$. The point \,$\varkappa^4 \in \partial \Hat{\calS}^2_0$\, corresponds to \,$(\lambda-1)^4 \in \partial \calS^2_1$\, which is the accumulation point of \,$(a_{t_n})$\, in case (A). In Section~\ref{Se:blow up sym point} we considered the blowup of $\Hat{\calS}^2_0$ at $\varkappa^4$. Here we shall blowup $\calS^2_1$ at $(\lambda-1)^4$ and hence directly relate $\lambda$ to the blownup parameter $\varkappa$. The normalisation used in the blowup of $\Hat{\calS}^2_0$ is not expressed neatly in terms of the roots $(\alpha_{1,n},\alpha_{2,n})$. However $\overline{\calS^\circ_0}$ provides a scale--invariant version of the blowup, and we shall choose an alternative scaling which is a simple expression in terms of these roots.
%In the following lemma we refer back to this blowup and use implicitly this identification of $\calS^2_1$ with $\Hat{\calS}^2_0$.
%
\begin{lemma}\label{Le:case a}
%\begin{enumerate}
%\item[(i)] The assumptions of Definition~\ref{def:cycles c} are satisfied on an open set $O\subset\calS^2_1$ which is mapped by the diffeomorphism in Lemma~\ref{cayley transform} onto an open set $\Hat{O}\subset\Hat{\calS}^2_0$ such that the union $\Psi^{-1}[\Hat{O}]\cup\Big(\{0\}\times\big(\calS^\circ_0\cap\bbS^3[\varkappa]\big)\Big)$ is a tubular neighbourhood of $\Big(\{0\}\times\big(\calS^\circ_0\cap\bbS^3[\varkappa]\big)\Big)\subset\Hat{\calE}^2_0$.
%\item[(ii)] 
%There exists a unique section \,$b$\, of \,$\calF|_O$\, such that~\eqref{period map} is satisfied. The composition of the inverse of the diffeomorphism \,$O\to\Hat{O}$\, with $\phi\circ b$ extends to a homeomorphism from $\overline{\calS_0^\circ}\cap \bbS^3[\varkappa]$ to $\{ (\varphi,\,\pi - \varphi) \mid \varphi \in [0,\,\pi] \}$. The two end points $(a_1,a_4)=\big(\pm\tfrac{2}{\sqrt{3}},0\big)$ and the central element with $a_1=0$ of the parametrisation of $\overline{\calS_0^\circ}\cap \bbS^3[\varkappa]$ in Theorem~\ref{th:blow-up}~(iv) are mapped onto $(\pi,0)$, $(0,\pi)$ and $(\tfrac{\pi}{2}, \tfrac{\pi}{2})$, respectively.
%\end{enumerate}
\hspace{2em}
\begin{enumerate}
\item[(i)] In Theorem~\ref{Th:global basis}~(v) case~(A) let $M_n:=\max\{|\alpha_{1,n}-1|,|\alpha_{2,n}-1|\}$. Then any accumulation point of $\left(\frac{\alpha_{1,n}-1}{\mi M_n},\frac{\alpha_{2,n}-1}{\mi M_n}\right)_{n\in\mathbb{N}}$ is a pair of roots of some element of $\overline{\calS^\circ_0}$.
\item[(ii)] There exists an increasing homeomorphism $\varphi_A\!:\!\big[\!-\!\frac{2}{\sqrt{3}},\!\frac{2}{\sqrt{3}}\big]\!\to\![0,\!\pi]$ with the following property: if $\left(\frac{\alpha_{1,n}-1}{\mi M_n},\frac{\alpha_{2,n}-1}{\mi M_n}\right)_{n\in\mathbb{N}}$ converges to the roots $(\alpha^\circ_1,\alpha^\circ_2)$ of  $\varkappa^4+a^\circ_1\varkappa^3+a^\circ_2\varkappa^2+a^\circ_4\in\overline{\calS^\circ_0}$, then
\begin{gather}
\lim\limits_{n\to\infty}\phi(a_{t_n},b_{1,t_n},b_{2,t_n})=\begin{cases}\Big(\varphi_A\big(a^\circ_1/\sqrt{a^\circ_2}\big),\pi-\varphi_A\big(a^\circ_1/\sqrt{a^\circ_2}\big)\Big)&\text{if }\RE\alpha^\circ_1<\RE\alpha^\circ_2,\\\Big(\pi-\varphi_A\big(a^\circ_1/\sqrt{a^\circ_2}\big),\varphi_A\big(a^\circ_1/\sqrt{a^\circ_2}\big)\Big)&\text{if }\RE\alpha^\circ_1>\RE\alpha^\circ_2.
\end{cases}\end{gather}
The case $\RE\alpha^\circ_1=\RE\alpha^\circ_2$ does not arise.
\end{enumerate}
\end{lemma}
\begin{proof} %(i): The M\"obius transformation~\eqref{eq:moebius-kappa} $\lambda\mapsto\frac{\lambda-1}{\mi(\lambda+1)}=\frac{\mi(1-|\lambda|^2)+2\IM\lambda}{|\lambda+1|^2}$ maps
(i): The corresponding sequence $(\Hat{a}_n)_{n\in\mathbb{N}}$ in $\Hat{\calS}^2_0$ converges to $\varkappa^4$. Then, due to Theorem~\ref{th:blow-up}, any accumulation point of $(|\Hat{a}_n|^{-1}.\Hat{a}_n)_{n\in\mathbb{N}}$ belongs to $\overline{\calS^\circ_0}$. We may rescale the roots $\Hat{\alpha}_{\ind,n}=\frac{\alpha_{\ind,n}-1}{\mi(\alpha_{\ind,n}+1)}$ of $\Hat{a}_n$ in the upper half plane by $\frac{2}{M_n}$ instead of $\frac{1}{|\Hat{a}_n|}$ and replace $\mi(\alpha_{\ind,n}+1)$ in the denominator by its limit $2\mi$. This shows that after passing to a convergent subsequence of the bounded sequence  $\left(\frac{\alpha_{1,n}-1}{\mi M_n},\frac{\alpha_{2,n}-1}{\mi M_n}\right)_{n\in\mathbb{N}}$ the limit $(\aax_1,\aax_2)$ is a pair of roots of some $a^\circ\in\overline{\calS^\circ_0}$.

(ii): In order to determine $\lim\phi(a_n,b_{1,n},b_{2,n})$, we shall investigate the limits of $\Theta(b_{\ind,n})$ for $\ind=1,2$. We first calculate the blowup $a^\circ$ by writing $a_n$ as:
\begin{align}\label{representation an}
a_n(\lambda)&=\frac{(\lambda-\alpha_{1,n})(1-\Bar{\alpha}_{1,n}\lambda)}{|\alpha_{1,n}|}\frac{(\lambda-\alpha_{2,n})(1-\Bar{\alpha}_{2,n}\lambda)}{|\alpha_{2,n}|}
\end{align}  
and by using the spectral parameter $\varkappa=\frac{\lambda-1}{\mi M_n}$ instead of $\lambda=\mi M_n\varkappa+1$: %The blownup roots $\frac{\alpha_{\ind,n}-1}{\mi M_n}$ are contained in $\overline{B(0,1)}$ and $\frac{\alpha_{l,n}-1}{\mi M_n}$ in $\partial B(0,1)$. After passing to a subsequence, the roots converge to $\aax_\ind:=\lim\frac{\alpha_{\ind,n}-1}{\mi M_n}$
\begin{align}\nonumber
\ax_n(\varkappa)&:=\frac{a_n(\mi M_n\varkappa+1)}{M_n^4}=\prod_{\ind=1}^2\frac{\left(\varkappa-\frac{\alpha_{\ind,n}-1}{\mi M_n}\right)\left(\Bar{\alpha}_{\ind,n}\varkappa+\frac{1-\Bar{\alpha}_{\ind,n}}{\mi M_n}\right)}{|\alpha_{\ind,n}|}\nonumber\\\label{blow up 2}\lim\ax_n(\varkappa)&=:\ax(\varkappa)=(\varkappa-\aax_1)(\varkappa-\bar{\alpha}^\circ_1)(\varkappa-\aax_2)(\varkappa-\bar{\alpha}^\circ_2)\in\bbR^4[\varkappa].
\end{align}
%Let us now explain why this blowups yields the same spectral curve as in Section~\ref{Se:blow up sym point}: There we first replaced $\lambda$ by $\varkappa=\tfrac{\lambda-1}{\mi(\lambda+1)}$ and rescaled $\varkappa$ by some appropriate sequence $(C_n)_{n\in\bbN}$ of positive numbers converging to zero such that for $\ind=1,2$ the corresponding transformed roots $C_n^{-1}\tfrac{\alpha_{\ind,n}-1}{\mi(\alpha_{\ind,n}+1)}$ stay finite. Since $\lim\alpha_{\ind,n}=1$ we have
%$$\limC_n^{-1}\tfrac{\alpha_{\ind,n}-1}{\mi(\alpha_{\ind,n}+1)}=\lim\tfrac{\alpha_{\ind,n}-1}{\mi M_n}\quad\text{for}\quad C_n=\tfrac{M_n}{2}.$$
The notation of the polynomials $a^\circ_n(\varkappa)$ always includes the argument $(\varkappa)$, in order to distinguish them from the coefficients $a^\circ_1$, $a^\circ_2$ and $a^\circ_4$ of the limit $a^\circ(\varkappa)$. Expressing $\lambda$ and $\nu$ in terms of $\varkappa$ and $\nx_n:=\frac{\nu}{M_n^2}$ and using the relations
\begin{align}\label{1-form blow up 3}
%\Theta^\circ(b^\circ_{\ind,n})&=\frac{b^\circ_{\ind,n}(\varkappa)}{\nx_n}d\varkappa,&
(\nx_n)^2&=(\mi M_n\varkappa+1)^3\ax_n(\varkappa),&b^\circ_{\ind,n}(\varkappa)&:=\frac{\mi b_{\ind,n}(\mi M_n\varkappa+1)}{M_n},
\end{align}
then $\Theta(b_{\ind,n})$ becomes equal to $\Theta^\circ(b^\circ_{\ind,n})$~\eqref{limit 1-form blow up 3}.
Here the denominator of $b^\circ_{\ind,n}$ together with the factor $\frac{d\lambda}{d\varkappa}=\mi M_n$ compensate for the square root of $M_n^4$ in the denominator of $\ax_n(\varkappa)$. After passing to a subsequence, the sequence $(b^\circ_{\ind,n}/\|b^\circ_{\ind,n}\|)_{n\in\mathbb{N}}$ converges. The non--trivial limit $\Tilde{b}^\circ_\ind$ defines the 1-form $\Theta^\circ(\Tilde{b}^\circ_\ind)$~\eqref{limit 1-form blow up 3} on the limit curve $(\nu^\circ)^2=a^\circ(\varkappa)$.
%
%\begin{align}\label{limit 1-form blow up 3}
%\Theta^\circ(\Tilde{b}^\circ_\ind)&=\frac{\Tilde{b}^\circ_\ind(\varkappa)}{\nx}d\varkapp(\nx)^2&=\ax(\varkappa).
%\end{align}
%
The limits of the fixed points of the hyperelliptic involution at \,$\varkappa=\frac{\mi}{M_n}$\, and at \,$\varkappa=\infty$\, coalesce a double point at \,$\varkappa=\infty$\, on this limit curve. Let $(A_1,A_2)$ denote the oriented cycles on the sequence of curves $(\nx_n)^2 = (\mi M_n\varkappa+1)^3\ax_n(\varkappa)$ as described in Section~\ref{Se:global Whitham}, and $(A^\circ_1,A^\circ_2)$ their limit on $(\nu^\circ)^2 = a^\circ(\varkappa)$. The orientations of $(A_1,A_2)$ determine a choice of \,$A_1 \pm A_2$\, which in the limit encircles the double point at \,$\varkappa=\infty$\,. Since the integrals of $\Theta(b_{\ind,n})$ along the cycles $(A_1,A_2)$ vanish, the 1-form $\Theta^\circ(\Tilde{b}^\circ_\ind)$ has the same property along the cycles $(A^\circ_1,A^\circ_2)$ and no residue at this double point of the limit curve. This implies $\deg\Tilde{b}^\circ_\ind\in\{2,3\}$.

Furthermore, on the former sequence of curves $(\nx_n)^2 = (\mi M_n\varkappa+1)^3\ax_n(\varkappa)$ let $B$ denote a cycle which encircles $\alpha^\circ_1$ and $\alpha^\circ_2$ and no other root of $a^\circ(\varkappa)$, see Figure~\ref{fig:canonical basis}. We claim that there exists a unique $b^\circ_\ind\in\mi\big(\calB^\circ_{a^\circ}\cap\bbR^2[\varkappa]\big)$ with $\int_B\Theta^\circ(b^\circ_\ind)=2\pi\mi A^\circ_\ind\cdot B$ and that this $b^\circ_\ind$ is the limit of $(b^\circ_{\ind,n})_{n\in\mathbb{N}}$. We first prove $\deg\Tilde{b}^\circ_\ind=2$ by contradiction. If $\deg \Tilde{b}^\circ_\ind=3$, then any accumulation point $\Tilde{b}_\ind$ of $(b_{\ind,n}/\|b_{\ind,n}\|)_{n\in\mathbb{N}}$ is a multiple of $(\lambda-1)^3$, since all three roots converge to $\lambda=1$. The corresponding 1-form on the limit curve  $\nu^2=\lambda a(\lambda)=\lambda(\lambda-1)^4$ is the derivative $%d(2\frac{\lambda+1}{\sqrt{\lambda}})=
\frac{(\lambda-1)^3}{\nu}\frac{d\lambda}{\lambda}$ of the meromorphic function $2\frac{\lambda+1}{\sqrt{\lambda}}$. This function is not zero at $\lambda=1$. Since each cycle $B_\ind$ converges to a path on the normalisation of the limit curve whose end points are the two points over $\lambda = 1$, we have $\int_{B_\ind}\Theta(\Tilde{b}_\ind)\ne0$. In particular, $(\|b_{\ind,n}\|)_{n\in\mathbb{N}}$ is bounded. By solving~\eqref{1-form blow up 3} for $b_{\ind,n}(\lambda)=-\mi M_nb^\circ_{\ind,n}\big(\frac{\lambda-1}{\mi M_n}\big)$ we see that  $M_n^{-2}$ times the $\varkappa^3$--coefficient of $b^\circ_{\ind,n}$ is bounded. For $\deg\Tilde{b}^\circ_\ind=3$ this implies $b^\circ_\ind=0$, which contradicts $\int_B\Theta^\circ(b^\circ_{\ind,n})=2\pi\mi A_\ind\cdot B=\pm2\pi\mi$. This proves $\deg\Tilde{b}^\circ_\ind=2$. To summarise, we have shown so far $\deg\Tilde{b}^\circ_\ind=2$ and that the 1-form $\Theta^\circ(\Tilde{b}^\circ_\ind)$~\eqref{limit 1-form blow up 3} on the limit curve $\nu^\circ)^2=a^\circ(\varkappa)$ has vanishing $A_1$-- and $A_2$--periods, no residue at $\varkappa=\infty$ and imaginary $B$--period. These four properties characterise the 1-forms $\Theta^\circ(b^\circ)$ with $b^\circ\in\mi\big(\calB^\circ_{a^\circ}\cap\bbR^2[\varkappa]\big)$, c.f.~the proof of Lemma \ref{lemma 1}. Furthermore, $b^\circ\mapsto\left|\int_B\Theta(b^\circ)\right|$ defines a norm on the one--dimensional real vector space $\mi\big(\calB^\circ_{a^\circ}\cap\bbR^2[\varkappa]\big)$. Now $\int_B\Theta^\circ(b^\circ_{\ind,n})= 2\pi\mi A_\ind\cdot B$ implies that $(b^\circ_{\ind,n})_{n\in\mathbb{N}}$ is bounded and has the unique accumulation point $b^\circ_\ind$ with $\int_B\Theta^\circ(b^\circ_\ind)=2\pi\mi A^\circ_\ind\cdot B$. This proves the claim. Hence $A^\circ_\ind\cdot B=\pm1$ implies $b^\circ_2=\pm b^\circ_1$.

Now we show $b^\circ_2=-b^\circ_1$. For the proof we apply the construction of the set $I_\ind$ and the discs $D_\ind$ in the proof of part~(i) of Theorem~\ref{Th:global basis} to the sequence $(a_n)_{n\in\mathbb{N}}$. Let for $\ind=1,2$ denote $(D_{\ind,n})_{n\in\mathbb{N}}$ the corresponding sequence of discs over $\lambda\in B(0,1)$. We wish to control the limits $I^\circ$ and $D^\circ$ of the images of $I_{\ind,n}$ and $D_{\ind,n}$ under the map $(\lambda,\nu)\mapsto(\varkappa,\nu^\circ_n)=\big(\frac{\lambda-1}{\mi M_n},\frac{\nu}{M_n^2}\big)$. To achieve this, we utilise that $(b^\circ_{\ind,n})_{n\in\mathbb{N}}$ converges to an element of $\mi\big(\calB^\circ_{a^\circ}\cap\bbR^2[\varkappa]\big)$ and apply the same construction to a nontrivial $b\in\mi\big(\calB^\circ_{a^\circ}\cap\bbR^2[\varkappa]\big)$. So let $I^\circ$ be the zero locus on the limit curve $(\nu^\circ)^2=a^\circ(\varkappa)$ of the unique real part $\RE q^\circ$ of the anti-derivative of $\Theta^\circ(b)$ which is anti-symmetric with respect to $(\varkappa,\nu^\circ)\mapsto(\varkappa,-\nu^\circ)$. The arguments in the proof of part~(i) of Theorem~\ref{Th:global basis} which explain the structure of $I_\ind\subset\Sigma_{a_0}$ carry over to $I^\circ$ and show that this set is the union of two pairs of disjoint embedded copies of $\bbS^1$. Each  submanifold of the first pair either passes through $\alpha^\circ_1$ and $\Bar{\alpha}^\circ_1$ or $\alpha^\circ_2$ and $\Bar{\alpha}^\circ_2$, respectively. The second pair are the two embedded copies of $\bbS^1$ over $\varkappa\in\bbR$. Each cycle of the first pair intersects both cycles of the second pair in one of the two points over one root of $b$. The convergence of $(b^\circ_{\ind,n})_{n\in\mathbb{N}}$ to a multiple of $b$ implies that the sets $I^\circ_{\ind,n}$ converge pointwise for $n\to\infty$ to $I^\circ$. The limits on $\partial D^\circ$ of the segments $S_1$, $S_2$, $S_3$ and of the half--circles $HC_1$ and $HC_2$, see Figure~\ref{fig:HC-cycles}, are denoted by $S^\circ_1$, $S^\circ_2$, $S^\circ_3$ and $HC^\circ_1$ and $HC^\circ_2$, respectively. In our case~(A), both half--circles and one of the segments $S_1$ or $S_2$ shrinks to a point at $\lambda=1$, whereas in general none of $S^\circ_1, S^\circ_2, S^\circ_3, HC^\circ_1$ and $HC^\circ_2$ collapse. In $\partial D^\circ$ only the limit of the half--circle through the fixed point $\lambda=0$ of $\sigma$ shrinks to the double point at $\varkappa=\infty$ of the limit curve $(\nu^\circ)^2=a^\circ(\varkappa)$. Hence the segment of $\partial D^\circ$ which connects both points over $\varkappa=0$ always passes through a root of $a^\circ$. For sufficiently large $n$ this gives the following possibilities for the three mutually exclusive cases in the proof of part~(i) in Theorem~\ref{Th:global basis}: for one index $\ind=1,2$ case~(3) holds and for the other index $\indd=1,2$ either case~(1) or case~(2). The orientations of $A_1$ and $A_2$ are chosen in such a way, that both components of $\phi$ are positive. So by~\eqref{def value phi}, $\IM q_{\ind,n}$ is decreasing along $\partial D_{\ind,n}$ for the index $\ind=1,2$ whose $HC_\ind$ connects $S_3$ with $S_1$ and increasing for $k$. This shows $b^\circ_2=-b^\circ_1$.

Now we finish the proof of (ii). Since the roots of $a^\circ$~\eqref{blow up 2} are equal to $\alpha^\circ_1$, $\Bar{\alpha}^\circ_1$, $\alpha^\circ_2$ and $\Bar{\alpha}^\circ_2$, we have
\begin{align}\label{coefficients roots}
-\aax_1-\Bar{\alpha}_1^\circ-\aax_2-\Bar{\alpha}_2^\circ&=a^\circ_1,&-(\aax_1+\Bar{\alpha}_1^\circ)|\aax_2|^2-(\aax_2+\Bar{\alpha}_2^\circ)|\aax_1|^2&=a_3^\circ=0,&|\aax_1|^2|\aax_2|^2=a_4^\circ.
\end{align}
We remark that $a^\circ_3=0$ follows from Theorem~\ref{th:blow-up}~(ii). For $a^\circ_4>0$ no root vanishes and by the middle equation $\RE\aax_1$ and $\RE\aax_2$ have opposite sign. Note that the quotient  $a^\circ_1/\sqrt{a^\circ_2}$ is independent of the action~\eqref{action} on $a^\circ$ and is therefore a suitable parameter for the blowup $\Hat{\calE}^2_0$. Due to Theorem~\ref{th:blow-up}~(ii), $a^\circ_4$ is non--negative and only vanishes for $a^\circ_1/\sqrt{a^\circ_2}\to\pm 2/\sqrt{3}$. In particular, in this limit one of the $\alpha^\circ_\ind$ vanishes and the other has non--zero real part, which implies $\RE\aax_1\not=\RE\aax_2$ for all $a^\circ\in\overline{\calS^\circ_0}$.

We first consider the case $\RE\aax_1<\RE\aax_2$. For $a^\circ_1/\sqrt{a^\circ_2}\in\big({-} 2/\sqrt{3},0\big)$~\eqref{coefficients roots} yields $\RE\aax_1<0<-\RE\aax_1<\RE\aax_2$ and $|\aax_1|<|\aax_2|$. In particular, $\aax_1\to 0$ for $a^\circ_1/\sqrt{a^\circ_2}\to -2/\sqrt{3}$. Consequently the half--circle $HC^\circ_1$ passes through $\varkappa=0$ and connects $S_3^\circ$ with $S^\circ_1$ (case(3) for $\ind=1$) whilst $HC^\circ_2$ connects $S^\circ_1$ with $S^\circ_2$ (case~(1) for $\ind=2$). Due to the proof of Theorem~\ref{th:blow-up}~(ii) the elements of $\calB^\circ_{a^\circ}\cap\bbR^2[\varkappa]$ are multiples of $\varkappa(\varkappa- a^\circ_1/2)$ (cf.~\eqref{blow up form 1}). Their two roots coalesce in the limit $a^\circ_1\uparrow 0$ and the segment $S^\circ_1$ shrinks to a point. Due to~\eqref{def value phi}, $\big(\phi_1(a_n,b_{1,n},b_{2,n})\big)_{n\in\mathbb{N}}$ converges to $-\frac{1}{2}\IM\int_{HC^\circ_1}\Theta^\circ(b^\circ_1)$ which is finite and tends to zero in the limit $a^\circ_1/\sqrt{a^\circ_2}\downarrow-\frac{2}{\sqrt{3}}$. Due to~\eqref{eq:whitham local} and Claim~2 shortly after~\eqref{eq:rescaling}, $c(0)=c_3$ is non--zero for $a^\circ_4>0$. Since both $\Theta^\circ(b^\circ_\ind)$ are multiples of~\eqref{blow up form 1} with constant periods of opposite sign, this implies that one component of $\lim\phi(a_n,b_{1,n},b_{2,n})$ is strictly monotonic increasing and the other component is strictly monotonic decreasing with respect to $a^\circ_1/\sqrt{a^\circ_2}\in\big[{-} 2/\sqrt{3}, 2/\sqrt{3}\big]$. The first component is strictly increasing, since it is non--negative and tends to zero at the left end point. In order to show that the two components sum to $\pi$ we modify the argument in the proof of part~(i) of Theorem~\ref{Th:global basis} and calculate $\lim\phi_2(a_n,b_{1,n},b_{2,n})$ in terms of $\lim\phi_1(a_n,b_{1,n},b_{2,n})=-\tfrac{1}{2}\IM\int_{HC^\circ_1}\Theta^\circ(b^\circ_1)$. The cycle $C_1$ defined just after~\eqref{vanishing phi} intersects $A_2$ in exactly one point with intersection number $\pm 1$ and $\big(\int_{C_1}\Theta(b_{2,n})\big)_{n\in\mathbb{N}}$ converges to $\pm2\pi\mi$. Consequently the arguments of the proof of part~(i) of Theorem~\ref{Th:global basis} give
$$\phi_2(a_n,b_{1,n},b_{2,n})=\tfrac{1}{2}\IM\left(\int_{C_1}\Theta(b_{2,n})-\int_{HC_1}\Theta(b_{2,n})\right)=\pm\pi-\tfrac{1}{2}\IM\int_{HC_1}\Theta(b_{2,n}).$$
Inserting $b^\circ_2=-b^\circ_1$ in the limits and using that both components are non--negative gives
\begin{align}\label{sum of components}
\lim\phi_1(a_n,b_{1,n},b_{2,n})+\lim\phi_2(a_n,b_{1,n},b_{2,n})=\pi.
\end{align}
In the limit $a^\circ_1/\sqrt{a^\circ_2}\to 0$ all three polynomials $a^\circ$ and $b^\circ_1=-b^\circ_2$ are even and the spectral curve has an extra symmetry $\tau:(\varkappa,\nx)\mapsto(-\Bar{\varkappa},\Bar{\nu}^\circ)$ which preserves the point over $\varkappa=0$ with $\nx>0$. Moreover $S^\circ_1$ has collapsed into this point, which is the end point of $HC^\circ_1$ and the initial point of $HC^\circ_2$. Hence $\tau$ maps $HC^\circ_1$ onto $-HC^\circ_2$. Here $b^\circ_2\in\mi\big(\calB^\circ_{a^\circ}\cap\bbR^2[\varkappa]\big)$ obeys $b^\circ_2(-\Bar{\varkappa})=-\overline{b^\circ_2(\varkappa)}$ and $\tau^\ast\Theta^\circ(b^\circ_2)=\overline{\Theta^\circ(b^\circ_2)}$. In case~(3) for $\ind=1$ and in case~(1) with collapsed $S^\circ_1$ for $\ind=2$~\eqref{def value phi} gives
\begin{align*}\lim\phi_2(a_n,b_{1,n},b_{2,n})&=\tfrac{1}{2}\IM\int_{HC^\circ_2}\Theta^\circ(b^\circ_2)=\tfrac{1}{2}\IM\overline{\int_{HC^\circ_2}\tau^\ast\Theta^\circ(b^\circ_2)}=-\tfrac{1}{2}\IM\int_{\tau(HC^\circ_2)}\Theta^\circ(b^\circ_2)\\&=-\tfrac{1}{2}\IM\int_{-HC_1}\Theta^\circ(-b^\circ_1)=-\tfrac{1}{2}\IM\int_{HC_1}\Theta^\circ(b^\circ_1)=\lim\phi_1(a_n,b_{1,n},b_{2,n}).
\end{align*}
Thus for $a^\circ_1/\sqrt{a^\circ_2}\uparrow0$ the two components of $\lim\phi(a_n,b_{1,n},b_{2,n})$ coincide and are equal to $\pi/2$.

Similarly, for  $a^\circ_1/\sqrt{a^\circ_2}\in\big(0, 2/\sqrt{3}\big)$~\eqref{coefficients roots} yields $\RE\aax_1<-\RE\aax_2<\RE\aax_2$ and $|\aax_1|>|\aax_2|$. In this case $\aax_2\to 0$ for  $a^\circ_1/\sqrt{a^\circ_2}\to 2/\sqrt{3}$, the half--circle $HC_1$ connects $S_2$ with $S_3$ (case~(2) for $\ind=1$) and $HC_2$ connects $S_3$ with $S_1$ (case~(3) for $\ind=2$). The application of~\eqref{def value phi} first gives $\lim\phi_2(a_n,b_{1,n},b_{2,n})=-\tfrac{1}{2}\IM\int_{HC^\circ_2}\Theta^\circ(b^\circ_2)$ and then an argument analogous to that in the proof of~\eqref{sum of components} again proves this equation. From the Whitham flow in the proof of Theorem~\ref{th:blow-up}~(ii) we know that both components of $\lim\phi(a_n,b_{1,n},b_{2,n})$ are in fact analytic and strictly monotonic in dependence on $a^\circ_1/\sqrt{a^\circ_2}\in\big(- 2/\sqrt{3}, 2/\sqrt{3}\big)$. So the limit $a^\circ_1/\sqrt{a^\circ_2}\downarrow 0$ coincides with the limit $a^\circ_1/\sqrt{a^\circ_2}\uparrow 0$ which we already computed. This finishes the case $\RE\aax_1<\RE\aax_2$.

In the other case with $\RE\aax_2<\RE\aax_1$ the components of all pairs $(\alpha_{1,n},\alpha_{2,n})$, $(\aax_1,\aax_2)$, $(b_{1,t_n},b_{2,t_n})$, $(b^\circ_1,b^\circ_2)$ and $\big(\phi_1(a_n,b_{1,n},b_{2,n}),\phi_2(a_n,b_{1,n},b_{2,n})\big)$ should be interchanged.
\end{proof}
\subsection{The case (B)}
\begin{lemma}\label{Le:case b}
In case~(B) of Theorem~\ref{Th:global basis}~(v) the sequence $(a_n,b_{1,n},b_{2,n})_{n\in\mathbb{N}}$ converges and $\lim a_n$ belongs to $\partial\calS^2_1$, which is described in Proposition~\ref{P:boundary-S21-C4lambda}. In particular, there exists a one--dimensional submanifold $I_B\subset\big(B(0,1)\setminus\{0\}\big)$, such that actually $(\alpha_1,\alpha_2)\in\big(I_B\times\{1\}\big)\cup\big(\{1\}\times I_B\big)$. Finally,
$$\lim\phi(a_{t_n},b_{1,t_n},b_{2,t_n})=\begin{cases}(\pi,0)&\text{for }(\alpha_1,\alpha_2)\in I_B\times\{1\},\\(0,\pi)&\text{for }(\alpha_1,\alpha_2)\in\{1\}\times I_B.\end{cases}$$
\end{lemma}
\begin{proof}
We interchange the indices $\ind=1,2$ if necessary such that $\alpha_2=1$ holds. By~\eqref{representation an} the sequence $(a_n)_{n\in\mathbb{N}}$ converges to $a(\lambda)=(\lambda-1)^2(\lambda-\alpha_1)(1-\Bar{\alpha}_1\lambda)/|\Bar{\alpha}_1|\in\partial \calS^2_1$. In the proof of Proposition~\ref{P:boundary-S21-C4lambda} the corresponding roots $\alpha_1$ are determined as the two families $\alpha_1=k\left(\frac{2r\pm\mi(r^2-1)}{r^2+1}\right)^{-1}$ parameterised by $k\in(0,1)$. Here $k\mapsto r\in(k,1)$ is an analytic function defined in \cite[Proposition~2.2]{KSS}. Each of these two families defines a connected component of the one--dimensional submanifold $I_B$.

Let us now show that for both $\ind=1,2$ the sequence $(b_{\ind,n})_{n\in\mathbb{N}}$ converges. To prove this it suffices to show any such sequence has a convergent subsequence and only one accumulation point. For any norm on $P_\bbR^3$ the sequence $(\frac{b_{\ind,n}}{\|b_{\ind,n}\|})_{n\in\mathbb{N}}$ has a convergent subsequence whose limit $\Tilde{b}_\ind$ vanishes at $\lambda=1$ and defines the meromorphic 1--form $\Theta(\Tilde{b}_\ind)$ on the curve~\eqref{blow up curve 2}:
\begin{gather}\label{blow up curve 2}
\{(\lambda,\nu)\in\bbC^2\mid\nu^2=\lambda a(\lambda)\}.
\end{gather}
The cycles $A_1$ and $B_1$ converge to a canonical basis of the normalisation of the limit curve.

We start with $\ind=2$: both periods of $\Theta(\Tilde{b}_2)$ along $A_1$ and $B_1$ vanish. Therefore this differential is the derivative of a meromorphic function $\Tilde{q}_2$, which has two first order poles at $\lambda=0$ and $\lambda=\infty$ and is anti--symmetric with respect to the hyperelliptic involution $\sigma$. This function has degree 2 and has two roots at the fixed points of $\sigma$ at $\lambda=\alpha_1$ and $\Bar{\alpha}_1^{-1}$. In fact, $\Tilde{q}_2$ is a non--zero multiple of the function $\frac{\nu}{\lambda(\lambda-1)}$ which extends holomorphically on the normalisation of~\eqref{blow up curve 2} to both points of the double point at $\lambda=1$. In particular, $\Tilde{q}_2$ does not vanish at either point of this double point. Since these points are interchanged by $\sigma$, the difference $\Tilde{q}_2(y(a))-\Tilde{q}_2(\sigma(y(a)))$ of the values at both points of this double point is non--zero. The cycle $B_2$ converges to a path from one point of this double point to the other point. From $\int_{B_2}\Theta(b_{2,n})=2\pi\mi$ we obtain $\lim \frac{2\pi\mi}{\|b_{2,n}\|} =\pm\Tilde{q}_2(y(a)) - \Tilde{q}_2(\sigma(y(a)))\ne0$. Hence $(b_{2,n})_{n\in\mathbb{N}}$ converges and the limit defines the 1--form $\lim\Theta(b_{2,n})=\pm\frac{2\pi\mi}{\Tilde{q}_2(y(a))-\Tilde{q}_2(\sigma(y(a)))}d\Tilde{q}_2$ on~\eqref{blow up curve 2}. This proves the convergence of $(b_{2,n})_{n\in\mathbb{N}}$.

Now the other limit $\lim\Theta(b_{1,n})$ is uniquely determined by the period along $B_1$ and the integral along the limit of $B_2$. Therefore also $b_{1,n}$ itself converges.

Finally we determine $\lim\phi(a_n,b_{1,n},b_{2,n})$ by decomposing the path from $\sigma(y(a))$ to $y(a)$ into a path from $\sigma(y(a))$ to $y(a)$ which is homologous to $\pm B_1$ and a path from $\sigma(y(a))$ to $y(a)$, which only surrounds $\alpha_2$ and no other fixed point of $\sigma$. In the limit the second path shrinks to a point. So for $l=1,2$ the integrals of $\Theta(b_{l,n})$ along the second path converge to $0$. This gives $\lim\phi_\ind(a_n,b_{1,n},b_{2,n})=\pm\frac{1}{2}\IM\int_{B_1}(\Theta(b_{l,n}))$. The non--negativity of $\phi_\ind(a_n,b_{1,n},b_{2,n})$ finishes the proof.
\end{proof}
\subsection{The cases (C) and (D)} In case (C) and (D) we interchange the indices $\ind=1,2$ if necessary such that $\alpha_1=0$ holds. We first blowup the sequence of spectral curves and use the spectral parameter $\lambda^-=|\alpha_{1,n}|\lambda$ instead of $\lambda=\frac{\lambda^-}{|\alpha_{1,n}|}$. The blownup roots $|\alpha_{1,n}|\alpha_{1,n}$, $|\alpha_{1,n}|\alpha_{2,n}$, $|\alpha_{1,n}|\Bar{\alpha}_{1,n}^{-1}$ and  $|\alpha_{1,n}|\Bar{\alpha}_{2,n}^{-1}$ are all bounded. After passing to a subsequence, $(|\alpha_{1,n}|\Bar{\alpha}_{1,n}^{-1})_{n\in\mathbb{N}}$ converges to $\alpha^-_1\in\bbS^1$ and the three other sequences converge to zero. Again~\eqref{representation an} yields the following limit:
\begin{gather}\label{blow up a 2}
\lim a_n^-(\lambda^-)\!:=\lim|\alpha_{1,n}|^4a_n\big(\tfrac{\lambda^-}{|\alpha_{1,n}|}\big)=:a^-(\lambda^-)=\alpha_0(\lambda^-)^3(\lambda^-\!-\!\alpha^-_1)\text{ with }\alpha_0=\lim\frac{\Bar{\alpha}_{1,n}\,\Bar{\alpha}_{2,n}}{|\alpha_{1,n}\,\alpha_{2,n}|}.
\end{gather}
This limit defines the blowup of the spectral curve with a higher order double point at $\lambda^-=0$:
\begin{gather}\label{blow up curve 3}
\{(\lambda^-,\nu^-)\in\mathbb{C}^2\mid(\nu^-)^2=\lambda^-a^-(\lambda^-)\}.
\end{gather}
The orientations of $A_1$ and $A_2$ determine a cycle $\Tilde{B}_1=B_1\pm A_2$ (see Figure~\ref{fig:symmetric cycles}) which converges to a cycle on the limit curve encircling the double point at $\lambda^-=0$. The cycle $B_2$ converges to a path which starts at one point of the double point, encircles $\lambda^-=\infty$ and ends at the other point of the double point. With respect to \,$\lambda^- =|\alpha_{1,n}|\lambda$\, the 1-forms $\Theta(b_{\ind,n})$\, are equal to 
\begin{align}\label{blow up form 3}
\Theta^-(b^-_{\ind,n})&=\frac{b^-_{\ind,n}(\lambda^-)}{\nu^-_n}\frac{d\lambda^-}{\lambda^-}\text{ with}&b^-_{\ind,n}(\lambda^-)&:=|\alpha_{1,n}|^{\frac{5}{2}}b_{\ind,n}\big(\tfrac{\lambda^-}{|\alpha_{1,n}|}\big)\text{ and}&(\nu^-_n)^2&=\lambda^-a^-_n(\lambda^-).
\end{align}
The factor $|\alpha_{1,n}|^{\frac{5}{2}}$ in the definition of $b^-_{\ind,n}$ compensates for the factor in the relation $\nu^-_n=|\alpha_{1,n}|^{\frac{5}{2}}\nu$, which follows from the definitions of $\lambda^-$, $a^-_n$ and $\nu^-_n$.
\begin{lemma}\label{Le:case c and d}
\hspace{2em}
\begin{enumerate}
\item[(i)] In Theorem~\ref{Th:global basis}~(v) case~(C) actually $(\alpha_1,\alpha_2)\in\{(0,1),(1,0)\}$ holds.
\item[(ii)] In Theorem~\ref{Th:global basis}~(v) case~(D) actually $(\alpha_1,\alpha_2)\in \bigr(\{0\}\times(0,1)\bigr)\cup\bigr((0,1)\times\{0\}\bigr)$ holds.
\item[(iii)] %There exists an increasing homeomorphism $\varphi_{CD}:(0,1]\to(0,\pi]$ such that
\begin{align*}
\lim\big(\phi(a_{t_n},b_{1,t_n},b_{2,t_n})\big)&=\begin{cases}\big(\pi+2\big(\arctan\sqrt{\alpha_2}-\arctan\frac{1}{\sqrt{\alpha_2}}\big),0\big)&\text{for }(\alpha_1,\alpha_2)\in\{0\}\times(0,1],\\\big(0,\pi+2\big(\arctan\sqrt{\alpha_1}-\arctan\frac{1}{\sqrt{\alpha_1}}\big)\big)&\text{for }(\alpha_1,\alpha_2)\in(0,1]\times\{0\}.\end{cases}
\end{align*}
\end{enumerate}                                                                 \end{lemma}
\begin{proof}
We interchange the indices $\ind=1,2$ if necessary such that $\alpha_1=0$ holds. The proof is similar to the proof of Lemma~\ref{Le:case b}. However, this time we first prove the convergence of $(b^-_{\ind,n})_{n\in\mathbb{N}}$~\eqref{blow up form 3} for $\ind=1,2$ on the blownup curve~\eqref{blow up curve 3}. Any limit $\Tilde{b}^-_\ind$ of a subsequence $(b^-_{\ind,n}/\|b^-_{\ind,n}\|)_{n\in\mathbb{N}}$ defines an anti--symmetric 1-form with respect to $\sigma:(\lambda^-,\nu^-)\mapsto(\lambda^-,-\nu^-)$ given by
\begin{align}\label{blow up form 3a}
\Theta^-(\Tilde{b}^-_\ind)&=\frac{\Tilde{b}^-_\ind(\lambda^-)}{\nu^-}\frac{d\lambda^-}{\lambda^-}.
\end{align}
Since $b_{\ind,n}$ has at least two roots which stay bounded for $n\to\infty$, the limit $\Tilde{b}^-_\ind$ has at least a double root at $\lambda^-=0$. Therefore the only possible poles of the meromorphic 1-form~\eqref{blow up form 3} on the curve~\eqref{blow up curve 3} are simple poles at $\lambda^-=0$ and a double pole at $\lambda^-=\infty$. Since $\lambda^-=\infty$ is a branch point the anti--symmetry with respect to $\sigma$ forces the residue there to be zero. If~\eqref{blow up form 3a} has zero integral along the limit of $\Tilde{B}_1$, then on the normalisation of the curve it has no pole at either point over $\lambda^-=0$. As the normalisation has genus zero, such a 1-form is the derivative of a meromorphic function which is anti--symmetric with respect to $\sigma$. This function has only one pole of first order at $\lambda^-=\infty$. Thus it is a biholomorphic map from the normalisation to \,$\CPone$\, and takes different values at the two points of the double point at $\lambda^-=0$. In particular, if the integral of~\eqref{blow up form 3a} along the limit of $B_2$ also is zero, then $\Tilde{b}^-_\ind=0$. If $(\|b^-_{\ind,n}\|)_{n\in\mathbb{N}}$ is unbounded, then, due to the relations $\big(\int_{\Tilde{B}_1}\Theta(b_{\ind,n}),\int_{B_2}\Theta(b_{\ind,n})\big)=2\pi\mi(\delta_{1,k},\delta_{2,k})$, the sequence $(b^-_{\ind,n}/\|b^-_{\ind,n}\|)_{n\in\mathbb{N}}$ has an accumulation point $\Tilde{b}^-_\ind$ so that $\big(\int_{\Tilde{B}_1}\Theta(\Tilde{b}^-_\ind),\int_{B_2}\Theta(\Tilde{b}^-_\ind)\big)=(0,0)$. Since this sequence does not have zero as an accumulation point, this proves the boundedness of $(\|b^-_{\ind,n}\|)_{n\in\mathbb{N}}$. The convergence of $(b^-_{\ind,n})_{n\in\mathbb{N}}$ follows, since the accumulation point $b^-_\ind$ is uniquely determined by $\big(\int_{\Tilde{B}_1}\Theta^-(b^-_\ind),\int_{B_2}\Theta^-(b^-_\ind)\big)=2\pi\mi(\delta_{1,\ind},\delta_{2,\ind})$.

For $\ind=1$ the order of the root at $\lambda^-=0$ of $b^-_1$ is at least two. Moreover, this order is two by $\mathrm{Res}_{\lambda^-=0}\Theta^-(b^-_1)\ne0$. Therefore one root of $b_{1,n}$ converges to $\lambda=\infty$ and, due to $b_{1,n}\in P_\bbR^3$, another one to $\lambda=0$. Hence any accumulation point of $(b_{1,n}/\|b_{1,n}\|)_{n\in\mathbb{N}}$ is a real multiple of $\mi\lambda(\lambda-1)\in P_\bbR^3$ and has only one unimodular root at \,$\lambda=1$\,. In case~(C)  Lemma~\ref{L:adouble} gives $\alpha_2=1$ which proves~(i).

For $\ind=2$ the limit $b^-_2$ is a multiple of $(\lambda^-)^3$, since~\eqref{blow up form 3a} has no residue at $\lambda^-=0$.

Now we calculate the limit of $(a_n,b_{1,n},b_{2,n})_{n\in\mathbb{N}}$ without blowing up $\lambda$ and determine for $\ind=1,2$ the limits of $(\Theta(b_{\ind,n}))_{n\in\mathbb{N}}$ on the corresponding limit curve. The formula~\eqref{representation an} for $a_n$ implies
$$\lim|\alpha_{1,n}|a_n=\Tilde{a}\quad\text{with}\quad \Tilde{a}(\lambda)=\frac{\lambda(\lambda-\alpha_2)(1-\Bar{\alpha}_2\lambda)}{|\alpha_2|}.$$
Therefore the pairs $(\lambda,|\alpha_{1,n}|^{\frac{1}{2}}\nu)$ converge to elements of the curve
\begin{gather}\label{limit curve 3}
\big\{(\lambda,\Tilde{\nu})\in\bbC^2\,\,\bigr|\,\,\Tilde{\nu}^2=\lambda \Tilde{a}(\lambda)\big\}.
\end{gather}
The cycles $\Tilde{B}_1$ converge to the cycle $\Tilde{B}_1$ on~\eqref{limit curve 3} encircling the double point at $\lambda=\infty$, and any accumulation point of the sequence of 1-forms $(\Theta(b_{\ind,n})/(|\alpha_{1,n}|^{\frac{1}{2}}\|b_{\ind,n}\|))_{n\in\mathbb{N}}$ is non--zero.

We showed above that any accumulation point of $(b_{1,n}/\|b_{1,n}\|)_{n\in\mathbb{N}}$ is a multiple of $\mi\lambda(\lambda-1)$. Because the limit is an element of $\bbC^2[\lambda]$ with non--zero highest coefficient, the convergence of $(b^-_{1,n})_{n\in\mathbb{N}}$ and the relation $|\alpha_{1,n}|^{\frac{1}{2}}b_{1,n}(\lambda)=|\alpha_{1,n}|^{-2}b^-_{1,n}(|\alpha_{1,n}|\lambda)$  imply that $(|\alpha_{1,n}|^{\frac{1}{2}}\|b_{1,n}\|)_{n\in\mathbb{N}}$ has a non--zero limit. Hence $(\Theta(b_{1,n}))_{n\in\mathbb{N}}$ converges to a 1-form on~\eqref{limit curve 3} with residues $\pm 1$ at the two points over $\lambda=0$.

For $\ind=2$, $\int_{\Tilde{B}_1}\Theta(b_{2,n})=0$, and any accumulation point of $(\Theta(b_{2,n})/\|b_{2,n}\|)_{n\in\mathbb{N}}$ has zero residue at both points over $\lambda=0$. Hence any accumulation point of $(b_{2,n}/\|b_{2,n}\|)_{n\in\mathbb{N}}$ belongs to $P_\bbR^3$ and has degree $3$. By~\eqref{blow up form 3} the highest coefficient of $|\alpha_{1,n}|^{\frac{1}{2}}b_{2,n}$ is $|\alpha_{1,n}|$ times the highest coefficient of $b^-_{2,n}$. Now the convergence of $(b^-_{2,n})_{n\in\mathbb{N}}$ implies $\lim|\alpha_{1,n}|^{\frac{1}{2}}b_{2,n}=0=\lim\Theta(b_{2,n})$.

As in the proof of Lemma~\ref{Le:case b} we decompose the path from $\sigma(y(a_n))$ to $y(a_n)$ into a path from $\sigma(y(a_n))$ to $y(a_n)$, which only surrounds $\alpha_2$ and no other fixed point of $\sigma$ and a path from $\sigma(y(a_n))$ to $y(a_n)$ which is homologous to $\pm B_1$. In both cases (C) and (D) we have $\lim\phi_2(a_n,b_{1,n},b_{2,n})_{n\in\mathbb{N}}=0$.

In case (C) the first path from $\sigma(y(a_n))$ to $y(a_n)$ shrinks to a point. By the non-negativity of $\phi_1(a_n,b_{1,n},b_{2,n})$ we obtain $\lim\phi_1(a_n,b_{1,n},b_{2,n})=\frac{1}{2}\IM\int_{B_1}\Theta(b_{1,n})=\pi$. This proves case~(C) in~(iii).

It remains to prove (ii) and in case~(D) determine $\lim\phi_1(a_n,b_{1,n},b_{2,n})$. To simplify this we replace $\Tilde{\nu}$ by the global parameter $z=\pm\mi\frac{\sqrt{|\alpha_2|}\Tilde{\nu}}{\lambda(1-\Bar{\alpha}_2\lambda)}$, where the sign is specified below. From~\eqref{limit curve 3} we get
\begin{align*}
z^2&=\frac{\lambda-\alpha_2}{\Bar{\alpha}_2\lambda-1}&\Longleftrightarrow&&\lambda&=\frac{z^2-\alpha_2}{\Bar{\alpha}_2z^2-1}.
\end{align*}
The involutions $\sigma$ and $\rho$ act as $z\mapsto -z$ and $z\mapsto\Bar{z}^{-1}$, respectively. The meromorphic 1-form $\lim\Theta(b_{1,n})$ satisfies $\rho^\ast\lim\Theta(b_{1,n})=-\overline{\lim\Theta(b_{1,n})}$ and has simple poles at both points over $\lambda=0$ with residues $\pm1$. For any choice of sign of $z$ we choose the sign $\alpha=\pm\sqrt{\alpha_2}$ such that
\begin{gather}\label{limit form 3b}\begin{aligned}
\lim\Theta(b_{1,n})&=\frac{dz}{z-\alpha}-\frac{dz}{z+\alpha}+\frac{\Bar{\alpha}dz}{\Bar{\alpha}z+1}-\frac{\Bar{\alpha}dz}{\Bar{\alpha}z-1}%=\frac{2\alpha(\Bar{\alpha}^2z^2-1)-2\Bar{\alpha}(z^2-\alpha^2)}{(z^2-\alpha^2)(\Bar{\alpha}^2z^2-1)}dz\\&
=\frac{2(|\alpha|^2-1)(\Bar{\alpha}z^2+\alpha)}{(z^2-\alpha^2)(\Bar{\alpha}^2z^2-1)}dz.
\end{aligned}\end{gather}
Given $\alpha_2\in B(0,1)\setminus\{0\}$, part~(ii) now follows from the vanishing of $\lim b_{1,n}$ at $\lambda=1$:
$$1=\frac{\alpha(-1-|\alpha|^2)}{\Bar{\alpha}(-|\alpha|^2-1)}\quad\Longleftrightarrow\quad\alpha\in(-1,0)\cup(0,1)\quad\Longleftrightarrow\quad\alpha_2\in(0,1).$$
We conclude the proof by showing that $\alpha_2\mapsto\varphi_{CD}(\alpha_2)=\lim\phi_2(a_n,b_{1,n},b_{2,n})$ is an increasing continuous bijection $(0,1)\to(0,\pi)$. At $\lambda=1$, $z^2=\frac{1-\alpha_2}{\Bar{\alpha_2-1}}=-1$. We hitherto fix the sign of $z$ so that $z=\mi$ at the point $y(\Tilde{a})=\lim y(a_n)$ of~\eqref{limit curve 3} over $\lambda=1$ with $\Tilde{\nu}>0$. Hence the two paths along the real part of~\eqref{limit curve 3} from $\sigma(y(\Tilde{a}))$ to $y(\Tilde{a})$ correspond to the two paths along $z\in\bbS^1$ from $z=-\mi$ to $z=\mi$. Recall that the sign of $b_{1,n}$ is chosen so that $\phi_1(a_n,b_{1,n},b_{2,n})>0$. We now show that with our choice of $z$, \eqref{limit form 3b} implies $\alpha>0$. By our choice of $z$ we obtain
$$\lim\phi_1(a_n,b_{1,b},b_{2,n})=\tfrac{1}{2}\IM\int_{-\mi}^\mi\lim\Theta(b_{1,n})=\tfrac{1}{2}\IM\int_{-\mi}^\mi d\ln\frac{(z-\alpha)(z+\alpha^{-1})}{(z-\alpha^{-1})(z+\alpha)},$$
where the integral is performed along either path in $\bbS^1$ from $-\mi$ to $\mi$. As illustrated in Figure~\ref{fig:angles} this integral defines an increasing, continuous and bijective function $(0,1)\to(0,\pi)$. Indeed, integration along the anti--clockwise path and defining $\vartheta=\arctan\frac{1}{\alpha}-\arctan\alpha$, then
$$\tfrac{1}{2}\IM\int_{-\mi}^\mi d\ln\frac{z-\alpha}{z-\alpha^{-1}}+\tfrac{1}{2}\IM\int_{-\mi}^\mi d\ln\frac{z+\alpha^{-1}}{z+\alpha}=\tfrac{1}{2}\big((2\pi-\vartheta)-\vartheta\big)+\tfrac{1}{2}\big(-\vartheta-\vartheta\big)=\pi-2\vartheta.$$
%
%% FIGURE 2 -- THE ONE WITH ANGLES
%
\begin{figure}
\begin{tikzpicture}[scale = 2.5]
    % circle
    \draw[densely dashed] (0,-1) arc (-90:90:1);
    \draw (0,-1) arc (270:90:1);

    % labels for i, -i
    \draw[fill = black] (0,1) circle (0.3pt) node[above] {$\mi$} node[above, inner sep = 26pt] {$2\pi-\theta$};
    \draw[fill = black] (0,-1) circle (0.3pt) node[below] {$-\mi$};

    % lines connecting i and -i to alpha and -alpha
    \draw (0,1) -- (0.6,0) node[at end, left] {$\alpha$} -- (0,-1);
    \draw[fill = black] (0.6,0) circle (0.3pt);
    \draw (0,1) -- (-0.6,0) node[at end, right] {$-\alpha$} -- (0,-1);
    \draw[fill = black] (-0.6,0) circle (0.3pt);

    % lines connecting i and -i to alpha^-1 and -alpha^-1
    \draw (0,1) -- (1.6,0) node[at end, right] {$\alpha^{-1}$} -- (0,-1) node[very near end, above] {$\theta$};
    \draw[fill = black] (1.6,0) circle (0.3pt);
    \draw (0,1) -- (-1.6,0) node[at end, left] {$-\alpha^{-1}$} -- (0,-1) node[very near end, above] {$\theta$};
    \draw[fill = black] (-1.6,0) circle (0.3pt);

    % arcs for theta near -i
    \draw (0.35,-0.7813) arc (40:66:0.4127);
    \draw (-0.35,-0.7813) arc (140:114:0.4127);

    % lines connecting alpha and alpha^-1
    \draw (0.6,0) -- (12:1) -- (1.6,0);
    \draw (0.6,0) -- (-12:1) -- (1.6,0);

    % Arc around i
    \draw (0.165,0.90) arc (-35:298:0.2   );% node[above, midway] {$2\pi-\theta$};

    % Arc for the theta next to this
    \draw (-0.35, 0.7813) arc (220:246:0.4127) node[midway, above right] {$\theta$};

    % arcs near \alpha^-1
    %\draw (1.16783,-0.14449) arc (45:133:0.23);
    \draw[domain=17:153, variable=\x] plot ({0.9781+0.13*cos(\x)}, {-0.2079+0.13*sin(\x)});
    \draw[domain=-20:210, variable=\x] plot({0.9781+0.09*(cos(\x)}, {0.2079+0.09*sin(\x)});
    %\draw (1.16783,0.14449) arc [start angle = 45, end angle = 139, x radius = 0.24, y radius = 0.45];

    % arrow
    \draw[thick, ->] (-3:1) arc (-3:8:1);
\end{tikzpicture}
\caption{Angles in the proof of Lemma~\ref{Le:case c and d}}\label{fig:angles}
\end{figure}
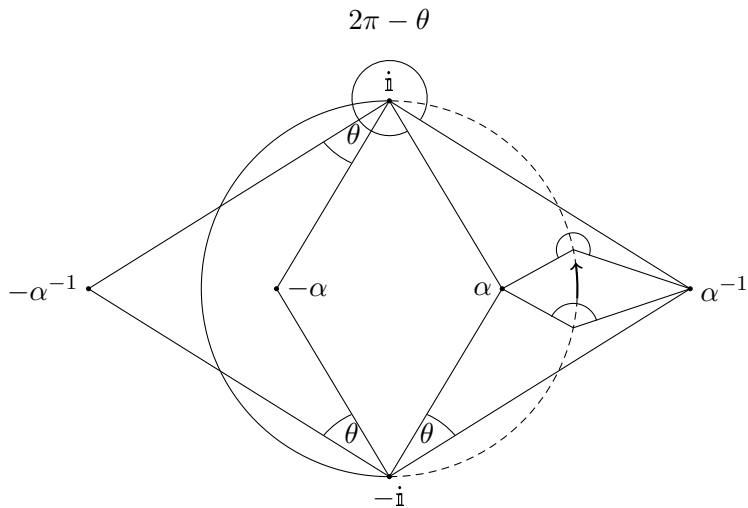
Note that the $2\pi$ is due to the fact that the path of integration passes between $\alpha$ and $\alpha^{-1}$, as is apparent from Figure~\ref{fig:angles}. Clearly $\pi-2\vartheta=\pi-2\big(\arctan\frac{1}{\alpha}-\arctan\alpha\big)$ is strictly increasing with respect to $\alpha=\sqrt{\alpha_2}\in(0,1)$ with the values $0$ and $\pi$ at the left and right end points, respectively.
\end{proof}
\subsection{The case (E)}
\begin{lemma}\label{Le:case e}
In Theorem~\ref{Th:global basis}~(v) case~(E) the sequence $(\phi(a_{t_n},b_{1,t_n},b_{2,t_n}))_{n\in\mathbb{N}}$ converges to $(0,0)$.
\end{lemma}
\begin{proof}
The proof is similar to some parts of the proof of Lemma~\ref{Le:case c and d}. Let $(a_n,b_{1,n},b_{2,n})_{n\in\mathbb{N}}$ be a sequence such that the corresponding $(\alpha_{1,n},\alpha_{2,n})_{n\in\mathbb{N}}$ converges to $(0,0)$. By possibly interchanging the indices $\ind=1,2$ and passing to a subsequence, we may assume that $|\alpha_{1,n}|\le|\alpha_{2,n}|$ holds for all $n\in\mathbb{N}$. We blowup the sequence of spectral curves and use the spectral parameter $\lambda^-=|\alpha_{1,n}|\lambda$ instead of $\lambda=\frac{\lambda^-}{|\alpha_{1,n}|}$. Again the blownup roots are all bounded. After passing to a subsequence they converge to $\alpha^-_\ind:=\lim|\alpha_{1,n}|\Bar{\alpha}_{\ind,n}^{-1}$ for $\ind=1,2$ and the limit curve is again~\eqref{blow up curve 3} with
\begin{align*}
  a^-(\lambda^-)&=\alpha_0(\lambda^-)^2(\lambda^--\alpha^-_1)(\lambda^--\alpha^-_2)\qquad\text{and with}&\alpha_0&=\lim\frac{\Bar{\alpha}_{1,n}\,\Bar{\alpha}_{2,n}}{|\alpha_{1,n}\,\alpha_{2,n}|}.
\end{align*}                                                                    The root $\alpha^-_1$ is unimodular and $\alpha^-_2\in\overline{B(0,1)}$. The argument is different for $\alpha^-_2=0$ and $\alpha^-_2\ne 0$.

We first assume $\alpha^-=0$, which held in the proof of Lemma~\ref{Le:case c and d}. Thus the same arguments again show that $(b^-_{\ind,n})_{n\in\mathbb{N}}$ converges for $\ind=1,2$. To finish the proof in this case we apply a blowup on an intermediate scale $\Tilde{\lambda}=|\alpha_{2,n}|\lambda$.  The blownup roots $|\alpha_{2,n}|\alpha_{1,n}$, $|\alpha_{2,n}|\alpha_{2,n}$, and  $|\alpha_{2,n}|\Bar{\alpha}_{2,n}^{-1}$ are all bounded. After passing to a subsequence, $(|\alpha_{2,n}|\Bar{\alpha}_{2,n}^{-1})_{n\in\mathbb{N}}$ converges to $\Tilde{\alpha}_2\in\bbS^1$ and the two other sequences converge to zero. Again~\eqref{representation an} yields the following limit:
\begin{gather}\label{blow up a 4}
\lim\Tilde{a}_n(\Tilde{\lambda}):=\lim|\alpha_{1,n}||\alpha_{2,n}|^3a_n\big(\tfrac{\Tilde{\lambda}}{|\alpha_{2,n}|}\big)=:\Tilde{a}(\Tilde{\lambda})=-\Tilde{\alpha}_2^{-1}\Tilde{\lambda}^2(\Tilde{\lambda}-\Tilde{\alpha}_2).
\end{gather}
This limit defines the blowup of the spectral curve with a cusp at $\Tilde{\lambda}=0$:
\begin{gather}\label{blow up curve 4}
\{(\Tilde{\lambda},\Tilde{\nu})\in\mathbb{C}^2\mid\Tilde{\nu}^2=\Tilde{\lambda}\,\Tilde{a}(\Tilde{\lambda})\}.
\end{gather}
The orientations of $A_1$ and $A_2$ determine a cycle $\Tilde{B}_1=B_1\pm A_2$ which converges to a cycle on~\eqref{blow up curve 4} encircling the double point at $\Tilde{\lambda}=\infty$. With respect to \,$\Tilde{\lambda}=|\alpha_{2,n}|\lambda$\, the 1-form $\Theta(b_{\ind,n})$\, equals
\begin{align}\label{blow up form 4}
\Tilde{\Theta}(\Tilde{b}_{\ind,n})&=\frac{\Tilde{b}_{\ind,n}(\Tilde{\lambda})}{\Tilde{\nu}_n}\frac{d\Tilde{\lambda}}{\Tilde{\lambda}}\quad\text{with}&\Tilde{b}_{\ind,n}(\Tilde{\lambda})&:=|\alpha_{1,n}|^{\frac{1}{2}}|\alpha_{2,n}|^2b_{\ind,n}\big(\tfrac{\Tilde{\lambda}}{|\alpha_{2,n}|}\big)\quad\text{and}&\Tilde{\nu}_n^2&=\Tilde{\lambda}\,\Tilde{a}_n(\Tilde{\lambda}).
\end{align}
The factor $|\alpha_{1,n}|^{\frac{1}{2}}|\alpha_{2,n}|^2$ in the definition of $\Tilde{b}_{\ind,n}$ compensates for the factor in the relation $\Tilde{\nu}_n=|\alpha_{1,n}|^{\frac{1}{2}}|\alpha_{2,n}|^2\nu$, which follows from~\eqref{blow up a 4} and~\eqref{blow up form 4}. We show that $(\Tilde{b}_{\ind,n})_{n\in\mathbb{N}}$ is bounded by writing $\Tilde{b}_{\ind,n}(\Tilde{\lambda})=\big|\tfrac{\alpha_{2,n}}{\alpha_{1,n}}\big|^2b^-_{\ind,n}\big(\big|\tfrac{\alpha_{1,n}}{\alpha_{2,n}}\big|\Tilde{\lambda}\big)$ in terms of the convergent sequence $(b^-_{\ind,n})_{n\in\mathbb{N}}$~\eqref{blow up form 3}. Indeed, by $\lim\big|\tfrac{\alpha_{1,n}}{\alpha_{2,n}}\big|=0$ we first obtain that for $d\in\{2,3\}$ the $\Tilde{\lambda}^d$--coefficients of $(\Tilde{b}_{\ind,n})_{n\in\mathbb{N}}$ are bounded. Because $b_{\ind,n}$ has at least two roots in $\overline{B(0,1)}$, any accumulation point of $\big(\Tilde{b}_{\ind,n}/\|\Tilde{b}_{\ind,n}\|\big)_{n\in\mathbb{N}}$ has a root of order at least two at $\Tilde{\lambda}=0$, and therefore has degree $d\in\{2,3\}$. Hence $(\Tilde{b}_{\ind,n})_{n\in\mathbb{N}}$ is bounded. In~\eqref{blow up curve 4} the curves from $\sigma(y(a_n))$ to $y(a_n)$ collapse for $n\to\infty$. Now $\lim\phi(a_n,b_{1,n},b_{2,n})=0$ follows since any accumulation point of $(\Tilde{\Theta}(\Tilde{b}_{\ind,n})_{n\in\mathbb{N}}$ is holomorphic at $\Tilde{\lambda}=0$ in the normalisation of~\eqref{blow up curve 4}.

If $\alpha^-_2\ne0$, then $B_2$ converges to a cycle encircling $\alpha^-_2$ and $\lambda^-=\infty$. We define $b^-_{1,n}$ as in~\eqref{blow up form 3} such that the corresponding 1-forms coincide with $\Theta(b_{\ind,n})$. As in Lemma~\ref{Le:case c and d} we prove now that $(b^-_{\ind,n})_{n\in\mathbb{N}}$ converges without renormalisation. Any accumulation point $\Tilde{b}^-_\ind$ of $(b^-_{\ind,n}/\|b^-_{\ind,n}\|)_{n\in\mathbb{N}}$ defines a 1-form~\eqref{blow up form 3a} on the curve~\eqref{blow up curve 3}. Again it has a root at $\lambda^-=0$ of order at least two, since $b_{\ind,n}$ has at least two roots in $\lambda\in\overline{B(0,1)}$. Hence the 1-forms~\eqref{blow up form 3a} are anti--symmetric with respect to $\sigma:(\lambda^-,\nu^-)\mapsto(\lambda^-,-\nu^-)$ and have only one pole at $\lambda^-=\infty$ which is a second order pole without residue. If the integral of one of the 1-forms~\eqref{blow up form 3a} along the limits of $\Tilde{B}_1$ and $B_2$ vanish, then this 1--form is the derivative of a meromorphic function with one simple pole at $\lambda^-=\infty$. Such a function has degree 1 and cannot exist on the limit curve~\eqref{blow up curve 3} which has geomtric genus one for $\alpha^-_2\ne 0$. The same argument as in the proof of Lemma~\ref{Le:case c and d} shows that the sequence $(b^-_{\ind,n})_{n\in\mathbb{N}}$ indeed converge to some $b^-_\ind$. In~\eqref{blow up curve 3} the path from $\sigma(y(a_n))$ to $y(a_n)$ collapses. Since the limit~\eqref{blow up form 3a} is holomorphic at $\lambda^-=0$ in the normalisation of~\eqref{blow up curve 3}, $\lim\phi(a_n,b_{1,n},b_{2,n})=0$ follows.
\end{proof}
%
%%%%%%%%%%%%%%%%
%
\section{The Proof of Theorem~\ref{Th:main}}\label{Se:proof}
\noindent{\it Proof of Theorem~\ref{Th:main}.} Fix any $a_0\in\calS_1^2$ together with a labelling $(\alpha_{1,0},\alpha_{2,0})$ of the roots of $a_0$ in $B(0,1)$. Let $(A_1,A_2)$ be the corresponding cycles with the orientations in Theorem~\ref{Th:global basis}~(i) and $(b_{1,0},b_{2,0})$ the corresponding basis of $\calB_{a_0}$ which obeys~\eqref{period map}. First we show by contradiction that the maximal domain $\Omega$ of the commuting Whitham flows in Theorem~\ref{Th:global basis}~(iii) is bounded. So let $(t_n)_{n\in\mathbb{N}}$ be a sequence in $\Omega$ without accumulation point, and let $(\alpha_{1,n},\alpha_{2,n})_{n\in\mathbb{N}}$ be the corresponding labelled roots of $(a_{t_n})_{n\in\mathbb{N}}$ contained in $B(0,1)$. After passing to a subsequence $(\alpha_{1,n},\alpha_{2,n})_{n\in\mathbb{N}}$ has a limit $(\alpha_1,\alpha_2)\in\overline{B(0,1)}{}^2$. Since $(t_n)_{n\in\mathbb{N}}$ has no accumulation point, Theorem~\ref{Th:global basis}~(v) applies and $(\alpha_1,\alpha_2)$ belongs to one of the sets in the cases~(A)-(E). The Lemmata~\ref{Le:case a}--\ref{Le:case e} show in all these cases that $(\phi(a_{t_n},b_{1,t_n},b_{2,t_n}))_{n\in\mathbb{N}}$ has a convergent subsequence. Theorem~\ref{Th:global basis}~(iv) gives $t_n=\phi(a_{t_n},b_{1,t_n},b_{2,t_n})-\phi(a_0,b_{1,0},b_{2,0})$, which contradicts the assumption on $(t_n)_{n\in\mathbb{N}}$. This shows that the closure of $\Omega$ is sequentially compact and therefore bounded.

Theorem~\ref{Th:global basis}~(v) and Lemmata~\ref{Le:case a}--\ref{Le:case e} also show that $t\mapsto\phi(a_t,b_{1,t},b_{2,t})$ maps $\partial\Omega^\circ:=\{t\in\partial\Omega\mid st\in\Omega\text{ for all }s\in[0,1)\}$ into $\partial\triangle$. Theorem~\ref{Th:global basis}~(iv) then gives $\partial\Omega^\circ\subset\partial\triangle-\phi(a_0,b_{1,0},b_{2,0})$. By definition of $\Omega$, its intersection with any straight line $L$ through $t=0$ is an open and bounded subinterval of $L$ containing $t=0$ whose end points belong to $\partial\Omega^\circ\subset\partial\triangle-\phi(a_0,b_{1,0},b_{2,0})$.

Next we prove $\phi(a_0,b_{1,0},b_{2,0})\in\triangle$. In fact, for $\phi(a_0,b_{1,0},b_{2,0})\not\in\triangle$ any line $L$ through $t=0$ contains at least one half line starting at $t=0$ which is disjoint from $\triangle-\phi(a_0,b_{1,0},b_{2,0})$. This has just been excluded and shows $\phi(a_0,b_{1,0},b_{2,0})\in\triangle$. Now, due to the convexity of $\triangle-\phi(a_0,b_{1,0},b_{2,0})$, any such $L$ intersects $\partial\triangle-\phi(a_0,b_{1,0},b_{2,0})$ in exactly two points. This implies $L\cap\Omega=L\cap\big(\triangle-\phi(a_0,b_{1,0},b_{2,0})\big)$ for any such $L$. Then $\Omega=\triangle-\phi(a_0,b_{1,0},b_{2,0})$ follows and $t\mapsto\phi(a_t,b_{1,t},b_{2,t})$~\eqref{def:image phi} maps $\Omega$ onto $\triangle$.

We formulate the next step as a lemma, whose first statement we just have proven:
\begin{lemma}\label{wente existence}
In the situation of Theorem~\ref{Th:global basis} the domain $\Omega$ in part~(iii) contains the diagonal
\begin{align}\label{def:diagonal 2}
\big\{(\varphi,\varphi)-\phi(a_0,b_{1,0},b_{2,0})\mid\varphi\in(0,\tfrac{\pi}{2})\big\}\subset\Omega=\triangle-\phi(a_0,b_{1,0},b_{2,0}).
\end{align}
Furthermore, on this set the map $t\mapsto a_t$ (cf.\ \eqref{eq:whitham flow}) is a diffeomorphism onto $\calW$.
\end{lemma}
\begin{proof}
By the proof of Lemma~\ref{L:wente:flowout} $\calW$ is an unstable manifold of the Whitham vector field defined by~\eqref{eq:wente:whitham-wente:dgl-apm} and~\eqref{eq:wente:whitham-wente:dgl-beta23} at $(x,y_1,y_2,y_3)=(0,1,-\tfrac{1}{16}\alpha,-\tfrac{1}{2}\alpha)$ in the exceptional fibre of the blowup~\eqref{eq:wente:flowout:blowup}. Since on $\calW$ the periods of $\Theta(b_2)$ are purely imaginary, the same holds for the limit of $\Theta(b_2)$ on the corresponding limit of spectral curves. Lemma~\ref{L:wente:delta-unique} shows that this property uniquely determines the parameter $\alpha$ as the unique root $\alpha_0\in(0,2)$ of the map $\alpha\mapsto\beta$ in Lemma~\ref{lemma 1}. In the proof we first apply the transformation of Lemma~\ref{cayley transform} and then perform the same blowup as in Section~\ref{Se:blow up sym point}. This means that the point $(x,y_1,y_2,y_3)=(0,1,-\tfrac{1}{16}\alpha_0,-\tfrac{1}{2}\alpha_0)$ in the exceptional fibre of the blowup~\eqref{eq:wente:flowout:blowup} coincides with the centre of $\Hat{\calE}^2_0$ as described in Theorem~\ref{th:blow-up}. Finally, for any initial $a_0\in\calS^2_1$ Lemma~\ref{Le:case a} shows that this centre of $\Hat{\calE}^2_0$ is located at $t_\infty=(\frac{\pi}{2},\frac{\pi}{2})-\phi(a_0,b_{1,0},b_{2,0})\in\partial\Omega$. Therefore the lemma follows, if for any initial $a_0\in\calW$ the Whitham vector field along $\calW$ is at $a_0$ colinear to the the vector field of the flow in Theorem~\ref{Th:global basis} along the diagonal~\eqref{def:diagonal 2} with $c_1(1)=c_2(1)$.

For the proof of this we assume $a_0=a\in\calW$ for some $t\in\Omega$ and determine the basis $(b_1,b_2)$ of $\calB_{a}$ in~\eqref{eq:wente:whitham-wente:b} in terms of the basis $(b_{1,0},b_{2,0})$ in Theorem~\ref{Th:global basis}~(iii). In order to determine the action of the additional involution $\tau:(\lambda,\nu)\mapsto(\Bar{\lambda},\Bar{\nu})$ on the oriented cycles $(A_1,A_2)$ of $\Sigma_{a_0}$ we make use of the segments $S^\circ_1$, $S^\circ_2$ and $S^\circ_3$ introduced in the proof of Lemma~\ref{Le:case a}. In the case $\RE\alpha^\circ_a<\RE\alpha^\circ_2$ the central element is constructed as the limit $a^\circ\uparrow 0$ with a collapsed segment $S^\circ_1$. Since $\tau$ preserves $y(a_0)$ and $\sigma(y(a_0)$, it maps the path $HC^\circ_1$ from $y(a_0)$ to $\sigma(y(a_0))$ onto the negative $-HC^\circ_2$ of the path $HC^\circ_2$ from $\sigma(y(a_0))$ to $y(a_0)$. If we endow the cycles $(A_1,A_2)$ with the orientations induced on $HC^\circ_1$ and $HC^\circ_2$ as part of the boundary of $D^\circ$, then the cycle $A_1+A_2$ is homologous to a cycle which does not intersect the cycle $B$ introduced in the part of the proof of Lemma~\ref{Le:case a} which shows $b^\circ_2=\pm b^\circ_1$. However, since afterwards $b^\circ_2=-b^\circ_1$ was proven, the cycles $(A_1,A_2)$ have the same intersection number with $B$. Hence $\tau$ interchanges the oriented cycles $(A_1,A_2)$. Now~\eqref{period map} gives $\tau^\ast b_{1,0}=-\overline{b_{2,0}}$ and $\tau^\ast b_{2,0}=-\overline{b_{1,0}}$, since $\tau$ reverses the orientation and the intersection numbers.

Therefore $(b_1,b_2)=(b_{1,0}-b_{2,0},b_{1,0}+b_{2,0})$ obeys~\eqref{eq:para:whitham-real:b-real}. Furthermore, in Theorem~\ref{Th:global basis}~(iii) the flow along $c_1(1)=c_2(1)$ corresponds to the flow in~Lemma~\ref{L:wente:whitham-wente} with $c_1(1)=0$ which preserves $\calW$. This confirms that $\calW$ is the image of $t\mapsto a_t$ along the diagonal~\eqref{def:diagonal 2} with the limit point $t_\infty\in\partial\Omega$.
\end{proof}
It remains to prove that $t\mapsto a_t$ is a bijection from $\Omega$ onto $\calS^2_1$. Due to Lemma~\ref{wente existence}, for any given $\Tilde{a}_0\in\calW$ there exists $\Tilde{t}\in\Omega$ with $a_{\Tilde{t}}=\Tilde{a}_0$. Now the maximal flow in Theorem\ref{Th:global basis}~(iii) with initial $\Tilde{a}_0$ and the corresponding labelled roots $(\alpha_{1,\Tilde{t}},\alpha_{2,\Tilde{t}})$ coincides with the composition of $t\mapsto t+\Tilde{t}$ with the original flow, since by Theorem~\ref{Th:global basis}~(iv) the map~\eqref{def:phi} is a local diffeomorphism. Since this is true for any $a_0\in\calS^2_1$ together with a suitable labelling $(\alpha_{1,0},\alpha_{2,0})$ this implies first that the flow with initial $\Tilde{a}_0\in\calW$ and $(\alpha_{1,\Tilde{t}},\alpha_{2,\Tilde{t}})$ maps onto $\calS^2_1$ and then the surjectivity of the original flow.

Now we prove the injectivity of $\Omega\to\calS^2_1,\,t\mapsto a_t$. On~\eqref{def:diagonal 2} this follows from Theorem~\ref{T:wente:wente}. If $a_t=a_{\Tilde{t}}$, then the corresponding labelled roots $(\alpha_{1,t},\alpha_{2,t})$ and $(\alpha_{1,\Tilde{t}},\alpha_{2,\Tilde{t}})$ are either equal or interchanged. Since the corresponding values $\phi(a_t,b_{1,t},b_{2,t})\in\triangle$ and $\phi(a_{\Tilde{t}},b_{1,\Tilde{t}},b_{2,\Tilde{t}})\in\triangle$ have only positive components, they are either equal or have interchanged entries. In the second case along a path in $\Omega$ to~\eqref{def:diagonal 2} the corresponding pairs $t$ and $\Tilde{t}$ with interchanged components of $\phi$ both correspond to the same element of $\calS^2_1$. Since the map is a local diffeomorphism this is not true nearby~\eqref{def:diagonal 2}. So the smooth map $\Omega\mapsto\calS_1^2$ is injective. All together this shows that $\Omega\to\calS^2_1,\,t\mapsto a_t$ is a diffeomorphism. Finally, the inverse of this map composed with~\eqref{eq:whitham flow} defines a global section $\calF$ with the properties specified in Theorem~\ref{Th:main}. This finishes the proof.\qed
\begin{remark}
The spectral curves of Wente tori are those $a\in\calW$, such that $\calB_a$ has a basis $(b_1,b_2)$ for which each $\Theta(b_\ind)$, $\ind=1,2$, is the logarithmic derivative of a function $\mu_\ind$ satisfying
\begin{enumerate}
\item[(i)] In a neighbourhood of $\lambda=0$ there exists a branch of $\ln\mu_\ind$ which is anti--symmetric with respect to the hyperelliptic involution $\sigma$~\eqref{eq:sigma}.
\item[(ii)] $\mu_1(y(a))=1=\mu_2(y(a))$ or $\mu_1(y(a))=-1=\mu_2(y(a))$.
\end{enumerate}
Lemma~\ref{wente existence} shows that this set is mapped by $\phi$~\eqref{def:phi} diffeomorphically onto $\{\varphi,\varphi)\mid\tfrac{2\varphi}{\pi}\in(0,1)\cap\bbQ\}$. In particular, this lemma again proves the existence of Wente tori.
\end{remark}
%
%%%%%%%%%%%%%%
\section{The boundary of $\Hat{\calS}^2_0$ in $\bbR^4[\varkappa]$}\label{sec:boundary}
We prove that $\Hat{\calS}^2_0$ is bounded. Hence each case in Theorem~\ref{Th:global basis}~(v) corresponds to a subset of $\partial\Hat{\calS}^2_0$.
\begin{lemma}
\label{L:boundary:S20-bounded}
$\Hat{\calS}^2_0$ is a bounded subset of $\bbR^4[\varkappa]$.
\end{lemma}
\begin{proof}
It suffices to show that the roots of any sequence in $\Hat{\calS}^2_0$ are bounded. Consider the sequence $(\alpha_{1,n},\alpha_{2,n})_{n\in\mathbb{N}}$ of roots in \,$B(0,1)$\, of the corresponding sequence in $\calS^2_1$. Our task is to prove that neither component of an accumulation point of \,$(\alpha_{1,n},\alpha_{2,n})$\, is equal to the value \,$-1$\, of $\lambda$ at $\varkappa=\infty$~\eqref{eq:moebius-kappa}. Since $B(0,1)$ is bounded, we may assume that $(\alpha_{1,n},\alpha_{2,n})_{n\in\mathbb{N}}$ converges. If one component of \,$\lim \, (\alpha_{1,n},\alpha_{2,n})$\, does not belong to \,$B(0,1)$\, then by Theorem~\ref{Th:global basis} this limit belongs to one of the sets in the cases~(A)--(E). The case~(C) is the only case where one component might be equal to $-1$. However, this possibility is excluded by Lemma~\ref{Le:case c and d}. This completes the proof.
\end{proof}
In Theorem~\ref{Th:global basis}~(v), $\partial\Hat{\calS}^2_0$ is decomposed into five subsets which correspond to the cases~(A)--(E). The results of Section~\ref
{Se:limits a to e} yield that the composition $\Hat{\phi}$ of the inverse diffeomorphism of Lemma~\ref{cayley transform} with $\phi$~\eqref{def:phi} maps the cases~(A) and (D) to the edges of $\triangle$, whereas it maps the cases~(B), (C) and (E) to the vertices. %Let us now transform the  Whitham vector fields in Lemma~\ref{L:para:whitham} to vector fields on $\Hat{\calF}$. Afterwards we investigate to which parts of the boundary of $\Hat{\calS}^2_0$ these vector fields extend smoothly. The analogues to the Whitham equations~\eqref{eq:para:whitham:1}--\eqref{eq:para:whitham:2} and equation~\eqref{eq:para:whitham:ck'} take the form

\bibliographystyle{amsplain}
\bibliography{ref}

\def\cydot{\leavevmode\raise.4ex\hbox{.}} \def\cprime{$'$}
\providecommand{\bysame}{\leavevmode\hbox to3em{\hrulefill}\thinspace}
\providecommand{\MR}{\relax\ifhmode\unskip\space\fi MR }
\providecommand{\MRhref}[2]{%
  \href{http://www.ams.org/mathscinet-getitem?mr=#1}{#2}
}
\providecommand{\href}[2]{#2}
\begin{thebibliography}{10}

\bibitem{Ab}
U.~Abresch, \emph{Constant mean curvature tori in terms of elliptic functions},
  J. Reine U. Angew Math. \textbf{374} (1987), 169--192.

\bibitem{ABR}
S.~Axler, P.~Bourdon, and W.~Ramey, \emph{Harmonic function theory.}, 2nd ed.,
  Grad. Texts Math., vol. 137, New York, NY: Springer, 2001.

\bibitem{Bob:cmc}
A.~I. Bobenko, \emph{Constant mean curvature surfaces and integrable
  equations}, Russian Math. Surveys \textbf{46} (1991), 1--45.

\bibitem{CKS}
E.~Carberry, S.~Klein, and M.~U. Schmidt, \emph{Blowing up sequences of
  constant mean curvature tori with fixed spectral genus in the euclidean
  3-space to minimal surfaces}, arXiv:2110.01574.

\bibitem{CO:19}
E.~Carberry and R.~Ogilvie, \emph{Whitham deformations and the space of
  harmonic tori in {$S^3$}}, J. Lond. Math. Soc. (2) \textbf{99} (2019), no.~3,
  945--964.

\bibitem{CO:20}
\bysame, \emph{The space of equivariant harmonic tori in the 3-sphere}, J.
  Geom. Phys. \textbf{157} (2020), 22.

\bibitem{CS1}
E.~Carberry and M.~U. Schmidt, \emph{The closure of spectral data for constant
  mean curvature tori in {$\mathbb{S}^3$}}, J. Reine Angew. Math. \textbf{721}
  (2016), 149--166.

\bibitem{CS}
\bysame, \emph{The prevalence of tori amongst constant mean curvature planes in
  {$\mathbb{R}^3$}}, J. Geom. Phys. \textbf{106} (2016), 352--366.

\bibitem{FaKr}
H.~M. Farkas and I.~Kra, \emph{Riemann surfaces}, second ed., Graduate Texts in
  Mathematics, vol.~71, Springer-Verlag, New York, 1992.

\bibitem{GerPS}
A.~Gerding, F.~Pedit, and N.~Schmitt, \emph{Constant mean curvature surfaces:
  an integrable systems perspective}, Harmonic maps and differential geometry.
  A harmonic map fest in honour of John C. Wood's 60th birthday, Cagliari,
  Italy, September 7--10, 2009, Providence, RI: American Mathematical Society
  (AMS), 2011, pp.~7--39.

\bibitem{GriS1}
P.~G. Grinevich and M.~U. Schmidt, \emph{Period preserving nonisospectral flows
  and the moduli space of periodic solutions of soliton equations}, Phys. D
  \textbf{87} (1995), no.~1-4, 73--98.

\bibitem{HKS3}
L.~Hauswirth, M.~Kilian, and M.~U. Schmidt, \emph{Mean-convex {A}lexandrov
  embedded constant mean curvature tori in the 3-sphere}, Proc. Lond. Math.
  Soc. (3) \textbf{112} (2016), no.~3, 588--622.

\bibitem{HKS2}
\bysame, \emph{Properly embedded minimal annuli in {$\mathbb{S}^2 \times
  \mathbb{R}$}}, J. Integrable Syst. \textbf{5} (2020), no.~1, xyaa005, 37.

\bibitem{H2021}
L.~Heller, \emph{Generalized {Whitham} flow and its applications}, Minimal
  surfaces: integrable systems and visualisation. m:iv workshops, 2016--19,
  Cham: Springer, 2021, pp.~131--146.

\bibitem{hit:tor}
N.~Hitchin, \emph{Harmonic maps from a 2-torus to the 3-sphere}, J.
  Differential Geom. \textbf{31} (1990), no.~3, 627--710.

\bibitem{KSS}
M.~Kilian, M.~U. Schmidt, and N.~Schmitt, \emph{Flows of constant mean
  curvature tori in the 3-sphere: the equivariant case}, J. Reine Angew. Math.
  \textbf{707} (2015), 45--86.

\bibitem{K-L-S-S}
S~Klein, E.~L\"{u}bcke, M.U. Schmidt, and T.~Simon, \emph{Singular curves and
  {B}aker-{A}khiezer functions}, arXiv:1609.07011.

\bibitem{KPS}
M.~Knopf, R.~Pe\~na Hoepner, and M.~U. Schmidt, \emph{Solutions of the
  {S}inh--{G}ordon equation of spectral genus two and constrained {W}illmore
  tori {I}}, arXiv:1708.00887.

\bibitem{MaOs}
V.~A. Mar{\v{c}}enko and I.~V. Ostrovs{\cprime}ki{\u\i}, \emph{A
  characterization of the spectrum of the {H}ill operator}, Mat. Sb. (N.S.)
  \textbf{97(139)} (1975), no.~4(8), 540--606, 633--634.

\bibitem{PinS}
U.~Pinkall and I.~Sterling, \emph{On the classification of constant mean
  curvature tori}, Ann. Math. \textbf{130} (1989), 407--451.

\bibitem{Wen}
H.~C. Wente, \emph{Counterexample to a conjecture of {H}. {H}opf}, Pac. J.
  Math. \textbf{121} (1986), 193--243.

\end{thebibliography}
\end{document}